\documentclass[12pt,leqno]{amsart}

\usepackage{amssymb,verbatim,amsmath,xcolor,mathrsfs,amsthm}

\usepackage[T1]{fontenc}
\usepackage[utf8]{inputenc}
\usepackage{csquotes}
\usepackage{euscript}
\usepackage{mathtools}
\usepackage{empheq}
\usepackage[american]{babel}
\usepackage[font=scriptsize]{caption}
\usepackage{CJKutf8}

\usepackage{xpatch}

\usepackage[inline]{enumitem}

\usepackage{bm}

\usepackage{tikz-cd}

\usepackage{tikz}
\usetikzlibrary{automata,positioning}
\usetikzlibrary{decorations.pathreplacing}

\usepackage[letterpaper,left=2.5cm,right=2.5cm,top=2.5cm,bottom=2.5cm,dvips]{geometry}

\usepackage[bookmarks=true,linktocpage=true,pdfencoding=auto]{hyperref}

\makeatletter
\newcommand{\optionaldesc}[2]{%
  \phantomsection
  #1\protected@edef\@currentlabel{#1}\label{#2}%
}
\makeatother

\renewcommand{\eqref}[1]{(\ref{#1})}

\usepackage{stackengine,graphicx}

\newcommand\constantoverleftarrow[1]{\smash{\cola#1\endcola}}
\def\cola#1#2\endcola{\stackengine{0pt}{}{$\overleftsmallarrow{#1}$}{O}{l}{F}{T}{L}%
  \phantom{#1}#2}

\newcommand\constantoverrightarrow[1]{\smash{\colas#1\endcolas}}
\def\colas#1#2\endcolas{\stackengine{0pt}{}{$\overrightsmallarrow{#1}$}{O}{l}{F}{T}{L}%
  \phantom{#1}#2}
  
\makeatletter
\newcommand{\overleftsmallarrow}{\mathpalette{\overarrowsmall@\leftarrowfill@}}
\newcommand{\overrightsmallarrow}{\mathpalette{\overarrowsmall@\rightarrowfill@}}
\newcommand{\overarrowsmall@}[3]{%
  \vbox{%
    \ialign{%
      ##\crcr
      #1{\smaller@style{#2}}\crcr
      \noalign{\nointerlineskip\vskip0.4pt}%
      $\m@th\hfil#2#3\hfil$\crcr
    }%
  }%
}
\def\smaller@style#1{%
  \ifx#1\displaystyle\scriptstyle\else
    \ifx#1\textstyle\scriptstyle\else
      \scriptscriptstyle
    \fi
  \fi
}
\makeatother

\DeclareFontFamily{U}{mathc}{}
\DeclareFontShape{U}{mathc}{m}{it}%
{<->s*[1.03] mathc10}{}
\DeclareMathAlphabet{\mathcalsmall}{U}{mathc}{m}{it}

\let\originalleft\left
\let\originalright\right
\renewcommand{\left}{\mathopen{}\mathclose\bgroup\originalleft}
\renewcommand{\right}{\aftergroup\egroup\originalright}

\DeclareFontFamily{OT1}{cmrx}{}
\DeclareFontShape{OT1}{cmrx}{m}{n}{<->cmr10}{}
\let\saveLongrightarrow\Longrightarrow
\makeatletter
\renewcommand*{\Longrightarrow}{%
    \mathrel{\rlap{\fontfamily{cmrx}\fontencoding{OT1}\selectfont=}%
    \hphantom{\saveLongrightarrow}%
    \llap{$\m@th\Rightarrow$}}}
\makeatother

\makeatletter
\newenvironment{proof_header}[1][\proofname]{\par
  \pushQED{\qed}%
  \normalfont\topsep6\p@\@plus6\p@\relax
  \trivlist
  \item[]
  {\itshape #1\@addpunct{.}}\hskip\labelsep\ignorespaces
}{%
  \popQED\endtrivlist\@endpefalse
}
\makeatother

\makeatletter
\def\@seccntformat#1{%
  \protect\textup{%
    \protect\@secnumfont
    \expandafter\protect\csname format#1\endcsname 
    \csname the#1\endcsname
    \protect\@secnumpunct
  }%
}

\makeatother

\newtheorem{thm}{Theorem}[section]
\newtheorem{prop}[thm]{Proposition}
\newtheorem{lem}[thm]{Lemma}
\newtheorem{cor}[thm]{Corollary}

\theoremstyle{definition}
\newtheorem{conj}[thm]{Conjecture}
\newtheorem{rem}[thm]{Remark}

\newtheorem{ex}[thm]{Example}
\newtheorem{defn}[thm]{Definition}
\newtheorem*{defn*}{Definition}
\newtheorem{fact}[thm]{Fact}

\newtheorem*{construction*}{Construction}

\newtheorem*{perspective*}{Perspective}
\newtheorem*{motivation*}{Motivation}
\newtheorem*{summary*}{\normalfont\scshape Summary}
\newtheorem*{references*}{References}

\title{On the dimension of the Fomin-Kirillov algebra and related algebras}

\author{Christoph Bärligea}
\dedicatory{\enquote{Faire de l'Algèbre, c'est essentiellement \emph{calculer}, c'est-à-dire effectuer, sur des éléments d'un ensemble, des « opérations algébriques », dont l'exemple le plus connu est fourni par les « quatre règles » de l'arithmétique élémentaire.} --- \emph{Nicolas Bourbaki\textsuperscript{1}}}
\email{christoph.baerligea@ruhr-uni-bochum.de}
\date{September~11, 2024}

\begin{document}

\begin{abstract}

In 1999, Fomin-Kirillov~\cite{fomin-kirillov} introduced the quadratic algebras $\EuScript{E}_m$ in terms of generators and relations which are the universal quadratic cover of the algebra generated by divided difference operators $\partial_{ij}$ acting on the polynomial ring $\mathbf{k}[x_1,\ldots,x_m]$. These algebras are mostly important due to their relations to Schubert calculus and geometry and to the general framework of quantum groups and Nichols algebras. Fomin and Kirillov asked about the dimension of $\EuScript{E}_m$. In this paper, we prove that $\EuScript{E}_m$ is infinite dimensional for all $m\geq 6$ which was a well-known conjecture. 
The techniques we use rely on braided differential calculus as developed by Liu~\cite{liu,liusubalgebras} and Bazlov~\cite{bazlov1} as well as on the notion of integrals for Hopf algebras as introduced by Sweedler~\cite{sweedler}.

\end{abstract}

\maketitle
\footnotetext[1]{See~\cite[Introduction, page~xi]{bourbaki2007algebre}.}
\stepcounter{footnote}
\setcounter{tocdepth}{1}
\tableofcontents

\section{\for{toc}{{\color{white}0}}Introduction}

We fix once and for all an arbitrary ground field $\mathbf{k}$. We will always work over this field. We denote by $\mathbb{S}_m$ the symmetric group on $m$ letters. We denote by $R$ the root system of type $\mathsf{A}_n$ where $n=m-1$. Fomin-Kirillov introduced in \cite[Definition~2.1]{fomin-kirillov} a quadratic $\mathbf{k}$-algebra $\EuScript{E}_m$, which is now called the Fomin-Kirillov algebra, defined by generators and relations where we have generators $x_\alpha$ for each $\alpha\in R$ and homogeneous relations
\begin{align*}
\begin{aligned}
x_{-\alpha}&\mathrlap{{}=-x_\alpha}
\vphantom{\text{\parbox[c]{2.5in}{for all $\alpha,\gamma\in R$ such that $\alpha$ and $\gamma$ are orthogonal,}}}\\
x_\alpha^2&\mathrlap{{}=0}
\end{aligned}&\hphantom{-x_\gamma x_\alpha=0}&&
\mathllap{\left.\vphantom{\begin{aligned}
x_{-\alpha}&\mathrlap{{}=-x_\alpha}
\vphantom{\text{\parbox[c]{2.5in}{for all $\alpha,\gamma\in R$ such that $\alpha$ and $\gamma$ are orthogonal,}}}\\
x_\alpha^2&\mathrlap{{}=0}
\end{aligned}}\right\}\quad}
\text{for all $\alpha\in R$,}\\
x_\alpha x_\gamma&\mathrlap{{}=x_\gamma x_\alpha}\hphantom{-x_\gamma x_\alpha=0}&&
\text{\parbox[c]{2.5in}{for all $\alpha,\gamma\in R$ such that $\alpha$ and $\gamma$ are orthogonal,}}\\
x_\alpha x_{\alpha+\gamma}+x_{\alpha+\gamma}x_\gamma&-x_\gamma x_\alpha=0&&
\text{\parbox[t]{2.5in}{for all $\alpha,\gamma\in R$ such that $\alpha$ and $\gamma$ span a root subsystem of $R$ of type $\mathsf{A}_2$ with base $\{\alpha,\gamma\}$.}}
\end{align*}

There are plenty of motivations to consider the Fomin-Kirillov algebra. Let us mention at least two of them which are already present in \cite{fomin-kirillov}.
\begin{itemize}
    \item 
    The divided difference operators $\partial_\alpha$ acting on the polynomial ring $\mathbf{k}[x_1,\ldots,x_m]$ 
    satisfy exactly the above relations in degree less or equal than two where $\partial_\alpha$ plays the role of $x_\alpha$. In other words, the algebra $\EuScript{E}_m$ projects onto the subalgebra of $\operatorname{End}_{\mathbf{k}}\mathbf{k}[x_1,\ldots,x_m]$ 
    generated by $\partial_\alpha$ where $\alpha$ runs through $R$. The ideas of the Schubert calculus of divided difference operators will be prevalent in this paper, cf.~Section~\ref{sec:braided-leibniz}. For an elementary introduction to them, we refer to \cite{kirillov,knutson,macdonald1991notes}.
    \item
    Recall that over the complex numbers the coinvariant algebra $S_{\mathbb{S}_m}$ is canonically isomorphic to the cohomology ring of the complete flag variety $\mathrm{SL}_m(\mathbb{C})/B$ where $B$ is a Borel subgroup of $\mathrm{SL}_m(\mathbb{C})$ (\cite[Theorem~6.1]{fomin-kirillov}). Because of this geometric interpretation, the coinvariant algebra is sometimes also called Borel's algebra (at least when considered over the complex numbers). In \cite[Theorem~7.1]{fomin-kirillov}, it was proved that $S_{\mathbb{S}_m}$ can be canonically embedded into $\EuScript{E}_m$ via Dunkl elements. In this way, the Fomin-Kirillov algebra can be thought of as a noncommutative model for Schubert calculus, and can be useful to prove something about the latter.
\end{itemize}

In \cite[Problem~2.2]{fomin-kirillov}, the authors ask about the dimension of $\EuScript{E}_m$ considered as a $\mathbf{k}$-vector space. In this paper, we prove the following theorem.

\begin{thm}[Corollary~\ref{cor:main}]
\label{thm:fomin-kirillov-intro}

The Fomin-Kirillov algebra $\EuScript{E}_m$ is infinite dimensional for all $m\geq 6$.

\end{thm}

\begin{rem}
\label{rem:intro}

It is known from \cite[Example~6.4]{milinski-schneider} that $\EuScript{E}_m$ is finite dimensional for all $m\leq 5$.

\end{rem}

\begin{proof}[Setup of the paper]

Let $(W,S)$ be a finite Coxeter system. Let $\EuScript{B}_W$ be the Nichols-Woronowicz algebra model for Schubert calculus on $W$ as it has been studied in \cite{bazlov1}. The algebra $\EuScript{B}_W$ is a Nichols algebra associated to a Yetter-Drinfeld $W$-module defined by the set of reflections $T$ of $W$ and a specific one-dimensional representation of a subgroup of $W$. It is a braided $\mathbb{Z}_{\geq 0}$-graded Hopf algebra in the Yetter-Drinfeld category over $W$. Over the symmetric group $\mathbb{S}_m$, the algebra $\EuScript{B}_{\mathbb{S}_m}$ is known to be a quotient of $\EuScript{E}_m$, and conjecturally isomorphic to $\EuScript{E}_m$ which amounts to say that the Hopf duality pairing between $\EuScript{E}_m$ and itself is nondegenerate which is currently only known when $m\leq 5$.\footnote{The same situation arises for all simply laced Weyl groups, and a more general conjecture can be formulated in this setting as Pierre-Emmanuel Chaput brought to our attention.} In order to prove Theorem~\ref{thm:fomin-kirillov-intro}, it therefore suffices to prove the following conjecture.

\begin{thm}[Theorem~\ref{thm:main}]
\label{thm:inf_Bm-intro}

The algebra $\EuScript{B}_{\mathbb{S}_m}$ is infinite dimensional for all $m\geq 6$.

\end{thm}

\begin{rem}

Since $\EuScript{B}_{\mathbb{S}_m}$ is a quotient of $\EuScript{E}_m$ for all $m$ and since $\EuScript{E}_m$ is finite dimensional for all $m\leq 5$ by Remark~\ref{rem:intro}, we know that $\EuScript{B}_{\mathbb{S}_m}$ is finite dimensional for all $m\leq 5$ as well. 
    
\end{rem}

From now on, we will work with $\EuScript{B}_W$ whenever we can prove our results in this generality, and specialize to $\EuScript{B}_{\mathbb{S}_m}$ whenever needed. This has the advantage over working directly with $\EuScript{E}_m$ that the statements can be formulated in a clear and type independent way. Moreover, the algebras $\EuScript{B}_W$ and in particular $\EuScript{B}_{\mathbb{S}_m}$ enjoy desirable properties which are only conjecturally known for $\EuScript{E}_m$ and which we want to use in our proofs: Most significantly, just as every Nichols algebra associated to a finite dimensional Yetter-Drinfeld module, they admit a nondegenerate Hopf duality pairing.
\end{proof}

\begin{proof}[Strategy of the proof of Theorem~\ref{thm:inf_Bm-intro}]

\textit{For this environment, let us set temporarily $W=\mathbb{S}_m$.} The idea of the proof of Theorem~\ref{thm:inf_Bm-intro} is to introduce a certain subalgebra $\EuScript{B}_W'$ of $\EuScript{B}_W$ which is already apparent in \cite{liusubalgebras} as an explicitification of the class of subalgebras of $\EuScript{B}_W$ defined in \cite[Definition~2.2]{liusubalgebras}. The subalgebra $\EuScript{B}_W'$ also is a tensor factor of $\EuScript{B}_W$ and the corresponding other tensor factor is the nilCoxeter algebra $\EuScript{N}_W$ (-- first introduced in \cite[Section~2]{nilCoxeter} and in this paper studied in Section~\ref{sec:nilcoxeter}). In other words, we have according to Corollary~\ref{cor:tensorproduct} that
\[
\EuScript{B}_W=\EuScript{N}_W\otimes\EuScript{B}_W'=\EuScript{B}_W'\otimes\EuScript{N}_W\,.
\] 
This result is a for our means important explicitification of \cite[Theorem~4.1]{liusubalgebras} in the sense that we identify based on reduction of monomials in Section~\ref{sec:reduction} explicitly the other tensor factor corresponding to $\EuScript{B}_W'$ as $\EuScript{N}_W$.

Now, there are two key facts about this tensor product decomposition: 
\begin{itemize}
    \item 
    If $\EuScript{B}_W$ is finite dimensional, the \enquote{integral} in $\EuScript{B}_W'$ which exists by Theorem~\ref{thm:subalgebra}
    in the sense of \cite[Section~2]{sweedler} is preserved by the antipode of $\EuScript{B}_W$ regarded as a braided Hopf algebra (cf.~Theorem~\ref{thm:hypo-antipode}).
    \item
    The Hopf duality pairing between $\EuScript{B}_W$ and itself restricted to $\EuScript{B}_W'\otimes\EuScript{B}_W'$ stays nondegenerate (cf.~Corollary~\ref{cor:nondegenerate_sub}).  
\end{itemize}
These two facts together with invariance properties (cf.\ Section~\ref{sec:invariance} and Section~\ref{sec:invhypo}) can be used to derive a contradiction to finite dimensionality of $\mathbb{B}_{\mathbb{S}_m}$ if $m\geq 6$. Namely, by studying the known structure of the finite dimensional cases for $m\leq 5$ (cf.~\cite[Proposition~6.15, Corollary~6.20]{bastian}, \cite[Example~6.4]{milinski-schneider}, \cite[Proposition~5.13(2)]{integrale5}), you can observe a pattern in the integrals in $\EuScript{B}_{\mathbb{S}_m}$ 
which must be observed by all finite dimensional cases (cf.\ Conjecture~\ref{conj:milinski} and Theorem~\ref{thm:milinski-conj}). The dimensions and the integrals for $m\leq 5$ are rather easy and intuitively to understand and to describe explicitly (cf.~\cite{bastian,milinski-schneider,integrale5}~loc.~cit.). Once the integral in $\EuScript{B}_{\mathbb{S}_6}$ is constructed algebraically in the expected form under the assumption of finite dimensionality, it violates the group structure of $\mathbb{S}_6$ (e.g.\ the center of $\mathbb{S}_6$ is trivial), however, which gives a contradiction.
\end{proof}

\begin{proof}[Intuition of the paper]

Whenever the algebra $\EuScript{B}_W$ is finite dimensional, we expect certain commutativity relations to hold, either on the nose under correspondent additional assumptions, or up to multiplication with a nonzero element in $\EuScript{B}_W$. We make this intuition precise in Lemma~\ref{lem:motiv}:~$\eqref{item:equiv-1}\Rightarrow\eqref{item:equiv-3}$, Theorem~\ref{thm:concrete} and Corollary~\ref{cor:abstr-comm}. In the setup of Theorem~\ref{thm:inf_Bm-intro}, such a commutativity relation is visible because once $\EuScript{B}_{\mathbb{S}_6}$ is assumed to be finite dimensional, the commutativity relations from Lemma~\ref{lem:motiv}:~$\eqref{item:equiv-1}\Rightarrow\eqref{item:equiv-3}$ are satisfied as we see form the proof of Theorem~\ref{thm:mainS6}. In some sense, we can also interpret this in addition to the group theoretic contradiction against the structure of $\mathbb{S}_6$ mentioned in the paragraph \textit{{\enquote{Strategy of the proof of Theorem~\ref{thm:inf_Bm-intro}}}} by saying that finite dimensionality of $\EuScript{B}_{\mathbb{S}_6}$ implies to many commutativity relations as are actually satisfied by definition in the Woronowicz ideal (cf.\ \cite[Subsection~3.3]{bazlov1} or \cite[Subsection~2.3]{skew}).
\end{proof}

\subsection*{References}

Although it did not become apparent by what we said up to now, the results in this paper rely on braided differential calculus as developed in \cite{bazlov1,liu,liusubalgebras}. Even if many of the original ideas are due to \cite{liu}, we will often refer to \cite{skew} for similar statements because the latter paper is written in a generality and language which is more suitable for us. 

\subsection*{Organization}

In Section~\ref{sec:coxeter}-\ref{sec:nilcoxeter}, we setup common terminology and notation which will be used from thereon. In Section~\ref{sec:inversion}--\ref{sec:invariance}, we deepen some aspects of the braided differential calculus developed in \cite{bazlov1,liu}, in particular with a view towards integrals in braided Hopf algebras and commutativity relations and invariance properties under the assumption of their existence. Section~\ref{sec:motiv}-\ref{sec:coproduct} are not strictly necessary to be read to understand the proof of Theorem~\ref{thm:inf_Bm-intro}, but they serve as a good illustration of the principle of commutativity explained in the paragraph \textit{\enquote{Intuition of the paper}} and, in particular, Section~\ref{sec:motiv} motivates the integral structure which must be observed in the finite dimensional case as mentioned in the paragraph \enquote{\textit{Strategy of the proof of Theorem~\ref{thm:inf_Bm-intro}}}. Section~\ref{sec:reduction}--\ref{sec:consequences} finally contain the proof of Theorem~\ref{thm:inf_Bm-intro}. While we work mostly in the generality of $\EuScript{B}_W$ as explained in the paragraph \textit{\enquote{Setup of the paper}}, it should be noted that some of our results, in particular Theorem~\ref{thm:concrete}, work even in greater generality, i.e.\ for Nichols algebras associated to more general Yetter-Drinfeld modules. We point out the details concerning these generalizations in Subsection~\ref{subsec:general}.

\subsection*{Context}

It is a recurrent theme in the literature to ask which groups admit a finite dimensional Nichols algebra. This question has been studied in particular for symmetric groups, alternating groups and dihedral groups \cite{symmetric,dihedral-irred,symmetric0,dihedral-red}. This paper can be seen in line with these works. Especially, the positive solution of Theorem~\ref{thm:inf_Bm-intro} supplements the classification theorem \cite[Theorem~1.1]{symmetric} of finite dimensional Nichols algebras over symmetric groups. It should be noted that there exists a general theory which allows, among other things, to treat the question of dimensionality of a Nichols algebra whenever it is associated to a \emph{reducible} Yetter-Drinfeld module over a Hopf algebra with invertible antipode, see for example \cite{lyndon,heckenberger1,heckenberger2,heckenberger2',recent,heckenberger2020hopfbook}. This theory is based on combinatorics on Lyndon words over the alphabet given by an index set of the irreducible components of the underlying Yetter-Drinfeld module. However, whenever the underlying Yetter-Drinfeld module is \emph{irreducible}, the theory of Andruskiewitsch, Heckenberger, Schneider et~al.\ gives no information according to the author's understanding. Thus, it cannot be suitable to proof Theorem~\ref{thm:inf_Bm-intro}.  

\subsection*{Acknowledgment}


We want to thank Nicolas Perrin and David Benson for giving a first impression on this paper. We want to thank István Heckenberger for taking a closer look at the results, giving some feedback and pointing out the references~\cite{bastian,integrale5}. Finally, we want to thank \begin{CJK}{UTF8}{gbsn}黄玲茹\end{CJK}, \begin{CJK}{UTF8}{gbsn}常林培\end{CJK}, \begin{CJK}{UTF8}{gbsn}孟娜\end{CJK} and \begin{CJK}{UTF8}{gbsn}丁思雨\end{CJK} for their support since 2019. Last but not least, we want to thank Christl Hedwig for her support since 2012.

\section{\for{toc}{{\color{white}0}}Coxeter groups}
\label{sec:coxeter}

We fix once and for all a finite Coxeter system $(W,S)$. We assume throughout that $W\neq 1$ or equivalently that $S\neq\varnothing$. We denote by $m(s,s')$ the necessarily finite order of $ss'$ in $W$ (cf.~\cite[Proposition~5.3]{humphreys-coxeter}). We denote by $T=\bigcup_{w\in W}wSw^{-1}$ the set of all reflections of $W$. We denote by $\mathfrak{h}$ the geometric representation of $W$ as in \cite[Section~5.3]{humphreys-coxeter}, i.e.\ $\mathfrak{h}$ is a real vector space of dimension $\left|S\right|$ with basis given by the set of simple roots $\Delta$, and each simple root $\beta$ corresponds bijectively to a simple reflection $s_\beta$ in $S$ which acts on $\mathfrak{h}$ via the formula $s_\beta(x)=x-2B(x,\beta)\beta$, where $B$ is the $W$-invariant scalar product on $\mathfrak{h}$ uniquely determined by the assignment
\[
B(\beta,\beta')=-\cos\frac{\pi}{m(s_\beta,s_{\beta'})}
\]
on simple roots $\beta,\beta'$ (cf.~\cite[Chapter~V, \S~4, n\textsuperscript{o}~8, Theorem~2]{bourbaki_roots} and \cite[Proposition~5.3]{humphreys-coxeter}). Recall that the geometric representation $\mathfrak{h}$ of $W$ is faithful by \cite[Corollary~5.4]{humphreys-coxeter}.

To the situation above, we attach a root system $R$ in $\mathfrak{h}$ and a partial order \enquote{$\leq$} on $\mathfrak{h}$ considered as an abelian group as in \cite[Section~5.4]{humphreys-coxeter}. The root system $R$ is a root system in the weak sense of \cite[Section~1.2]{humphreys-coxeter}, i.e.\ it is a finite subset of nonzero vectors in $\mathfrak{h}$ satisfying the axioms
\begin{enumerate}
    \item[\optionaldesc{(R1)}{axiom:R1}]
    $R\cap\mathbb{R}\alpha=\{-\alpha,\alpha\}$ for all $\alpha\in R$,
    \item[(R2)]
    $w(R)=R$ for all $w\in W$,
\end{enumerate}
in particular, it is in general not crystallographic in the sense of \cite[Section~2.9]{humphreys-coxeter}. We refer to Axiom~\ref{axiom:R1} by saying that $R$ is reduced. We define the set of positive roots $R^+$ and the set of negative roots $R^-$ by the equations $R^+=\{\alpha\in R\mid\alpha>0\}$ and $R^-=\{\alpha\in R\mid\alpha<0\}$. Note that we have $R=R^+\cup R^-$ where the union is obviously disjoint (cf.~\cite[Theorem~5.4]{humphreys-coxeter}). For a subset $\Theta$ of $R$, we denote by $-\Theta$ the set of roots $-\Theta=\{-\alpha\mid\alpha\in\Theta\}$. With this notation, we have for example $R^-=-R^+$. Because $R$ is reduced, note that $R^+$ and $R^-$ are in bijection with $T$. Each time, the bijection is given by the assignment $\alpha\mapsto s_\alpha$ where $s_\alpha$ is the reflection associated to $\alpha$ as in \cite[Section~5.7]{humphreys-coxeter}, which acts on $\mathfrak{h}$ via the formula $s_\alpha(x)=x-2B(x,\alpha)\alpha$ analogously as in the case of a simple root $\alpha=\beta$. This formula clearly implies $s_\alpha=s_{-\alpha}$ for all $\alpha\in R$.

We denote by $\ell$ the length function on $W$ as defined in \cite[Section~5.2]{humphreys-coxeter}. We denote by $w_o$ the longest element of $W$, i.e. the unique element in $W$ of maximum length (cf.~\cite[Section~5.6, Exercise~2]{humphreys-coxeter}). We denote by \enquote{$\leq$} the (strong) Bruhat order on $W$ as defined in \cite[Section~5.9]{humphreys-coxeter}.

\begin{rem}

Whenever $W$ is a simply laced Weyl group or equivalently whenever $R$ is a simply laced root system, every root $\alpha$ written as a linear combination of simple roots has integral coefficients, and we define the height of $\alpha$, in formulas $\operatorname{ht}(\alpha)$, as the sum of those coefficients (cf.~\cite[Section~3.20, p.~83]{humphreys-coxeter}). In this situation, we define further a $W$-invariant scalar product on $\mathfrak{h}$ given by $(-,-)=2B(-,-)$ which has the property that $(\alpha,\gamma)\in\{-1,0,1\}$ for all non-proportional roots $\alpha,\gamma$ and that $(\alpha,\alpha)=2$ for all roots $\alpha$.

\end{rem}

\begin{rem}

In this remark, we want to recall some basic facts about the Bruhat order on $W$ which we will use from now on without reference. For a positive root $\alpha$, we have by \cite[Proposition~5.7]{humphreys-coxeter} the following equivalences (to be read from top to bottom):
\begin{align*}
&\begin{aligned}
&&ws_\alpha&<w\\
\Longleftrightarrow&&\ell(ws_\alpha)&<\ell(w)\\
\Longleftrightarrow&&w(\alpha)&<0   
\end{aligned}
&&\pmb{\left.\vphantom{\begin{aligned}
&&ws_\alpha&<w\\
\Longleftrightarrow&&\ell(ws_\alpha)&<\ell(w)\\
\Longleftrightarrow&&w(\alpha)&<0   
\end{aligned}}\right|}&&\begin{aligned}
&&ws_\alpha&>w\\
\Longleftrightarrow&&\ell(ws_\alpha)&>\ell(w)\\
\Longleftrightarrow&&w(\alpha)&>0
\end{aligned}\\
\intertext{If $\alpha=\beta$ is a simple root, it follows from the strong exchange condition \cite[Theorem~5.8]{humphreys-coxeter} that the above equivalences are further equivalent to:}
&\begin{aligned}
\Longleftrightarrow&&\text{\parbox[c]{2.1in}{there exists a reduced expression of $w$ which ends with $s_\beta$}}
\end{aligned}
&&\pmb{\left.\vphantom{\begin{aligned}
\Longleftrightarrow&&\text{\parbox[c]{2.1in}{there exists a reduced expression of $w$ which ends with $s_\beta$}}
\end{aligned}}\right|}&&\begin{aligned}
\Longleftrightarrow&&\text{\parbox[c]{2.1in}{no reduced expression of $w$ ends with $s_\beta$}}
\end{aligned}\\
\intertext{By \cite[Section~5.2, Equation~(L1), Exercise~5.9]{humphreys-coxeter}, for a positive root $\alpha$, there also exist left analogues of the above two lines of equivalences as follows:}
&\begin{aligned}
&&s_\alpha w&<w\\
\Longleftrightarrow&&\ell(s_\alpha w)&<\ell(w)\\
\Longleftrightarrow&&w^{-1}(\alpha)&<0   
\end{aligned}
&&\pmb{\left.\vphantom{\begin{aligned}
&&s_\alpha w&<w\\
\Longleftrightarrow&&\ell(s_\alpha w)&<\ell(w)\\
\Longleftrightarrow&&w^{-1}(\alpha)&<0   
\end{aligned}}\right|}&&\begin{aligned}
&&s_\alpha w&>w\\
\Longleftrightarrow&&\ell(s_\alpha w)&>\ell(w)\\
\Longleftrightarrow&&w^{-1}(\alpha)&>0   
\end{aligned}\\
\intertext{If $\alpha=\beta$ is a simple root, it follows further that the above equivalences are equivalent to:}
&\begin{aligned}
\Longleftrightarrow&&\text{\parbox[c]{2.1in}{there exists a reduced expression of $w$ which starts with $s_\beta$}}
\end{aligned}
&&\pmb{\left.\vphantom{\begin{aligned}
\Longleftrightarrow&&\text{\parbox[c]{2.1in}{there exists a reduced expression of $w$ which starts with $s_\beta$}}
\end{aligned}}\right|}&&\begin{aligned}
\Longleftrightarrow&&\text{\parbox[c]{2.1in}{no reduced expression of $w$ starts with $s_\beta$}}
\end{aligned}
\end{align*}

\end{rem}

\begin{rem}[The center of $W$]
\label{rem:center}

Let $(W,S)=\prod_i(W_i,S_i)$ be the decomposition of $(W,S)$ into irreducible components where $(W_i,S_i)$ is an irreducible finite Coxeter system corresponding to the $i$\textsuperscript{th} connected component of the Coxeter diagram of $(W,S)$ (cf.~\cite[Proposition~6.1]{humphreys-coxeter}). The geometric representation $\mathfrak{h}$ of $W$ decomposes as a direct sum $\bigoplus_i\mathfrak{h}_i$ into irreducible components where $\mathfrak{h}_i$ is the geometric representation of $W_i$ (cf.~\cite[Chapter~V, \S~4, n\textsuperscript{o}~8, Corollary of Theorem~2]{bourbaki_roots}). The center of $W$ decomposes into a product $\prod_i Z_i$ where $Z_i$ is the center of $W_i$. For each $i$, the group $Z_i$ embeds into $\mathbb{Z}/2\mathbb{Z}$ where the $i$\textsuperscript{th} embedding is an isomorphisms if and only if $w_o$ acts as minus the identity on $\mathfrak{h}_i$ (in which case $Z_i=\{1,w_{o,i}\}$ where $w_{o,i}$ is the longest element of $W_i$ and where $Z_i$ is trivial otherwise). This description of the center of $W$ was discussed in \cite[Lemma~4.5]{quasi-II}. The center of $W$ is relevant for this paper because of Lemma~\ref{lem:char}.

\end{rem}

\section{\for{toc}{{\color{white}0}}Disjoint systems in Coxeter groups}
\label{sec:disjoint}

In this section, we introduce the notion of disjoint systems in Coxeter groups. In the form we define it, it is only suitable for finite Coxeter groups. This notion will be present in many of our considerations.

\begin{defn}

Let $w\in W$. We denote by $T_w$ the subset of $R^+$ defined by
\[
T_w=\{\alpha\in R^+\mid s_\alpha\in wSw^{-1}\}\,.
\]

\end{defn}

\begin{fact}
\label{fact:T}

Let $w$ be an element in the centralizer of $w_o$. Then, we have $T_{w_ow}=T_{ww_o}=T_w$.

\end{fact}

\begin{proof}

By assumption on $w$, the equality $T_{w_ow}=T_{ww_o}$ is clear. The equality $T_{ww_o}=T_w$ follows because $w_o$ is an involution and because $w_oSw_o=S$ by \cite[Section~5.6, Exercise~2]{humphreys-coxeter}.
\end{proof}

\begin{rem}
\label{rem:T}

Fact~\ref{fact:T} has a converse if $(W,S)$ is irreducible. In this case, if $w,w'\in W$ are such that $T_w=T_{w'}$, it follows that $w=w'$ or $w=w'w_o$. This is clear from \cite[Section~5.6, Exercise~2]{humphreys-coxeter} and the existence of a unique highest root for irreducible Coxeter systems (i.e.\ a root $\theta_1$ such that $B(\theta_1,\beta)\geq 0$ for all $\beta\in\Delta$).

\end{rem}

\begin{defn}

\leavevmode

\begin{itemize}
    \item 
    We say that $D$ is a disjoint system if $D$ is contained in the centralizer of $w_o$ and if the union $\bigcup_{w\in D}wSw^{-1}$ is disjoint.
    \item
    If $D$ is a disjoint system, we call the cardinality of $D$ the order of $D$.
    \item
    We say that a disjoint system $D$ is complete if the order of $D$ is equal to $\smash{\frac{\left|T\right|}{\left|S\right|}}$, in other words, if $D$ is a disjoint system such that $T=\smash{\bigcup_{w\in D}wSw^{-1}}$ where the union is disjoint.
    \item
    We say that a disjoint system $D$ is normalized if $1\in D$.
    \item
    If we want to emphasize the group $W$ with respect to which a disjoint system is defined, we explicitly speak about a disjoint system in $W$.
\end{itemize}

\end{defn}

\begin{rem}

Note that we allow the empty set as disjoint system of order zero in any $W$.

\end{rem}

\begin{rem}
\label{rem:sub-disj}

If $D$ is a disjoint system, then any subset of $D$ is also a disjoint system.

\end{rem}

\begin{rem}
\label{rem:coxeter}

We remark that the integer $\smash{\frac{2\left|T\right|}{\left|S\right|}}\in\mathbb{Z}_{>0}$ is called the Coxeter number of $W$ (cf.~\cite[Proposition~3.18]{humphreys-coxeter}). Note that half the Coxeter number of $W$ appears in the definition of a complete disjoint system. If a complete disjoint system $D$ exists, then the Coxeter number of $W$ is even and half the Coxeter number of $W$ is equal to the order of $D$.

\end{rem}

\begin{rem}

Let $D$ be a disjoint system of order $r$. Let $w_1,\ldots,w_s\in D$. Then, the set 
\[
D'=(D\setminus\{w_1,\ldots,w_s\})\cup\{v_1,\ldots,v_s\}
\]
where $v_i=w_ow_i=w_iw_o$ for all $1\leq i\leq s$ is also a disjoint system of order $r$. This follows from Fact~\ref{fact:T} and because the centralizer of $w_o$ is a subgroup of $W$ which contains the involution $w_o$. The difference between $D$ and disjoint systems $D'$ derived from $D$ in this way will always be irrelevant for our applications in this paper.

\end{rem}

\begin{lem}
\label{lem:transl_D}

Let $D$ be a disjoint system of order $r$. Let $v$ be an element in the centralizer of $w_o$. Then, the set $vD$ is also a disjoint system of order $r$.

\end{lem}

\begin{proof}

Let the notation be as in the statement. Let $D'=vD$ for short. It is clear that $D'$ has the same cardinality as $D$, that $D'$ is contained in the centralizer of $w_o$ (because this centralizer is a subgroup of $W$), and that the union $\bigcup_{w\in D}vwSw^{-1}v^{-1}$ is disjoint as a translate of the disjoint union $\bigcup_{w\in D}wSw^{-1}$. By replacing $w$ with $v^{-1}w$ in the translated disjoint union, we find that the union $\bigcup_{w\in D'}wSw^{-1}$ is also disjoint. Thus, we know that $D'$ is a disjoint system of order $r$.
\end{proof}

\begin{cor}[Normalization of disjoint systems]
\label{cor:normal}

Let $D$ be a disjoint system of order $r$. For all $v\in D$, the set $v^{-1}D$ is a normalized disjoint system of order $r$. In particular, if a complete disjoint system exists, then there exists also a normalized complete disjoint system.

\end{cor}

\begin{proof}

The particular case is immediate from the more general statement. Let $D$ be a disjoint system of order $r$. Let $v\in D$ and let $D'=v^{-1}D$. Note that $v$ and $v^{-1}$ are both contained in the centralizer of $w_o$, by assumption, and because this centralizer is a subgroup of $W$. If we apply Lemma~\ref{lem:transl_D} to $D$ and $v^{-1}$, we find that $D'$ is a disjoint system of order $r$. But $D'$ is also normalized by definition.
\end{proof}

\begin{rem}

The most urgent combinatorial question raised by the notion of disjoint systems is as follows: If the Coxeter number of $W$ is even (which is necessary by Remark~\ref{rem:coxeter}), does there always exist a complete disjoint system?

\end{rem}

\begin{rem}

Let $D$ be a normalized disjoint system. Then, every element $w\in D\setminus\{1\}$ satisfies $wSw^{-1}\cap S=\varnothing$. If $D$ is a normalized disjoint system in $\mathbb{S}_m$, elements of $D\setminus\{1\}$ are therefore a special example of so-called permutations without rising or falling successions. These permutations in general are subject to combinatorial studies and enumerations. We refer to \cite{OEIS} for a list of literature.\footnote{We want to thank Richard Stanley for pointing out the reference \cite{OEIS} in a comment to our question on MathOverflow \cite{MO}.}

\end{rem}

\subsection*{Disjoint systems in the symmetric group}

For the symmetric group $\mathbb{S}_m$, we write permutations either in one-line notation or in cycle notation where the difference between the two notations is visible by the absence or presence of parenthesis. Sometimes, we also write permutations in $\mathbb{S}_m$ as bijections of $\{1,\ldots,m\}$. 

\begin{rem}
\label{rem:sym_wo}

Recall that the longest element of $\mathbb{S}_m$ is simply given by $m(m-1)\cdots 1$. Hence, the centralizer of the longest element of $\mathbb{S}_m$ is given by permutations $\sigma$ such that $\sigma(i)+\sigma(j)=m+1$ for all $1\leq i,j\leq m$ where $i+j=m+1$.

\end{rem}

\begin{ex}[Complete disjoint system in $\mathbb{S}_6$]
\label{ex:S6}

For this example, let us assume that $W=\mathbb{S}_6$ and that the notation from Section~\ref{sec:coxeter} is realized for $\mathbb{S}_6$. Let us consider the elements $w_1=241635$ and $w_2=315264$ in $\mathbb{S}_6$. By Remark~\ref{rem:sym_wo}, we see that both $w_1$ and $w_2$ lie in the centralizer of the longest element of $\mathbb{S}_6$. Further, the union
\[
w_1Sw_1^{-1}\cup w_2 Sw_2^{-1}=\{(24),(14),(16),(36),(35)\}\cup\{(13),(15),(25),(26),(46)\}
\]
is disjoint and equal to $T\setminus S$. Hence, we conclude that $\{w_1,w_2,1\}$ is a normalized complete disjoint system of order three in $\mathbb{S}_6$. We have illustrated this disjoint system in Figure~\ref{fig:s6}. Note that the nontrivial elements of this disjoint system satisfy the additional relations
\[
w_1^{-1}=w_2\,,\,w_2^{-1}=w_1\,,\,w_1^{-1}w_2=w_2^2=w_1w_o\,,\,w_2^{-1}w_1=w_1^2=w_2w_o\,,
\]
and that as a consequence the partition of $T=w_1Sw_1^{-1}\cup w_2Sw_2^{-1}\cup S$ is invariant under conjugation with $w_1$, $w_2$. This example shows in particular that a complete disjoint system in $\mathbb{S}_6$ exists.

\end{ex}

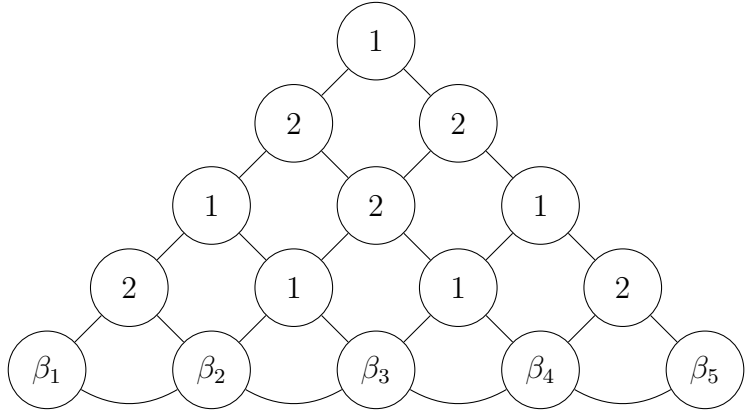
\begin{figure}
\centering
\begin{tikzpicture}[node distance=0.5cm]
   \node[state] (b_1) {$\beta_1$}; 
   \node[state] (c_1) [above right=of b_1] {$2$}; 
   \node[state] (b_2) [below right=of c_1] {$\beta_2$};
   \node[state] (c_2) [above right=of b_2] {$1$};
   \node[state] (b_3) [below right=of c_2] {$\beta_3$};
   \node[state] (c_3) [above right=of b_3] {$1$};    
   \node[state] (b_4) [below right=of c_3] {$\beta_4$};  
   \node[state] (c_4) [above right=of b_4] {$2$}; 
   \node[state] (b_5) [below right=of c_4] {$\beta_5$};
   \node[state] (d_1) [above right=of c_1] {$1$}; 
   \node[state] (d_2) [above right=of c_2] {$2$}; 
   \node[state] (d_3) [above right=of c_3] {$1$};
   \node[state] (e_1) [above right=of d_1] {$2$};
   \node[state] (e_2) [above right=of d_2] {$2$};   
   \node[state] (f_1) [above right=of e_1] {$1$}; 
   \path[-,bend right] 
   (b_1) edge (b_2)
   (b_2) edge (b_3)
   (b_3) edge (b_4)
   (b_4) edge (b_5);
   \path[-]
   (b_1) edge (c_1)
   (b_2) edge (c_2)
   (b_3) edge (c_3)
   (b_4) edge (c_4)
   (c_1) edge (d_1)
   (c_2) edge (d_2)
   (c_3) edge (d_3)
   (d_1) edge (e_1)
   (d_2) edge (e_2)
   (e_1) edge (f_1)
   (b_2) edge (c_1)
   (b_3) edge (c_2)
   (b_4) edge (c_3)
   (b_5) edge (c_4)
   (c_2) edge (d_1)
   (c_3) edge (d_2)
   (c_4) edge (d_3)
   (d_2) edge (e_1)
   (d_3) edge (e_2)
   (e_2) edge (f_1);
\end{tikzpicture}
\caption{For this figure, we assume that $W=\mathbb{S}_6$ and that the notation from Section~\ref{sec:coxeter} is realized for $\mathbb{S}_6$. Let $w_1$ and $w_2$ be the elements of $\mathbb{S}_6$ as defined in Example~\ref{ex:S6}. The barycentric graph illustrates all positive roots in $R^+$ where $\beta_1,\beta_2,\beta_3,\beta_4,\beta_5$ denote all simple roots in $\Delta$ with the labeling as in \cite[Plate~I]{bourbaki_roots}. The positive roots labeled with $1$ and $2$ correspond each up to multiplication with $w_o$ from the right to a permutation in $\mathbb{S}_6$ (in this case to $w_1,w_1w_o$ and $w_2,w_2w_o$, cf.~Remark~\ref{rem:T}) because they form a diagram with straight lines isomorphic to the Dynkin diagram of type $\mathsf{A}_5$ where we possibly allow reflection along the horizontal bottom line below the graph. They correspond to elements in the centralizer of the longest element of $\mathbb{S}_6$ because their diagrams are symmetric with respect to the vertical line in the middle of the graph. They correspond to permutations without rising or falling succession because none of them is simple. If we denote by $T_1,T_2$ the set of positive roots labeled with $1,2$, respectively, then we have $T_1=T_{w_1}$ and $T_2=T_{w_2}$.}
\label{fig:s6}
\end{figure}

\section{\for{toc}{{\color{white}0}}Nichols algebras and braided differential calculus}

In this section, we recall some basic facts about Nichols algebras and braided differential calculus. Braided differential calculus is the calculus of braided partial derivatives acting on a Nichols algebra associated to a finite dimensional Yetter-Drinfeld module, and defined in terms of a nondegenerate Hopf duality pairing. For more details concerning the general theory presented in this section, we refer to \cite{integrals-braided,andruskiewitsch,skew,bazlov1,bazlov2,kassel,liu,majid,heckenberger2020hopfbook}. More specifically, we refer for
\begin{itemize}
    \item
    $\mathrlap{\text{Subsection~\ref{subsec:YDHopf} to }}\hphantom{\text{Subsection~\ref{subsec:nichols} to }}$\cite[Subsection~1.1--1.3]{andruskiewitsch}, \cite[Subsection~2.1]{bazlov1}, \cite[Subsection~5.1,~5.2]{bazlov2},
    \item
    Subsection~\ref{subsec:nichols} to \cite[Subsection~2.1]{andruskiewitsch}, \cite[Section~2,~3]{bazlov1}, \cite[Sibsection~5.3--5.8]{bazlov2},
    \item
    Subsection~\ref{subsec:YDGamma} to \cite[Section~3]{skew}, \cite[Subsection~2.3]{liu},
    \item
    Subsection~\ref{subsection:model-Schubert} to \cite[Subsection~4.1,~4.2]{skew}, \cite[Section~4]{bazlov1}.
\end{itemize}

\setcounter{tocdepth}{2}
\subsection{Yetter-Drinfeld category over a Hopf algebra}
\label{subsec:YDHopf}

We fix once and for all a Hopf algebra $H$ over $\mathbf{k}$ with invertible antipode. We denote by ${}_H^H\EuScript{YD}$ the Yetter-Drinfeld category over $H$. Recall that the objects in ${}_H^H\EuScript{YD}$ are Yetter-Drinfeld $H$-modules, i.e.\ $\mathbf{k}$-vector spaces $V$ which are simultaneously left $H$-modules and left $H$-comodules which satisfy the compatibility condition
\[
(hv)_{(-1)}\otimes(hv)_{(0)}=h_{(1)}v_{(-1)}\EuScript{S}(h_{(3)})\otimes h_{(2)}v_{(0)}
\]
for all $h\in H$ and all $v\in V$. Here and everywhere else where it is suitable, we use sumless Sweedler notation for coproducts and coactions. In Section~\ref{sec:coproduct}, we will be however obliged to choose specific decompositions of coproducts in which case we renounce to use (sumless) Sweedler notation. By \cite{compatibility}, the above compatibility condition for a Yetter-Drinfeld $H$-module $V$ can be equivalently expressed as
\[
h_{(1)}v_{(-1)}\otimes h_{(2)}v_{(0)}=(h_{(1)}v)_{(-1)}h_{(2)}\otimes (h_{(1)}v)_{(0)}
\]
for all $h\in H$ and all $v\in V$. The morphisms in ${}_H^H\EuScript{YD}$ are the obvious ones, i.e.\ those which preserve the $H$-module and $H$-comodule structure. The Yetter-Drinfeld category over $H$ is a braided monoidal category where the braiding $\Psi$ of two Yetter-Drinfeld $H$-modules $V$ and $W$ is given by $\Psi(v\otimes w)=v_{(-1)}w\otimes v_{(0)}$ with inverse $\Psi^{-1}(w\otimes v)=v_{(0)}\otimes\EuScript{S}^{-1}(v_{(-1)})w$ for all $v\in V$ and all $w\in W$.\footnote{In this subsection, we do not need Coxeter groups. We will use temporarily $W$ with another meaning, and continue with the original setup from the previous sections in Subsection~\ref{subsection:model-Schubert}.} Just as in the previous sentence, we usually suppress the components of the braiding $\Psi$. Note that every finite dimensional Yetter-Drinfeld $H$-module $V$ is rigid in the sense of \cite[Subsection~1.1(d)]{andruskiewitsch} or \cite[Definition~9.3.1]{majid}, i.e.\ it admits a left dual, one isomorphic copy of which we denote by $V^*$, namely the linear dual of $V$ equipped with the structure of a Yetter-Drinfeld $H$-module uniquely determined by the formulas
\begin{align*}
(hf)(v)&=f(\EuScript{S}(h)v)\,,\\
f_{(-1)}f_{(0)}(v)&=\EuScript{S}^{-1}(v_{(-1)})f(v_{(0)})
\end{align*}
for all $h\in H$, $v\in V$, $f\in V^*$.

Whenever $A$ is a braided $\mathbb{Z}_{\geq 0}$-graded Hopf algebra in ${}_H^H\EuScript{YD}$, we denote by $A^m$ the component of $A$ of $\mathbb{Z}_{\geq 0}$-degree $m$, i.e.\ we always use superscripts to indicate $\mathbb{Z}_{\geq 0}$-graded components, and we follow this convention even on the level of elements in $A$. For two braided Hopf algebras $A$ and $B$ in ${}_H^H\EuScript{YD}$, we say according to \cite[Subsection~5.3]{bazlov2} or \cite[Definition~1.4.3]{majid} that a morphism $\left<-,-\right>\colon A\otimes B\to\mathbf{k}$ in the Yetter-Drinfeld category over $H$ is a Hopf duality pairing between $A$ and $B$ if it satisfies the axioms
\begin{align*}
\left<\phi\psi,x\right>&=\left<\phi,x_{(2)}\right>\left<\psi,x_{(1)}\right>\,,\\
\left<\phi,xy\right>&=\left<\phi_{(2)},x\right>\left<\phi_{(1)},y\right>
\end{align*}
which entail the properties
\[
\left<1,x\right>=\epsilon(x)\,,\quad
\left<\phi,1\right>=\epsilon(\phi)\,,\quad
\left<\EuScript{S}(\phi),x\right>=\left<\phi,\EuScript{S}(x)\right>
\]
by \cite{axioms}, \cite[Section~III.9]{kassel}, \cite[Proposition~1.3.1]{majid}, and where everywhere in the two previous displayed equations $\phi,\psi\in A$ and $x,y\in B$. For a braided Hopf algebra $A$ in ${}_H^H\EuScript{YD}$, we call a Hopf duality pairing between $A$ and $A$ simply a Hopf duality pairing between $A$ and itself. For two braided $\mathbb{Z}_{\geq 0}$-graded Hopf algebras $A$ and $B$ in ${}_H^H\EuScript{YD}$, we say that a Hopf duality pairing $\left<-,-\right>$ between $A$ and $B$ respects the $\mathbb{Z}_{\geq 0}$-grading if $\left<-,-\right>$ restricted to $A^m\otimes B^{m'}$ is zero unless $m=m'$.

\begin{rem}

Let $\left<-,-\right>$ be a Hopf duality pairing between $A$ and $B$ where $A$ and $B$ are two braided $\mathbb{Z}_{\geq 0}$-graded Hopf algebras. Because $\left<-,-\right>$ is a morphism in the Yetter-Drinfeld category over $H$, we have the following invariance properties
\[
\left<hx,y\right>=\langle x,\EuScript{S}^{-1}(h)y\rangle\quad\text{and}\quad\left<x,hy\right>=\langle\EuScript{S}(h)x,y\rangle
\]
where $h\in H$, $x\in A$, $y\in B$.

\end{rem}

\begin{rem}
\label{rem:ident-into-end}

Every Yetter-Drinfeld $H$-module $V$ also carries a natural structure of a right $H$-module given by $vh=\EuScript{S}^{-1}(h)v$ for all $h\in H$ and all $v\in V$. This structure is natural as it gives rise to an equivalence of categories as discussed in \cite[Proposition~2.2.1, Item~(1):~$(\mathrm{i})\Leftrightarrow(\mathrm{iii})$]{integrals-braided}. We will equip from now on every $V$ with this additional structure, and consider every $h\in H$ either as plain element in $H$ or as homothety in $\operatorname{End}_{\mathbf{k}}V\,$\footnote{Here and in what follows, $\operatorname{End}_{\mathbf{k}}$ and the word \enquote{endomorphism} without further specification refer to endomorphisms of vector spaces.} acting from the left or the right on $V$, where the direction will be clear from context or will be indicated by evaluation on elements in $V$ or placeholders $(-)$. 

Whenever $A$ is a braided Hopf algebra in ${}_H^H\EuScript{YD}$, we consider elements in $A$ sometimes as endomorphisms of $A$ given by multiplication from the left or the right and acting correspondingly. It is finally convenient to consider the antipode of such an $A$ (and its inverse if existent) as endomorphisms of $A$ acting either on the left or the right. Each time, for elements in $A$ considered as endomorphisms as well as for the antipode of $A$ (and its inverse if existent), the acting direction will be clear from context or will be indicated by evaluation on elements in $A$ or placeholders $(-)$. 

Let $A$ be a braided Hopf algebra in ${}_H^H\EuScript{YD}$. With the conventions in the two previous paragraphs in mind, we have
\[
(-)zh=(-)h_{(2)}(zh_{(1)})\quad\text{and}\quad hz(-)=(h_{(1)}z)h_{(2)}(-)
\]
for all $h\in H$ and all $z\in A$. We will use the first of the two formulas above in the proof of Theorem~\ref{thm:leibniz}.

\end{rem}

\subsection{Nichols algebras and braided differential calculus}
\label{subsec:nichols}

Let $V$ be a Yetter-Drinfeld $H$-module. We denote by $\EuScript{B}(V)$ the Nichols-Woronowicz algebra of $V$. We call it simply the Nichols algebra of $V$ from now on, or the Nichols algebra associated to $V$. By \cite[Definition~2.1, Proposition~2.2]{andruskiewitsch}, this is a braided $\mathbb{Z}_{\geq 0}$-graded Hopf algebra in ${}_H^H\EuScript{YD}$ uniquely determined up to isomorphism by the axioms:
\begin{enumerate}
    \item[(N1)]
    $\EuScript{B}(V)$ is connected, i.e.\ $\EuScript{B}(V)^0=\mathbf{k}1$,
    \item[(N2)]
    $\EuScript{B}(V)$ is generated as an algebra in $\mathbb{Z}_{\geq 0}$-degree one,
    \item[(N3)]
    the Yetter-Drinfeld $H$-module consisting of all primitive elements in $\EuScript{B}(V)$ equals $\EuScript{B}(V)^1$ which is in turn equal to $V$.
\end{enumerate}
The explicit construction of $\EuScript{B}(V)$ is discussed in \cite[Subsection~2.1]{andruskiewitsch}, \cite[Subsection~3.1]{bazlov1}, \cite[Subsection~5.7]{bazlov2}. We assume for the rest of this subsection that $V$ is finite dimensional. Recall the important property that $\EuScript{B}(V)$ comes equipped with a nondegenerate Hopf duality pairing $\left<-,-\right>$ between $\EuScript{B}(V^*)$ and $\EuScript{B}(V)$ which respects the $\mathbb{Z}_{\geq 0}$-grading and which is uniquely determined by the property that its restriction to $V^*\otimes V$ equals the evaluation pairing of $V$. Using this pairing, one defines a right action $y\mapsto\constantoverleftarrow{D_y}$ of the algebra $\EuScript{B}(V)$ on $\EuScript{B}(V^*)$ via braided right partial derivatives $\constantoverleftarrow{D_x}$ and a left action $y^*\mapsto\constantoverrightarrow{D_{y^*}}$ of the algebra $\EuScript{B}(V^*)$ on $\EuScript{B}(V)$ via braided left partial derivatives $\constantoverrightarrow{D_{y^*}}$, where
\begin{align*}
(z^*)\constantoverleftarrow{D_y}&=z_{(1)}^*\langle z_{(2)}^*,y\rangle\,,\\
\constantoverrightarrow{D_{y^*}}(z)&=\langle y^*,z_{(1)}\rangle z_{(2)}\\
\intertext{for $y,z\in\EuScript{B}(V)$ and $y^*,z^*\in\EuScript{B}(V^*)$. We will be mostly concerned with braided right partial derivatives in this work. If it is clear from context whether left or right is meant or if it does not matter, we simply speak about braided partial derivatives. We record the formulas}
\left<x^*,yx\right>&=\big<(x^*)\constantoverleftarrow{D_y},x\big>\,,\\
\left<x^*y^*,x\right>&=\big<x^*,\constantoverrightarrow{D_{y^*}}(x)\big>\,,
\end{align*}
where $x,y\in\EuScript{B}(V)$ and $x^*,y^*\in\EuScript{B}(V^*)$, which will be used in several of our considerations without further reference, e.g.\ in the proof of Proposition~\ref{prop:Nicholsint}\eqref{item:Nicholsint2}. Having the conventions in Remark~\ref{rem:ident-into-end} in mind, braided partial derivatives are uniquely determined by the multiplicativity of the actions they define and the formulas
\begin{empheq}[box={\fboxsep=6pt\fbox}]{gather*}
\vphantom{D^{\EuScript{S}(z^{(-1)})}}z^*\constantoverleftarrow{D_v}=(z^*)\constantoverleftarrow{D_v}+\constantoverleftarrow{D_{v_{(0)}}}\big(\EuScript{S}^{-1}(v_{(-1)})z^*\big)\,,\\
\constantoverrightarrow{D_{v^*}}z=\constantoverrightarrow{D_{v^*}}(z)+z_{(0)}\constantoverrightarrow{D_{\EuScript{S}^{-1}(z_{(-1)})v^*}}\,,
\end{empheq}
where $v\in V$, $v^*\in V^*$, $z\in\EuScript{B}(V)$, $z^*\in\EuScript{B}(V^*)$. The first of these formulas is called the braided Leibniz rule and we refer to it under this name from now. The braided Leibniz rule is of utmost importance for this work and will be often in use. The second of these formulas is used only in Remark~\ref{rem:tensor-square} which follows and in the proof of Lemma~\ref{lem:basic-rev}\eqref{item:basic4'}, and is referenced also as braided Leibniz rule. Let us finally mention the following formulas
\[
\constantoverleftarrow{D_{hy}}=\EuScript{S}^2(h_{(1)})\constantoverleftarrow{D_y}\EuScript{S}^3(h_{(2)})\quad\text{and}\quad\constantoverrightarrow{D_{hy^*}}=\EuScript{S}^{-2}(h_{(1)})\constantoverrightarrow{D_{y^*}}\EuScript{S}^{-1}(h_{(2)})
\]
where $h\in H$, $y\in\EuScript{B}(V)$, $y^*\in\EuScript{B}(V^*)$, which are a straight forward generalization of \cite[Remark~3.16]{skew}.

\begin{rem}[{Embedding the tensor square into endomorphisms \cite[Lemma~1]{kirillov-maeno}}]
\label{rem:tensor-square}

Let $V$ be a finite dimensional Yetter-Drinfeld $H$-module. As a consequence of the nondegeneracy of the Hopf duality pairing between $\EuScript{B}(V^*)$ and $\EuScript{B}(V)$, we have according to \cite[Lemma~1]{kirillov-maeno} two embeddings of vector spaces
\begin{align*}
\EuScript{B}(V)\otimes\EuScript{B}(V^*)&\xhookrightarrow{\hphantom{\longrightarrow}}\operatorname{End}_{\mathbf{k}}\EuScript{B}(V^*)\,,\\
\EuScript{B}(V)\otimes\EuScript{B}(V^*)&\xhookrightarrow{\hphantom{\longrightarrow}}\operatorname{End}_{\mathbf{k}}\EuScript{B}(V)
\intertext{defined on pure tensors by the assignments}
y\otimes\xi&\longmapsto\constantoverleftarrow{D_y}\xi\,,\\
y\otimes\xi&\longmapsto y\constantoverrightarrow{D_\xi}
\end{align*}
and extended linearly. It follows from the braided Leibniz rule that the images of these embeddings inherit the structures of a $\mathbb{Z}$-graded algebra in ${}_H^H\EuScript{YD}$ whenever the antipode on $H$ is an involution, where the algebra structure is inherited from $\operatorname{End}_{\mathbf{k}}\EuScript{B}(V^*)$ and $\operatorname{End}_{\mathbf{k}}\EuScript{B}(V)$, where the structure of a Yetter-Drinfeld $H$-module is inherited from $\EuScript{B}(V)\otimes\EuScript{B}(V^*)$, and where the $\mathbb{Z}$-grading is given by the way the operators in the image manipulate the $\mathbb{Z}_{\geq 0}$-degree when applied to elements in $\EuScript{B}(V^*)$ and $\EuScript{B}(V)$, cf.~\cite[Remark~2.16]{skew}, i.e.\ the $\mathbb{Z}$-degree of $\constantoverleftarrow{D_y}\xi$ is given by $m-m'$ and that of  $y\constantoverrightarrow{D_\xi}$ by $m'-m$ whenever $\xi$ and $y$ are homogeneous of $\mathbb{Z}_{\geq 0}$-degree $m$ and $m'$.

The original context of the embeddings of the tensor square into endomorphisms is cohomology and quantum cohomology of $G/B$ where $G$ is a reductive linear algebraic group and $B$ is a Borel subgroup of $G$, cf.~\cite[Thoerem~5.4]{bazlov1}, \cite[Theorem~1]{kirillov-maeno}, \cite[Lemma~1.3]{mare}. We remark that there is a formal analogy between this construction and similar construction which evolve from generalizing Drinfeld's quantum double, cf. \cite[Chapter~IX]{kassel}, \cite[Chapter~6,~7]{majid}.

\end{rem}

\subsection{Yetter-Drinfeld category over a group}
\label{subsec:YDGamma}

We fix once and for all a group $\Gamma$. We denote the Yetter-Drinfeld category over the group algebra $\mathbf{k}\Gamma$ by ${}_\Gamma^\Gamma\EuScript{YD}$, and we call its objects Yetter-Drinfeld $\Gamma$-modules. Let $V$ be a Yetter-Drinfeld $\Gamma$-module. We always denote by $V_g$ the component of $V$ of $\Gamma$-degree $g$. We assume for the rest of this subsection that $V$ is finite dimensional and that its support in the sense of \cite[Section~4, p.~8]{milinski-schneider}, i.e.\ the set of all $g\in\Gamma$ such that $V_g\neq 0$, consists of involutions. Let $\mathcalsmall{b}$ be a homogeneous basis of $V$ with respect to the $\Gamma$-grading, and let $\mathcalsmall{b}^*$ be its dual basis of $V^*$, which is of course again homogeneous with respect to the $\Gamma$-grading. By sending a member of $\mathcalsmall{b}$ to its corresponding dual member in $\mathcalsmall{b}^*$ and extending linearly, we define an isomorphism $V\cong V^*$ of Yetter-Drinfeld $\Gamma$-modules, which in turn induces an isomorphism $\EuScript{B}(V)\cong\EuScript{B}(V^*)$ of braided $\mathbb{Z}_{\geq 0}$-graded Hopf algebras. Upon identification of $V^*$ with $V$ and $\EuScript{B}(V^*)$ with $\EuScript{B}(V)$ along these isomorphisms induced by $\mathcalsmall{b}$, the Hopf duality pairing between $\EuScript{B}(V^*)$ and $\EuScript{B}(V)$ becomes a symmetric Hopf duality pairing between $\EuScript{B}(V)$ and itself whose restriction to $V\otimes V$ has identical representation matrix when represented with respect to $\mathcalsmall{b}$. In the current situation, one introduces according to  \cite[Section~3]{skew}, \cite[Subsection~2.3]{liu} two $\mathbb{Z}_{\geq 0}$-graded endomorphisms $\rho,\bar{\EuScript{S}}$ of $\EuScript{B}(V)$ which are uniquely determined by the requirement that
\begin{enumerate}
    \item[(S1)]
    $\rho$ is the identity in $\mathbb{Z}_{\geq 0}$-degree zero and one,
    \item[(S2)]
    $\rho$ is an anti-algebra homomorphism, i.e.\ we have $\rho(xy)=\rho(y)\rho(x)$ for all $x,y\in\EuScript{B}(V)$,
    \item[(S3)]
    $\bar{\EuScript{S}}$ restricted to $\EuScript{B}(V)^m$ is given by $(-1)^m\rho\EuScript{S}$.
\end{enumerate}
As in Remark~\ref{rem:ident-into-end}, it is convenient to consider the maps $\rho$ and $\bar{\EuScript{S}}$ (as well as the antipode of $\EuScript{B}(V)$ and its inverse) as endomorphisms of $\EuScript{B}(V)$ acting either on the left or the right, where, each time, the acting direction will be clear from context or will be indicated by evaluation on elements in $\EuScript{B}(V)$ or placeholders $(-)$.

\subsection{Nichols-Woronowicz algebra model for Schubert calculus on~\texorpdfstring{$W$}{W}}
\label{subsection:model-Schubert}

Let $V_W$ be the Yetter-Drinfeld $W$-module defined as the quotient of the free vector space with basis $([\alpha])_{\alpha\in R}$ by its vector subspace $\operatorname{span}_{\mathbf{k}}\{[\alpha]+[-\alpha]\mid\alpha\in R\}$, equipped with the $W$-action $wx_\alpha=x_{w(\alpha)}$ for all $w\in W$ and all $\alpha\in R$, where $x_\alpha$ always denotes the image of $[\alpha]$ in $V_W$, and equipped with the $W$-grading given by assigning the $W$-degree $s_\alpha$ to $x_\alpha$ for all $\alpha\in R$. The Yetter-Drinfeld $W$-module $V_W$ has support $T$ and is of dimension $\left|R^+\right|$ where a canonical homogeneous basis of $V_W$ with respect to the $W$-grading is given by $(x_\alpha)_{\alpha\in R^+}$. Hence, all the assumptions of Subsection~\ref{subsec:YDGamma} are satisfied for $V_W$. As a consequence of Subsection~\ref{subsec:YDGamma}, we can and will from now on identify $V_W^*$ with $V_W$ and $\EuScript{B}(V_W^*)$ with $\EuScript{B}(V_W)$ along the isomorphisms $V_W\cong V_W^*$, $\EuScript{B}(V_W)\cong\EuScript{B}(V_W^*)$ induced by $(x_\alpha)_{\alpha\in R^+}$ as it was proposed in \cite[Subsection~4.4]{bazlov1}. Following \cite{bazlov1}, we set $\EuScript{B}_W=\EuScript{B}(V_W)=\EuScript{B}(V_W^*)$ for short and call $\EuScript{B}_W$ the Nichols-Woronowicz algebra model for Schubert calculus on $W$.

\begin{rem}
\label{rem:shortcuts}

From now on, we work mostly with the Nichols algebra $\EuScript{B}_W$, or even with the special case $\EuScript{B}_{\mathbb{S}_m}$ whenever indicated, while we point out some generalizations in the next subsection. In this situation, we use shortcuts for braided partial derivatives, namely, we write $\constantoverleftarrow{D_\alpha}=\constantoverleftarrow{D_{x_\alpha}}$ and $\constantoverrightarrow{D_\alpha}=\constantoverrightarrow{D_{x_\alpha}}$ for all $\alpha\in R$.

\end{rem}

\subsection{Generalizations}
\label{subsec:general}

We want to point out some generalizations which might be relevant for further work on the dimension of Nichols algebras in general:

\begin{itemize}
    \item
    Lemma~\ref{lem:basic-rev}\eqref{item:basic1}--\eqref{item:basic4'} works for any $\EuScript{B}(V)$ where $V$ is a finite dimensional Yetter-Drinfeld $\Gamma$-module whose support consists of involutions once we choose a homogeneous basis of $V$ with respect to the $\Gamma$-grading and do the identifications as in Subsection~\ref{subsec:YDGamma}. Lemma~\ref{lem:basic-rev}\eqref{item:basic0},\eqref{item:basic0'}, however, make use of the nilpotent relation in $\EuScript{B}_W$ which is not evident for arbitrary $\EuScript{B}(V)$ as in the previous sentence.
    \item
    Section~\ref{sec:inversion} without its unnumbered subsection, i.e.\ Proposition~\ref{prop:rhoD} and Corollary~\ref{cor:vanish_rho}, work for any $\EuScript{B}(V)$ where $V$ is a finite dimensional Yetter-Drinfeld $\Gamma$-module whose support consists of involutions once we choose a homogeneous basis of $V$ with respect to the $\Gamma$-grading and do the identifications as in Subsection~\ref{subsec:YDGamma}. 
    \item
    Corollary~\ref{cor:BWintegral}\eqref{item:sign},\eqref{item:sign_elab} except the part \enquote{$x\sim gx$} works for any finite dimensional $\EuScript{B}(V)$ where $V$ is a Yetter-Drinfeld $\Gamma$-module whose support consists of involutions. We do not have to make the exception \enquote{$x\sim gx$} if $V$ is in addition link-indecomposable in the sense of \cite[Section~4,  p.~8]{milinski-schneider}, i.e.\ if in addition the support of $V$ generates $\Gamma$.
    \item
    Corollary~\ref{cor:BWintegral}\eqref{item:Aint2_elab},\eqref{item:cor:M} works for any finite dimensional connected braided $\mathbb{Z}_{\geq 0}$-graded Hopf algebra $A$ in ${}_H^H\EuScript{YD}$ which is generated in $\mathbb{Z}_{\geq 0}$-degree one once we choose a basis of $A^1$.
    \item
    The principle of concrete commutativity in Theorem~\ref{thm:concrete} works for arbitrary finite dimensional Nicholas algebras $\EuScript{B}(V)$ associated to a Yetter-Drinfeld $\Gamma$-module $V$ upon choice of a basis of $V$.
\end{itemize}

\setcounter{tocdepth}{1}
\section{\for{toc}{{\color{white}0}}Terminology concerning monomials}

In this section, we want to setup a common language concerning monomials to speak about certain situations which will arise throughout the paper. All statements in this section are either trivial or immediate consequences of the braided Leibniz rule.

\begin{defn}
\label{def:mon}

Let $V$ be a Yetter-Drinfeld $H$-module. Let $(x_\alpha)_{\alpha\in I}$ be a basis of $V$. Let $\gamma\in I$. Let $\Theta\subseteq I$.

\begin{itemize}
    \item 
    We say that $M\in\EuScript{B}(V)$ is a monomial if there exist $\lambda\in\mathbf{k}$ and $\alpha_1,\ldots,\alpha_m\in I$ such that $M=\lambda x_{\alpha_1}\cdots x_{\alpha_m}$.
    \item
    We say that a monomial $M\in\EuScript{B}(V)$ starts with $\gamma$ if there exist $\lambda\in\mathbf{k}$ and $\alpha_1,\ldots,\alpha_m\in I$ such that $M=\lambda x_{\alpha_1}\cdots x_{\alpha_m}$ and such that $\alpha_1=\gamma$.
    \item
    We say that a monomial $M\in\EuScript{B}(V)$ ends with $\gamma$ if there exist $\lambda\in\mathbf{k}$ and $\alpha_1,\ldots,\alpha_m\in I$ such that $M=\lambda x_{\alpha_1}\cdots x_{\alpha_m}$ and such that $\alpha_m=\gamma$.
    \item
    We say that a monomial $M\in\EuScript{B}(V)$ starts with $\Theta$ if it is a monomial which starts with $\delta$ for some $\delta\in\Theta$.
    \item
    We say that a monomial $M\in\EuScript{B}(V)$ ends with $\Theta$ if it is a monomial which ends with $\delta$ for some $\delta\in\Theta$.
    \item
    We say that a monomial $M\in\EuScript{B}(V)$ does only involve $\Theta$ if there exist $\lambda\in\mathbf{k}$ and $\alpha_1,\ldots,\alpha_m\in\Theta$ such that $M=\lambda x_{\alpha_1}\cdots x_{\alpha_m}$.
    \item
    We say that an element $z\in\EuScript{B}(V)$ starts with $\gamma$ if it can be written as a sum of monomials which start with $\gamma$.
    \item
    We say that an element $z\in\EuScript{B}(V)$ ends with $\gamma$ if it can be written as a sum of monomials which end with $\gamma$.
    \item
    We say that an element $z\in\EuScript{B}(V)$ starts with $\Theta$ if it can be written as a sum of monomials which start with $\Theta$.
    \item
    We say that an element $z\in\EuScript{B}(V)$ ends with $\Theta$ if it can be written as a sum of monomials which end with $\Theta$.
    \item
    We say that an element $z\in\EuScript{B}(V)$ does only involve $\Theta$ if it can be written as a sum of monomials which do only involve $\Theta$.
\end{itemize}

\end{defn}

\begin{rem}

Let the notation be as in Definition~\ref{def:mon}. Note that we allow every scalar as a monomial in Definition~\ref{def:mon}. Consequently, every scalar is also a monomial in $\EuScript{B}(V)$ which does only involve $\Theta$ for any $\Theta\subseteq I$. Conversely, every element in $\EuScript{B}(V)$ which does only involve $\varnothing$ must be a scalar.

\end{rem}

\begin{rem}

Let the notation be as in Definition~\ref{def:mon}. Note that a monomial in $\EuScript{B}(V)$ which starts with $\gamma$ / ends with $\gamma$ / starts with $\Theta$ / ends with $\Theta$ / does only involve $\Theta$ is in particular an element in $\EuScript{B}(V)$ which starts with $\gamma$ / ends with $\gamma$ / starts with $\Theta$ / ends with $\Theta$ / does only involve $\Theta$. In that sense, Definition~\ref{def:mon} is consistent.

\end{rem}

\begin{rem}

Let the notation be as in Definition~\ref{def:mon}. Note that a monomial or element in $\EuScript{B}(V)$ starts with $\gamma$ / ends with $\gamma$ if and only if it starts with $\{\gamma\}$ / ends with $\{\gamma\}$.

\end{rem}

\begin{rem}
\label{rem:starts-ends}

Let the notation be as in Definition~\ref{def:mon}. Note that an element $z\in\EuScript{B}(V)$ starts with $\gamma$ if and only if it can be written as $x_\gamma z'$ for some $z'\in\EuScript{B}(V)$. Similarly, an element $z\in\EuScript{B}(V)$ ends with $\gamma$ if and only if it can be written as $z'x_\gamma$ for some $z'\in\EuScript{B}(V)$.

\end{rem}

\begin{rem}
\label{rem:rho}

Let the notation be as in Definition~\ref{def:mon}. Note that a monomial $M\in\EuScript{B}_W$ starts with $\gamma$ / ends with $\Theta$ / does only involve $\Theta$ if and only if $\rho(M)$ is a monomial in $\EuScript{B}_W$ which ends with $\gamma$ / ends with $\Theta$ / does only involve $\Theta$. Consequently, an element $z\in\EuScript{B}_W$ starts with $\gamma$ / ends with $\Theta$ / does only involve $\Theta$ if and only if $\rho(z)$ is an element in $\EuScript{B}_W$ which ends with $\gamma$ / ends with $\Theta$ / does only involve $\Theta$. This is clear from \cite[Proposition~3.7(2),(6)]{skew} and the definition of $\rho$.

\end{rem}

\begin{rem}

Le the notation be as in Definition~\ref{def:mon}. Note that every monomial in $\EuScript{B}(V)$ is homogeneous with respect to the $\mathbb{Z}_{\geq 0}$-grading. More specifically, let $V$ be a Yetter-Drinfeld $\Gamma$-module and let $(x_\alpha)_{\alpha\in I}$ be a basis of $V$ homogeneous with respect to the $\Gamma$-grading. In the situation of the previous sentence, every monomial in $\EuScript{B}(V)$ is even homogeneous with respect to the $\mathbb{Z}_{\geq 0}$-grading \emph{and} the $\Gamma$-grading.

\end{rem}

\begin{rem}

Whenever we work with $\EuScript{B}_W$ as introduced in and subject to the identifications in Subsection~\ref{subsection:model-Schubert}, we consider the canonical homogeneous basis of $V_W$ with the respect to the $W$-grading given by $(x_\alpha)_{\alpha\in R^+}$, and the terminology in Definition~\ref{def:mon} will always refer to this basis and the index set $I=R^+$.

\end{rem}

\begin{lem}
\label{lem:basic-rev}

Let $\Theta\subseteq R^+$. 

\begin{enumerate}
    \item\label{item:basic0}
    Let $z\in\EuScript{B}_W$ be an element which starts with $\gamma$ for all $\gamma\in\Theta$. Let $\xi\in\EuScript{B}_W$ be an element which ends with $\Theta$. Then, we have $\xi z=0$.
    \item\label{item:basic0'}
    Let $z\in\EuScript{B}_W$ be an element which ends with $\gamma$ for all $\gamma\in\Theta$. Let $\xi\in\EuScript{B}_W$ be an element which starts with $\Theta$. Then, we have $z\xi=0$.
    \item\label{item:basic1}
    Let $z\in\EuScript{B}_W$. We have $(z)\constantoverleftarrow{D_\alpha}=0$ for all $\alpha\in\Theta$ if and only if $(z)\constantoverleftarrow{D_\xi}=0$ for all $\xi\in\EuScript{B}_W$ which start with $\Theta$.
    \item\label{item:basic2}
    Let $z\in\EuScript{B}_W$. We have $\constantoverrightarrow{D_\alpha}(z)=0$ for all $\alpha\in\Theta$ if and only if $\constantoverrightarrow{D_\xi}(z)=0$ for all $\xi\in\EuScript{B}_W$ which end with $\Theta$.
    \item\label{item:basic2-3}
    Let $z_1,\ldots,z_m\in\EuScript{B}_W$ be such that $(z_1)\constantoverleftarrow{D_\alpha}=\cdots=(z_m)\constantoverleftarrow{D_\alpha}=0$ for all $\alpha\in\Theta$. Then, we have $(z_1\cdots z_2)\constantoverleftarrow{D_\xi}=0$ for all $\xi\in\EuScript{B}_W$ which start with $\Theta$.   
    \item\label{item:basic3}
    Let $z\in\EuScript{B}_W$ be an element which does only involve $R^+\setminus\Theta$. Let $\xi\in\EuScript{B}_W$ be an element which starts with $\Theta$. Then, we have $(z)\constantoverleftarrow{D_\xi}=0$.
    \item\label{item:basic4}
    Let $z_1,z_2\in\EuScript{B}_W$ be such that $(z_1)\constantoverleftarrow{D_\alpha}=0$ for all $\alpha\in\Theta$. Then, we have $(z_1z_2)\constantoverleftarrow{D_\xi}=z_1(z_2)\constantoverleftarrow{D_\xi}$ for all $\xi\in\EuScript{B}_W$ which do only involve $\Theta$.
    \item\label{item:basic4'}
    Let $z_1,z_2\in\EuScript{B}_W$ be such that $\constantoverrightarrow{D_\alpha}(z_1)=0$ for all $\alpha\in\Theta$ and such that $z_1$ is homogeneous of $W$-degree $g$. Then, we have $\constantoverrightarrow{D_\xi}(z_1z_2)=z_1\constantoverrightarrow{D_{g^{-1}\xi}}(z_2)$ for all $\xi\in\EuScript{B}_W$ which do only involve $\Theta$.
\end{enumerate}

\end{lem}

\begin{proof}[Proof of Item~\eqref{item:basic0}]

For the proof of the desired vanishing, we can assume that $\xi$ is a monomial which ends with $\Theta$. In that case, we may write $\xi$ as $\xi'x_\alpha$ for some $\alpha\in\Theta$ and some monomial $\xi'\in\EuScript{B}_W$. By assumption, the element $z\in\EuScript{B}_W$ starts with $\alpha$. Hence, it can be written as $x_\alpha z'$ for some $z'\in\EuScript{B}_W$ by Remark~\ref{rem:starts-ends}. We conclude that $\xi z=\xi'x_\alpha^2 z'$ is zero since $x_\alpha^2=0$ by \cite[Example~4.4]{skew}.
\end{proof}

\begin{proof}[Proof of Item~\eqref{item:basic0'}]

For the proof this item, one can argue analogously as in the proof of Item~\eqref{item:basic0}. Alternaftively, one can apply Item~\eqref{item:basic0} to $\rho(z)$ and $\rho(\xi)$ and apply Remark~\ref{rem:rho} and the definition of $\rho$.
\end{proof}

\begin{proof}[Proof of Item~\eqref{item:basic1}]

The implication from right to left is obvious because every $x_\alpha$ where $\alpha\in\Theta$ is a monomial which starts with $\Theta$. For the other implication, it suffices to prove the vanishing if $\xi$ is an arbitrary but fixed monomial which starts with $\Theta$. In that case, we may write $\xi$ as $x_\alpha\xi'$ for some $\alpha\in\Theta$ and some monomial $\xi'\in\EuScript{B}_W$. It follows that
\begin{equation*}
(z)\constantoverleftarrow{D_\xi}=(z)\constantoverleftarrow{D_\alpha}\constantoverleftarrow{D_{\xi'}}=0
\end{equation*}
because $(z)\constantoverleftarrow{D_\alpha}$ vanishes by assumption.
\end{proof}

\begin{proof}[Proof of Item~\eqref{item:basic2}]

For the proof this item, one can argue analogously as in the proof of Item~\eqref{item:basic1}. Alternatively, one can apply Item~\eqref{item:basic1} to $\bar{\EuScript{S}}(z)$ and $\rho(\xi)$. The result follows this way using \cite[Proposition~3.7(6), Proposition~4.2, Remark~4.3]{skew} and Remark~\ref{rem:rho}.
\end{proof}

\begin{proof}[Proof of Item~\eqref{item:basic2-3}]

By Item~\eqref{item:basic1}, we may assume that $\xi=x_\alpha$ for some $\alpha\in\Theta$. The vanishing of $(z_1\cdots z_m)\constantoverleftarrow{D_\alpha}$ then follows from the braided Leibniz rule.
\end{proof}

\begin{proof}[Proof of Item~\eqref{item:basic3}]

For the proof of the desired vanishing, we may assume that $z$ is a monomial which does only involve $R^+\setminus\Theta$ and further, by Item~\eqref{item:basic1}, that $\xi=x_\alpha$ for some $\alpha\in\Theta$. By this assumption, there exist $\lambda\in\mathbf{k}$ and $\alpha_1,\ldots,\alpha_m\in R^+\setminus\Theta$ such that $z=\lambda x_{\alpha_1}\cdots x_{\alpha_m}$. If we apply Item~\eqref{item:basic2-3} to $z_i=x_{\alpha_i}$ for all $1\leq i\leq m$ and to $\xi$, the result follows.
\end{proof}

\begin{proof}[Proof of Item~\eqref{item:basic4},\eqref{item:basic4'}]

For the proof of these items, we may assume that $\xi$ is a monomial which does only involve $\Theta$. By induction on the $\mathbb{Z}_{\geq 0}$-degree of the monomial $\xi$, we can further assume that its $\mathbb{Z}_{\geq 0}$-degree is one. But in that case, the desired equations follow again from the braided Leibniz rule.
\end{proof}

\begin{lem}
\label{lem:zdeg-length}

Let $\xi\in\EuScript{B}_W$ be a nonzero homogeneous element of some $\mathbb{Z}_{\geq 0}$-degree $m$ and some $W$-degree $w$. Then, the parity of $\ell(w)$ equals the parity of $m$.

\end{lem}

\begin{proof}

To prove the lemma, we can clearly assume that $\xi$ is a nonzero monomial of some $\mathbb{Z}_{\geq 0}$-degree $m$ and some $W$-degree $w$. If we assume this from now on, we can find $\lambda\in\mathbf{k}^\times$ and $\alpha_1,\ldots,\alpha_m\in R^+$ such that $\xi=\lambda x_{\alpha_1}\cdots x_{\alpha_m}$. Then, we necessarily have $w=s_{\alpha_1}\cdots s_{\alpha_m}$. Every reflection has odd length. Thus, by the deletion condition \cite[Corollary~5.8]{humphreys-coxeter}, we see that $\ell(w)$ has the same parity as $m$. 
\end{proof}

\section{\for{toc}{{\color{white}0}}NilCoxeter algebra}
\label{sec:nilcoxeter}

The nilCoxeter algebra was first introduced in \cite{nilCoxeter} and studied in the context of the symmetric group. Schubert polynomials and Stanley symmetric functions are treated in \cite{nilCoxeter} as coefficients of certain expressions in the nilCoxeter algebra. Independently from this, we adopt here the point of view developed in \cite{bazlov1} where the nilCoxeter algebra of $W$ is canonically embedded into $\EuScript{B}_W$ for any $W$.

Let $\EuScript{N}_W$ be the subalgebra of $\EuScript{B}_W$ generated by $x_\beta$ where $\beta$ runs through $\Delta$. As it was proved in \cite[Theorem~6.3(i),(ii)]{bazlov1}, these generators of $\EuScript{N}_W$ satisfy the so-called nilCoxeter relations, i.e.\ the homogeneous relations $x_\beta^2=0$ for all $\beta\in\Delta$ and
\[
\underset{m(s_\beta,s_{\beta'})\text{ factors on each side}}{\underbrace{x_\beta x_{\beta'}x_\beta\cdots}=\underbrace{x_{\beta'}x_\beta x_{\beta'}\cdots}}
\]
for all $\beta,\beta'\in\Delta$, and all other relations between them are generated from those. Therefore, the algebra $\EuScript{N}_W$ is called the nilCoxeter algebra of $W$ in the sense of \cite[Section~2]{nilCoxeter}. Let $s_{\beta_1}\cdots s_{\beta_m}$ be a reduced expression of some $w\in W$. Then, we define $x_w=x_{\beta_1}\cdots x_{\beta_m}$. This element is well-defined, i.e.\ independent of the choice of the reduced expression of $w$, because of the nilCoxeter relations and because any two reduced expressions of $w$ can be connected by a sequence of braid moves by the Word Property \cite[Theorem~3.3.1(ii)]{bjorner}. With this definition, the family $(x_w)_w$ becomes a basis of $\EuScript{N}_W$. We call this basis the standard basis of $\EuScript{N}_W$. We finally remark that by definition $x_1=1$ is a member of the standard basis of $\EuScript{N}_W$.

Following Liu's coproduct approach \cite[Proposition~2.7]{liu}, see also \cite[Definition~5.1]{skew}, we introduce for all $v,w\in W$ uniquely determined elements $x_{w/v}\in\EuScript{B}_W$ such that $\Delta(x_w)=\sum_vx_{w/v}\otimes x_v$. By definition, we then have $x_{w/v}=0$ for all $v\not\leq w$, and further $x_{w/1}=x_w$ and $x_{w/w}=1$ for all $w\in W$. In line with Remark~\ref{rem:shortcuts}, we use further shortcuts, namely, we write $\constantoverleftarrow{D_w}=\constantoverleftarrow{D_{x_w}}$ and $\constantoverleftarrow{D_{w/v}}=\constantoverleftarrow{D_{x_{w/v}}}$ for all $v,w\in W$.

\begin{fact}
\label{fact:rho_amel}

Let $u,v,w\in W$. Let $y=wx_v$.

\begin{enumerate}
    \item\label{item:deg_y}
    The element $y$ is nonzero and homogeneous of $\mathbb{Z}_{\geq 0}$-degree $\ell(v)$ and of $W$-degree $wvw^{-1}$. In particular, if $v=w_o$ and if $w$ is an element in the centralizer of $w_o$, then the element $y$ is nonzero and homogeneous of $\mathbb{Z}_{\geq 0}$-degree $\ell(w_o)$ and of $W$-degree $w_o$.
    \item\label{item:mult_wo}
    We have $w_ox_v=(-1)^{\ell(v)}x_{w_ovw_o}$. If $w$ is an element in the centralizer of $w_o$, this means that $w_oy=(-1)^{\ell(v)}wx_{w_ovw_o}$. If $v$ and $w$ are both elements in the centralizer of $w_o$, we see in particular that $w_oy=(-1)^{\ell(v)}y$. If $v=w_o$ and if $w$ is an element in the centralizer of $w_o$, we see further that $w_oy=(-1)^{\ell(w_o)}y$.
    \item\label{item:nil_pair}
    We have $\left<x_u,x_v\right>=\delta_{u,v^{-1}}$. If $v$ is an involution, e.g.\ if $v=w_o$, we have in particular that $\left<y,y\right>=1$.
    \item\label{item:rho_invert}
    We have $\rho(x_v)=x_{v^{-1}}$. If $v$ is an involution, e.g.\ if $v=w_o$, it follows in particular that $\rho(y)=y$.
\end{enumerate}

\end{fact}

\begin{rem}

The proof of Item~\eqref{item:deg_y} and its content are plain. We will use Item~\eqref{item:deg_y} without reference from now on.

\end{rem}

\begin{rem}

Note that it follows from Item~\eqref{item:nil_pair} that the family $(x_w)_w$ is a basis of $\EuScript{N}_W$.

\end{rem}

\begin{rem}

If $v=w_o$ and if $w$ is an element in the centralizer of $w_o$, then Item~\eqref{item:nil_pair} says in particular that $\left<y,y\right>=1$ which is precisely the meaning of Corollary~\ref{cor:bracket} applied to a disjoint system of order one. In this sense, there is some overlap between Item~\eqref{item:nil_pair} and Corollary~\ref{cor:bracket}.

\end{rem}

\begin{proof}[Proof of Item~\eqref{item:mult_wo}]

In the proof of this item, we repeatedly use the fact that $w_o$ is an involution. We only have to prove the first formula in the statement of Item~\eqref{item:mult_wo}. The special cases are immediate from it. To this end, recall that $\ell(v)=\ell(w_ovw_o)$ because of \cite[Section~1.8, Equation~(2), Section~5.2, Equation~(L1)]{humphreys-coxeter}, so that $s_{\beta_1}\cdots s_{\beta_m}$ is a reduced expression of $v$ if and only if $s_{w_o(\beta_1)}\cdots s_{w_o(\beta_m)}$ is a reduced expression of $w_ovw_o$ by \cite[Section~5.6, Exercise~2]{humphreys-coxeter}. In the transition from $w_ox_v$ to $x_{w_ovw_o}$, we therefore pick up a sign for each of the $\ell(v)$ factors of $x_v$ because $x_{-\alpha}=-x_\alpha$ for all $\alpha\in R$. These observations prove the desired formula.
\end{proof}

\begin{proof}[Proof of Item~\eqref{item:nil_pair}]

We only have to prove the first formula in the statement of Item~\eqref{item:nil_pair}. The special case follows because $\left<-,-\right>$ is a morphism in the Yetter-Drinfeld category over $W$. But the first formula was already proved in \cite[Proposition~6.4]{skew}.
\end{proof}

\begin{proof}[Proof of Item~\eqref{item:rho_invert}]

It is clear that $s_{\beta_1}\cdots s_{\beta_m}$ is a reduced expression of $v$ if and only if $s_{\beta_m}\cdots s_{\beta_1}$ is a reduced expression of $v^{-1}$. The formula $\rho(x_v)=x_{v^{-1}}$ follows from this. The special case follows then by \cite[Proposition~3.7(1)]{skew}.
\end{proof}

\begin{lem}
\label{lem:y_van_amel}

Let $w\in W$. Let $y=wx_{w_o}$. The element $y\in\EuScript{B}_W$ is a monomial which starts with $\gamma$ for all $\gamma\in T_w$. The element $y\in\EuScript{B}_W$ is also a monomial which ends with $\gamma$ for all $\gamma\in T_w$ and a monomial which does only involve $T_w$.

\end{lem}

\begin{proof}

Let the notation be as in the statement. The statement in the last sentence is clear from the statement in the second last sentence because of Remark~\ref{rem:rho} and Fact~\ref{fact:rho_amel}\eqref{item:rho_invert}. We prove the statement in the second last sentence now. Let $\gamma\in T_w$ be fixed but arbitrary. Then, there exists a unique $\beta\in\Delta$ and a unique sign $\epsilon$ such that $w(\beta)=\epsilon\gamma$. We may write $x_{w_o}=x_\beta x_{s_\beta w_o}$ and thus $y=\epsilon x_\gamma(wx_{s_\beta w_o})$. In this form, the monomial $y$ is visibly a monomial which starts with $\gamma$.
\end{proof}

\begin{cor}
\label{cor:y_van_amel}

Let $w\in W$. Let $y=wx_{w_o}$. Then, we have $zy=0$ for every element $z\in\EuScript{B}_W$ which ends with $T_w$ and $yz=0$ for every element $z\in\EuScript{B}_W$ which starts with $T_w$.

\end{cor}

\begin{proof}

This is clear from Lemma~\ref{lem:basic-rev}\eqref{item:basic0},\eqref{item:basic0'} and Lemma~\ref{lem:y_van_amel}.
\end{proof}

\begin{lem}
\label{lem:Sstartswith}

Let $v,w\in W$ be such that $v\neq w_o$, then $wx_{w_o/v}$ is a monomial which starts with $T_w$.

\end{lem}

\begin{proof}

Let $v,w\in W$ be such that $v\neq w_o$. By \cite[Example~5.4, Theorem~6.11]{skew} and because $w_o$ is an involution, we know that $wx_{w_o/v}=w\bar{\EuScript{S}}(x_{vw_o})$. Since $v\neq w_o$, we can find $\beta\in\Delta$ such that $s_\beta vw_o<vw_o$. Further, there exists a unique $\gamma\in T_w$ and a unique sign $\epsilon$ such that $w(\beta)=\epsilon\gamma$. By \cite[Proposition~3.7(4)]{skew} and by what was said, we see that $wx_{w_o/v}=\epsilon x_\gamma (s_\gamma w\bar{\EuScript{S}}(x_{s_\beta vw_o}))$. From this equality and \cite[Remark~3.8]{skew}, we see that $wx_{w_o/v}$ is indeed a monomial which starts with $\gamma$, and this $\gamma$ lies in $T_w$ by definition.
\end{proof}

\section{\for{toc}{{\color{white}0}}Inversion of braided partial derivatives}
\label{sec:inversion}

By inversion of braided partial derivatives, we mean the third formula in Proposition~\ref{prop:rhoD} which shows how $\rho$ and $\constantoverleftarrow{D_\xi}$ where $\xi\in\EuScript{B}_W$ commute when considered as endomorphisms of $\EuScript{B}_W$ acting from the right. Such a formula has relevance for the vanishing in Corollary~\ref{cor:vanish_rho} which is used in the proof of Lemma~\ref{lem:left-to-right-hypo}. Further, in this section, we work out a variant of the Nichols-Zoeller theorem for $\EuScript{B}_W$, namely Corollary~\ref{cor:nichols-zoeller2}, based on Corollary~\ref{cor:nichols-zoeller1}, where the latter result is used in the proof of Corollary~\ref{cor:BWintegral}\eqref{item:sign_elab2}.

\begin{prop}
\label{prop:rhoD}

Let $\xi\in\EuScript{B}_W$ be homogeneous of $W$-degree $g$. Then, we have
\begin{align*}
\EuScript{S}\constantoverleftarrow{D}_\xi&=g^{-1}\EuScript{S}\constantoverrightarrow{D_{\EuScript{S}(\xi)}}\,,\\
\constantoverrightarrow{D_\xi}\EuScript{S}^{-1}&=\constantoverleftarrow{D_{\EuScript{S}^{-1}(\xi)}}\EuScript{S}^{-1}g^{-1}\,,\\
\rho\constantoverleftarrow{D_\xi}&=\constantoverleftarrow{D_{\bar{\EuScript{S}}(\xi)}}\rho g\,.
\end{align*}

\end{prop}

\begin{proof}

Let us recall the equivalent formulas
\begin{gather*}
\Delta\EuScript{S}=(\EuScript{S}\otimes\EuScript{S})\Psi\Delta\,,\\
\Psi^{-1}(\EuScript{S}^{-1}\otimes\EuScript{S}^{-1})\Delta=\Delta\EuScript{S}^{-1}
\end{gather*}
from \cite[Equation~(9.39) on page~477]{majid} where we suppress the components of the braiding $\Psi$. The first of those formulas can be put in words by saying that the antipode in a braided Hopf algebra is a braided anti coalgebra homomorphism. If we use suggestive Sweedler notation and plug in the braiding of ${}_W^W\EuScript{YD}$, we can equally well write these formulas as
\begin{align*}
\EuScript{S}(z)_{(1)}\otimes\EuScript{S}(z)_{(2)}&=g_{(1)}\EuScript{S}(z_{(2)})\otimes\EuScript{S}(z_{(1)})\,,\\
\EuScript{S}^{-1}(z)_{(1)}\otimes\EuScript{S}^{-1}(z)_{(2)}&=\EuScript{S}^{-1}(z_{(2)})\otimes g_{(2)}^{-1}\EuScript{S}^{-1}(z_{(1)})
\end{align*}
where $z$ is an arbitrary element in $\EuScript{B}_W$ and where $g_{(1)},g_{(2)}$ denote the $W$-degree of $z_{(1)},z_{(2)}$, respectively.

With the help of the previous two formulas and the basic properties of $\left<-,-\right>$, we compute
\begin{align*}
(\EuScript{S}(z))\constantoverleftarrow{D_\xi}=g_{(1)}\EuScript{S}(z_{(2)})\left<\EuScript{S}(z_{(1)}),\xi\right>&=\left<\EuScript{S}(\xi),z_{(1)}\right>g^{-1}\EuScript{S}(z_{(2)})=g^{-1}\EuScript{S}\constantoverrightarrow{D_{\EuScript{S}(\xi)}}(z)\,,\\
\constantoverrightarrow{D_\xi}\EuScript{S}^{-1}(z)=\left<\xi,\EuScript{S}^{-1}(z_{(2)})\right>g_{(2)}^{-1}\EuScript{S}^{-1}(z_{(1)})&=g\EuScript{S}^{-1}(z_{(1)})\left<z_{(2)},\EuScript{S}^{-1}(\xi)\right>=(z)\constantoverleftarrow{D_{\EuScript{S}^{-1}(\xi)}}\EuScript{S}^{-1}g^{-1}\,.
\end{align*}
This proves the first two formulas in the statement.

To prove the third formula, we can restrict the operators in the formula to a graded component $\EuScript{B}_W^m$ of some $\mathbb{Z}_{\geq 0}$-degree $m$. Once the equality is proved for arbitrary but fixed $m$, it will be valid everywhere on $\EuScript{B}_W$. Let $\epsilon=(-1)^m$. Because $\rho$ and $\bar{\EuScript{S}}$ are involutions by \cite[Proposition~3.7(6)]{skew}, we see from its definition that $\bar{\EuScript{S}}$ restricted to $\EuScript{B}_W^m$ is given by $\epsilon\rho\EuScript{S}=\epsilon\EuScript{S}^{-1}\rho$. With the help of \cite[Proposition~3.7(1), Proposition~4.2, Remark~2.16, Remark~4.3]{skew}, this last formula for $\bar{\EuScript{S}}$ restricted to $\EuScript{B}_W^m$ and the second formula in the statement which was already justified, we compute
\[
\rho\constantoverleftarrow{D_\xi}=\EuScript{S}^{-1}\bar{\EuScript{S}}\constantoverleftarrow{D_\xi}\epsilon=\epsilon\bar{\EuScript{S}}\constantoverrightarrow{D_{\rho(\xi)}}\EuScript{S}^{-1}=\constantoverleftarrow{D_{\EuScript{S}^{-1}\rho(\xi)}}\EuScript{S}^{-1}\bar{\EuScript{S}}g\epsilon=\constantoverleftarrow{D_{\bar{\EuScript{S}}(\xi)}}\rho g
\]
where this computation takes place restricted to $\EuScript{B}_W^m$. By what said before, this completes the proof.
\end{proof}

\begin{cor}
\label{cor:vanish_rho}

Let $b\in\EuScript{B_W}$. For all $\alpha\in R^+$, we have
\[
(b)\constantoverleftarrow{D_\alpha}=0\quad\text{if and only if}\quad(\rho(b))\constantoverleftarrow{D_\alpha}=0\,.
\]
\end{cor}

\begin{proof}
The third formula in Proposition~\ref{prop:rhoD} applied to $\xi=x_\alpha$ for some $\alpha\in R^+$ becomes
\[
\rho\constantoverleftarrow{D_\alpha}=\constantoverleftarrow{D_\alpha}\rho s_\alpha
\]
because $\bar{\EuScript{S}}$ is the identity in $\mathbb{Z}_{\geq 0}$-degree one by definition. The claimed equivalence is immediate from this formula and the fact that the operator $\rho$ is invertible by \cite[Proposition~3.7(6)]{skew}.
\end{proof}

\subsection*{A variant of the Nichols-Zoeller theorem}

\begin{thm}
\label{thm:nichols-zoeller}

Let $\xi\in\EuScript{B}_W$ be a homogeneous element of $W$-degree $g$. Then, we have $\EuScript{S}^{-1}(\xi)=(-1)^{\ell(g)}g^{-1}\EuScript{S}(\xi)$.

\end{thm}

\begin{proof}

For the proof of the desired formula, we may assume that $\xi$ is a monomial of $W$-degree $g$. We assume this from now on and proceed by induction on the $\mathbb{Z}_{\geq 0}$-degree of $\xi$. The case of $\mathbb{Z}_{\geq 0}$-degree equal to zero, i.e.\ $\xi$ being a scalar multiple of $1$, being obvious, we assume further that the $\mathbb{Z}_{\geq 0}$-degree of $\xi$ is $>0$. In that case, we may write $\xi$ as $x_\alpha\xi'$ for some $\alpha\in R^+$ and some monomial $\xi'\in\EuScript{B}_W$. With the help of \cite[Equation~(9.39) on page~477]{majid} and the induction hypothesis, we compute
\begin{align*}
\EuScript{S}^{-1}(\xi)&=\EuScript{S}^{-1}(x_\alpha\xi')\\
&=\EuScript{S}^{-1}(\xi')(g^{-1}s_\alpha\EuScript{S}^{-1}(x_\alpha))\\
&=(-1)^{\ell(g)+1}(g^{-1}s_\alpha\EuScript{S}(\xi'))(g^{-1}s_\alpha\EuScript{S}(x_\alpha))\\
&=(-1)^{\ell(g)}g^{-1}((s_\alpha\EuScript{S}(\xi'))\EuScript{S}(x_\alpha)\\
&=(-1)^{\ell(g)}g^{-1}\EuScript{S}(x_\alpha\xi')\\
&=(-1)^{\ell(g)}g^{-1}\EuScript{S}(\xi)\qedhere
\end{align*}
\end{proof}

\begin{cor}
\label{cor:nichols-zoeller1}

Let $\xi\in\EuScript{B}_W$ be a homogeneous element of $W$-degree $g$. Then, we have $\EuScript{S}^2(\xi)=(-1)^{\ell(g)}g\xi$.

\end{cor}

\begin{proof}

This is an immediate corollary of Theorem~\ref{thm:nichols-zoeller}.
\end{proof}

\makeatletter
\xpatchcmd{\@thm}{\thm@headpunct{.}}{\thm@headpunct{}}{}{}
\makeatother
\begin{cor}
\label{cor:nichols-zoeller2}

{\normalfont (A variant of the Nichols-Zoeller theorem \cite[Theorem~10.5.6, Corollary 10.5.7(a)]{radford})\textbf{.}} Let $e$ be the exponent of $W$, i.e.\ the least common multiple of the orders of all elements of $W$ as in \cite{exponent}. Then, we have $\EuScript{S}^{2e}=1$.

\end{cor}
\makeatletter
\xpatchcmd{\@thm}{\thm@headpunct{}}{\thm@headpunct{.}}{}{}
\makeatother

\begin{proof}

Indeed, it follows from Corollary~\ref{cor:nichols-zoeller1} that $\EuScript{S}^{2e}(\xi)=(-1)^{e\cdot\ell(g)}g^e\xi$ for all homogeneous elements $\xi\in\EuScript{B}_W$ of $W$-degree $g$. By definition of $e$, we know that $e$ is even, and further that $e$ is the smallest positive integer $N$ such that $g^N=1$ for all $g\in W$. It follows that $(-1)^{e\cdot\ell(g)}g^e=1$ for all $g\in W$, and consequently $\EuScript{S}^{2e}=1$ -- as claimed.
\end{proof}

\begin{rem}

In~\cite{exponent},  the conjecture was put forward that the exponent of $W$, i.e.\ the least common multiple of the orders of all elements of $W$, as it appears in the statement of Corollary~\ref{cor:nichols-zoeller2}, equals the least common multiple of the degrees of $W$ (-- degrees of $W$ in the sense of \cite[Section~3.7]{humphreys-coxeter}). This conjecture can be easily verified in type $\mathsf{A}$ using the cycle decomposition of permutations. However, directly generalizing this proof using the generalized cycle decomposition as in~\cite[Thoerem~1.3]{gobetpub} fails as was shown in~\cite{gobet}. It would be interesting to see a type independent proof of the conjecture mentioned in this remark for all finite Coxeter groups. 

\end{rem}

\section{\for{toc}{{\color{white}0}}Consequences of the generalized braided Leibniz rule}
\label{sec:braided-leibniz}

The general braided Leibniz rule appears in the context of the symmetric group, divided difference operators and skew divided difference operators in the papers \cite{kirillov,liu,macdonald1991notes}, for example. We adopt here the point of view developed in \cite{skew} where a generalized braided Leibniz rule was introduced for braided partial derivatives acting on $\EuScript{B}_W$. We give a selective list of consequences of this rule.

\makeatletter
\xpatchcmd{\@thm}{\thm@headpunct{.}}{\thm@headpunct{}}{}{}
\makeatother
\begin{thm}
\label{thm:leibniz}

{\normalfont (Generalized braided Leibniz rule \cite[Theorem~5.14]{skew}, \cite[Proposition~3(ii)]{kirillov}, \cite[Chapter~2, Theorem~2.18]{macdonald1991notes})\textbf{.}} Let $v,w,w'\in W$. Let $z\in\EuScript{B}_W$. Then, we have
\[
z\constantoverleftarrow{D_{w'x_{w/v}}}=\sum_{v\leq u\leq w}\constantoverleftarrow{D_{w'x_{u/v}}}\big((z)\constantoverleftarrow{D_{w'x_{w/u}}}w'uv^{-1}w'^{-1}\big)
\]
where we can equally well take the sum over all $u\in W$.

\end{thm}
\makeatletter
\xpatchcmd{\@thm}{\thm@headpunct{}}{\thm@headpunct{.}}{}{}
\makeatother

\begin{proof}

Note that $x_{w/v}=0$ for all $v\not\leq w$ by definition. This shows that we can equally well sum over all $u\in W$ in the formula in the statement and everywhere else (in this proof) where similar situations arise. Let $v,w,w'\in W$ be arbitrary. Suppose that the claimed formula is proved for $w'=1$, then we find for arbitrary $w'$ that
\begin{alignat*}{99}
&\textstyle{\sum_u}\constantoverleftarrow{D_{w'x_{u/v}}}\big((z)\constantoverleftarrow{D_{w'x_{w/u}}}w'uv^{-1}w'^{-1}\big)&&{}={}\textstyle{\sum_u}\constantoverleftarrow{D_{w'x_{u/v}}}\big((zw')\constantoverleftarrow{D_{x_{w/u}}}uv^{-1}w'^{-1}\big)\\
{}={}&\textstyle{\sum_u} w'\constantoverleftarrow{D_{x_{u/v}}}\big((zw')\constantoverleftarrow{D_{x_{w/u}}}uv^{-1}\big)w'^{-1}&&{}={}w'(zw')\constantoverleftarrow{D_{x_{w/v}}}w'^{-1}=z\constantoverleftarrow{D_{w'x_{w/v}}}\raisebox{.5ex}{\footnotemark}
\end{alignat*}\footnotetext{The expression $(zw')\constantoverleftarrow{D_{x_{w/v}}}$ before the last equality in this equation has to be understood as the composition of two endomorphisms of $\EuScript{B}_W$ acting from the right, namely the composition of multiplication from the right with $zw'=w'^{-1}z$ and $\constantoverleftarrow{D_{x_{w/v}}}$, and should not be confused with the endomorphism given by right multiplication with $(zw')\constantoverleftarrow{D_{x_{w/v}}}$. Except maybe in the proof of Corollary~\ref{cor:bcommuteswithskew} where we explicitly make it clear in writing, this is the only instance where the conventions in Remark~\ref{rem:ident-into-end} lead to ambiguity. Everywhere else the absence or presence of parenthesis and arrows makes the meaning clear.}\noindent where we used Remark~\ref{rem:ident-into-end}, \cite[Remark~3.16]{skew} and the assumption for $w'=1$. This means that it suffices to prove the claimed formula for $w'=1$. We now do so.

Let $x,y\in\EuScript{B}_W$. Evaluating once $(x(yz))\constantoverleftarrow{D_w}$ and twice $((xy)z)\constantoverleftarrow{D_w}$ with the help of \cite[Theorem~5.14]{skew}, we find the equality
\[
\sum_{v'}(x)\constantoverleftarrow{D_{v'}}\big((yz)\constantoverleftarrow{D_{w/v'}}v'\big)=
\sum_{u,v'}(x)\constantoverleftarrow{D_{v'}}\big((y)\constantoverleftarrow{D_{u/v'}}v'\big)\big((z)\constantoverleftarrow{D_{w/u}}u\big)\,.
\]
If we plug $x=x_{v^{-1}}$ into this equality, it becomes
\begin{equation}
\label{eq:induction-skew}
\sum_{v'}x_{v^{-1}/v'^{-1}}\big((yz)\constantoverleftarrow{D_{w/v'}}v'\big)=
\sum_{u,v'}x_{v^{-1}/v'^{-1}}\big((y)\constantoverleftarrow{D_{u/v'}}v'\big)\big((z)\constantoverleftarrow{D_{w/u}}u\big)\,.
\end{equation}
in view of \cite[Proposition~8.1]{skew}. We now prove by induction on $\ell(v)$ that
\[
(yz)\constantoverleftarrow{D_{w/v}}v=\sum_u((y)\constantoverleftarrow{D_{u/v}}v\big)\big((z)\constantoverleftarrow{D_{w/u}}u\big)
\]
which suffices to finish the proof of the theorem because we can multiply with $v^{-1}$ from the right. If $v=1$, the desired formula is simply \cite[Theorem~5.14]{skew} by \cite[Example~5.4]{skew}. Suppose that $\ell(v)>0$ and that the induction hypothesis is satisfied for all $v'$ of length $<\ell(v)$. With the help of Equation~\eqref{eq:induction-skew} and \cite[Example~5.4]{skew} we find that
\begin{multline*}
(yz)\constantoverleftarrow{D_{w/v}}v+\sum_{v'<v}x_{v^{-1}/v'^{-1}}\big((yz)\constantoverleftarrow{D_{w/v'}}v'\big)=\\
\sum_u\big((y)\constantoverleftarrow{D_{u/v}}v\big)\big((z)\constantoverleftarrow{D_{w/u}}u\big)+\sum_{\substack{v'\leq u\leq w\\v'<v}}x_{v^{-1}/v'^{-1}}\big((y)\constantoverleftarrow{D_{u/v'}}v'\big)\big((z)\constantoverleftarrow{D_{w/u}}u\big)
\end{multline*}
where we of course use the definitional vanishing mentioned in the first sentence of this proof. The second sums after the plus on each side of the previous equation are equal by induction hypothesis. If we subtract them, we are left with the desired formula for $v$. This completes the induction step and the proof of the theorem.
\end{proof}

\begin{cor}
\label{cor:tow_inv}

Let $v,w\in W$. Let $y=wx_{w_o}$. Then, we have
\[
\constantoverleftarrow{D_y}z\constantoverleftarrow{D_{wx_v}}=\constantoverleftarrow{D_y}\big((z)\constantoverleftarrow{D_{wx_v}}\big)
\]
for all $z\in\EuScript{B}_W$.

\end{cor}

\begin{proof}

Let the notation be as in the statement. Let $z\in\EuScript{B}_W$ be fixed but arbitrary. By Theorem~\ref{thm:leibniz}, we have
\[
\constantoverleftarrow{D_y}z\constantoverleftarrow{D_{wx_v}}=\sum_u\constantoverleftarrow{D_{y(wx_u)}}\big((z)\constantoverleftarrow{D_{wx_{v/u}}}wuw^{-1}\big)\,.
\]
Since $wx_u$ is a monomial which starts with $T_w$ whenever $u\neq 1$, we see from Corollary~\ref{cor:y_van_amel} that $y(wx_u)$ vanishes for all $u\neq 1$. In view of \cite[Example~5.4]{skew}, the above sum therefore reduces to the right side of the claimed formula in the statement of the corollary.
\end{proof}

\begin{cor}
\label{cor:bcommuteswithskew}

Let $v,w\in W$. Let $b\in\EuScript{B}_W$ be such that $(b)\constantoverleftarrow{D_\alpha}=0$ for all $\alpha\in T_w$. Then, we have
\[
b\constantoverleftarrow{D_{wx_{w_o/v}}}=\constantoverleftarrow{D_{wx_{w_o/v}}}(wvw_ow^{-1}b)\,.
\]

\end{cor}

\begin{proof}

Let the notation be as in the statement. By Lemma~\ref{lem:basic-rev}\eqref{item:basic1}, Lemma~\ref{lem:Sstartswith} and \cite[Example~5.4]{skew}, we see that $(b)\constantoverleftarrow{D_{wx_{w_o/u}}}$ vanishes for all $u\neq w_o$ and evaluates as $b$ for $u=w_o$. This means that if we apply Theorem~\ref{thm:leibniz} to the left side of the claimed formula in the statement of the corollary, in the corresponding sum over $u$ only one term associated to $u=w_o$ survives and this term is equal to the right side of the claimed formula in the statement of the corollary. Note also that $bww_ov^{-1}w^{-1}=wvw_ow^{-1}b$ considered as equality of elements in $\EuScript{B}_W$ because $w_o$ is an involution. This completes the proof.
\end{proof}

\begin{cor}
\label{cor:ofbskew}

Let $D$ be a disjoint system of order two. Let $w_1,w_2$ be some ordering of the elements of $D$. Let $y_1=w_1x_{w_o}$ and let $y_2=w_2x_{w_o}$. Then, we have $y_1\constantoverleftarrow{D_{y_2}}=\constantoverleftarrow{D_{y_2}}y_1(-1)^{\ell(w_o)}$.

\end{cor}

\begin{proof} 

Let the notation be as in the statement. By Lemma~\ref{lem:y_van_amel}, the monomial $y_1$ does only involve $T_{w_1}\subseteq R^+\setminus T_{w_2}$. Thus, by Lemma~\ref{lem:basic-rev}\eqref{item:basic3}, we know that $(y_1)\constantoverleftarrow{D_\alpha}=0$ for all $\alpha\in T_{w_2}$. Therefore, Corollary~\ref{cor:bcommuteswithskew} applies to $v=1,w_2,b=y_1$ and we obtain the claimed formula because of Fact~\ref{fact:rho_amel}\eqref{item:mult_wo}.
\end{proof}

\begin{cor}
\label{cor:bracket}

Let $D$ be a disjoint system of order $r$. Let $w_1,\ldots,w_r$ be some ordering of the elements of $D$. Let $y_i=w_ix_{w_o}$ for all $1\leq i\leq r$. Then, we have
\[
\left<y_{\pi(1)}\cdots y_{\pi(r)},y_{\sigma(1)}\cdots y_{\sigma(r)}\right>=(-1)^{\left(\frac{(r-1)r}{2}+\ell(\sigma\pi^{-1})\right)\cdot\ell(w_o)}
\]
for all permutations $\sigma,\pi\in\mathbb{S}_r$. In particular, it follows that $y_{\sigma(1)}\cdots y_{\sigma(r)}$ is nonzero for all permutations $\sigma\in\mathbb{S}_r$.

\end{cor}

\begin{proof}

Let the notation be as in the statement. The particular case is obvious from the claimed equation. We prove this equation. Let $\sigma,\pi\in\mathbb{S}_r$ be arbitrary. If we change the indices of the ordering of $D$ from $1,\ldots,r$ to $\pi(1),\ldots,\pi(r)$ and replace $\sigma$ by $\sigma\pi^{-1}$, we see that we can assume directly in the beginning that $\pi=1$. We will assume $\pi=1$ from now on.

Let us define indices $i_1^j,\ldots,i_{r-j}^j$ such that
\[
\{i_1^j<\cdots<i_{r-j}^j\}=\{1,\ldots,r\}\setminus\{\sigma(1),\ldots,\sigma(j)\}
\]
for all $0\leq j\leq r-1$. In order to prove the corollary, it suffices to prove the formula
\begin{equation}
\label{eq:bracket}
\big<y_{i_1^j}\cdots y_{i_{r-j}^j},y_{\sigma(j+1)}\cdots y_{\sigma(r)}\big>=(-1)^{n_{j+1}\cdot\ell(w_o)}
\big<y_{i_1^{j+1}}\cdots y_{i_{r-j-1}^{j+1}},y_{\sigma(j+2)}\cdots y_{\sigma(r)}\big>
\end{equation}
where
\[
n_{j+1}=r-\sigma(j+1)-\#\{1\leq i<j+1\mid\sigma(i)>\sigma(j+1)\}
\]
and where $0\leq j\leq r-1$. Indeed, once Equation~\eqref{eq:bracket} is established, we apply it iteratively for all $0\leq j\leq r-1$ and we find that the bracket in the statement of the corollary is equal to minus one to the power of $\smash{\sum_{j=1}^r n_j}$ times $\ell(w_o)$. But the parity of the previous sum is precisely the parity of $\smash{\frac{(r-1)r}{2}}+\ell(\sigma)$. In this way, we see that it indeed suffices to prove Equation~\eqref{eq:bracket}.

To prove Equation~\eqref{eq:bracket}, let us fix some arbitrary $0\leq j\leq r-1$. Let $1\leq k\leq r-j$ be such that $i_k^j=\sigma(j+1)$. By definition, we know that $n_{j+1}=r-j-k$. From this and repeated application of Corollary~\ref{cor:ofbskew} to suitable subsets of $D$ with two elements (cf.~Remark~\ref{rem:sub-disj}), we see that
\[
\big(y_{i_1^j}\cdots y_{i_{r-j}^j}\big)\constantoverleftarrow{D_{y_{\sigma(j+1)}}}=
(-1)^{n_{j+1}\cdot\ell(w_o)}\big(y_{i_1^j}\cdots y_{i_k^j}\big)\constantoverleftarrow{D_{y_{\sigma(j+1)}}}y_{i_{k+1}^j}\cdots y_{i_{r-j}^j}\,.
\]
If we use additionally Lemma~\ref{lem:basic-rev}\eqref{item:basic3},\eqref{item:basic4},~\ref{lem:y_van_amel} and Fact~\ref{fact:rho_amel}\eqref{item:nil_pair}, we see that the above quantity equals
\[
(-1)^{n_{j+1}\cdot\ell(w_o)}y_{i_1^{j+1}}\cdots y_{i_{r-j-1}^{j+1}}\,.
\]
But this shows everything we have to show in order to justify Equation~\eqref{eq:bracket} and completes the proof of the corollary.
\end{proof}

\begin{cor}
\label{cor:prep-abstr-comm}

Let $w\in W$. Let $y=wx_{w_o}$. Let $x_1,x_2,x_3\in\EuScript{B}_W$ be such that 
\[
(x_1)\constantoverleftarrow{D_\alpha}=(x_2)\constantoverleftarrow{D_\alpha}=(ww_ow^{-1}x_2)\constantoverleftarrow{D_\alpha}=(x_3)\constantoverleftarrow{D_\alpha}=0
\]
for all $\alpha\in T_w$. Then, we have
\[
(x_1yx_2x_3)\constantoverleftarrow{D_y}=(x_1(ww_ow^{-1}x_2)yx_3)\constantoverleftarrow{D_y}=x_1(ww_ow^{-1}(x_2x_3))\,.
\]

\end{cor}

\begin{proof}

The claimed equalities follow directly from Fact~\ref{fact:rho_amel}, Lemma~\ref{lem:basic-rev}\eqref{item:basic2-3},\eqref{item:basic4}, Corollary~\ref{cor:bcommuteswithskew} and the assumptions.
\end{proof}

\begin{cor}
\label{cor:prep-abstr-comm2}

Let $w$ be an element in the centralizer of $w_o$. Let $y=wx_{w_o}$. Let $x_1,x_2,x_3\in\EuScript{B}_W$ be such that $(x_1)\constantoverleftarrow{D_\alpha}=(x_2)\constantoverleftarrow{D_\alpha}=(x_3)\constantoverleftarrow{D_\alpha}=0$ for all $\alpha\in T_w$. Then, we have
\[
(x_1yx_2x_3)\constantoverleftarrow{D_y}=(x_1(w_ox_2)yx_3)\constantoverleftarrow{D_y}=x_1(w_o(x_2x_3))\,.
\]

\end{cor}

\begin{proof}

Let the notation be as in the statement. The result follows from application of Corollary~\ref{cor:prep-abstr-comm}. We simply have to verify that $(w_ox_2)\constantoverleftarrow{D_\alpha}$ vanishes for all $\alpha\in T_w$. By Fact~\ref{fact:T} and \cite[Remark~3.16]{skew} this vanishing is equivalent to the vanishing $(x_2)\constantoverleftarrow{D_\alpha}=0$ for all $\alpha\in T_w$ which is part of the assumption.
\end{proof}

\subsection*{An analogue of Corollary~\ref{cor:ofbskew} for braided left partial derivatives}

\begin{lem}
\label{lem:leftbraided}

Let $D$ be a disjoint system of order two. Let $w_1,w_2$ be some ordering of the elements of $D$. Let $y=w_1x_{w_o}$. Then, we have $\constantoverrightarrow{D_\alpha}(y)=0$ for all $\alpha\in T_{w_2}$.
    
\end{lem}

\begin{proof}

Indeed, let the notation be as in the statement. By the very definition of the objects involved, we have $\constantoverrightarrow{D_\alpha}(y)=\sum_{u\in W}\left<x_\alpha,w_1x_{w_o/u}\right>w_1x_u$ where the bracket can only be nonzero for $\mathbb{Z}_{\geq 0}$-degree reasons whenever $\ell(w_o)-\ell(u)=1$, i.e.\ whenever $u=s_\beta w_o$ for a simple root $\beta$. In this case, by \cite[Theorem~7.1]{skew}, we have $x_{w_o/u}=x_\beta$ and the corresponding bracket still vanishes because $T_{w_1}$ and $T_{w_2}$ are disjoint by assumption. This proves the lemma.  
\end{proof}

\begin{cor}
\label{cor:leftbraided}

Let $D$ be a disjoint system of order two. Let $w_1,w_2$ be some ordering of the elements of $D$. Let $y_1=w_1x_{w_o}$ and let $y_2=w_2x_{w_o}$. Then, we have $\constantoverrightarrow{D_{y_2}}y_1=(-1)^{\ell(w_o)}y_1\constantoverrightarrow{D_{y_2}}$.
    
\end{cor}

\begin{proof}
This is now immediate form Lemma~\ref{lem:basic-rev}\eqref{item:basic4'}, Lemma~\ref{lem:leftbraided} and Fact~\ref{fact:rho_amel}\eqref{item:mult_wo}.
\end{proof}

\section{\for{toc}{{\color{white}0}}Integrals for Hopf algebras}
\label{sec:integrals}

In this section, we study integrals for Hopf algebras with a view towards our applications to Nichols algebras. For background material on integrals, we refer to the foundational works \cite{larson-sweedler,sullivan,sweedler}, to the more recent reference \cite{integrals-braided} and to the pedagogical text book \cite{majid}. For the relation between integrals and Frobenius algebras, we refer to \cite{gen_frob,larson-sweedler}.

\begin{defn}[{\cite[Equation~(2.3.6)]{integrals-braided}, \cite[Section~2]{sweedler}}]
\label{def:int}

Let $A$ be a braided Hopf algebra in ${}_H^H\EuScript{YD}$. Let $x\in A$.
\begin{itemize}
    \item 
    We say that $x$ is a left integral in $A$ if $zx=\epsilon(z)x$ for all $z\in A$. We say that $x$ is a nonzero left integral in $A$ if $x$ is nonzero and if $x$ is a left integral in $A$.
    \item
    We say that $x$ is a right integral in $A$ if $xz=\epsilon(z)x$ for all $z\in A$. We say that $x$ is a nonzero right integral in $A$ if $x$ is nonzero and if $x$ is a right integral in $A$.    
    \item
    We say that $x$ is an integral in $A$ if $x$ is a left and right integral in $A$. We say that $x$ is a nonzero integral in $A$ if $x$ is nonzero and if $x$ is an integral in $A$.
\end{itemize}

\end{defn}

\begin{rem}

Let $A$ be a braided Hopf algebra in ${}_H^H\EuScript{YD}$. The presence of a nonzero left integral / right integral / integral in $A$ gives rise to a one-dimensional left ideal / right ideal / two-sided ideal in $A$.

\end{rem}

\begin{rem}
\label{rem:int-conv}

If $A$ is a classical Hopf algebra with trivial braiding, then there exists a nonzero left integral in $A$ if and only if there exists a nonzero right integral in $A$ if and only if $A$ is finite dimensional. This theorem is proved in \cite[Corollary~2.7, Equivalence~$(2.11)\Leftrightarrow(2.13)$]{sweedler}. In general, it is known that for a braided Hopf algebra in ${}_H^H\EuScript{YD}$ of finite dimension the two-sided ideal of all left integrals as well as the two-sided ideal of all right integrals are one-dimensional (cf.~\cite[Section~2.3]{integrals-braided}). By using the Radford biproduct, one can reduce the braided situation to the classical situation and show that a braided Hopf algebra in ${}_H^H\EuScript{YD}$ is finite dimensional if it admits a nonzero left integral or a nonzero right integral, see \cite{braidedintegral} for details. 


\end{rem}

\begin{rem}

To handle infinite dimensional braided Hopf algebras $A$ in ${}_H^H\EuScript{YD}$, and to introduce sensible notions of integrals for them which extend the notions in Definition~\ref{def:int}, people often consider left integrals / right integrals / integrals in $A^*$ instead of $A$ where $A^*$ is the algebra dual of $A$ considered as a coalgebra, cf.~\cite[Section~1.7]{majid} and \cite{sullivan}. We mention it but we will not need this approach in this work.

\end{rem}

\begin{prop}
\label{prop:Aint}

Let $A$ be a connected braided $\mathbb{Z}_{\geq 0}$-graded Hopf algebra in ${}_H^H\EuScript{YD}$. Let $x$ be a nonzero left or right integral in $A$.

\begin{enumerate}
    \item\label{item:Aint1}
    The element $x$ is homogeneous of some $\mathbb{Z}_{\geq 0}$-degree $m$. The integer $m$ satisfies $A^m=\mathbf{k}x$ and $A^{m'}=0$ for all $m'>m$. In particular, every left or right integral in $A$ is an integral in $A$.
    \item\label{item:Aint2}
    An element $x'$ is an integral in $A$ if and only if $zx'=0$ for all homogeneous elements $z\in A$ of $\mathbb{Z}_{\geq 0}$-degree $>0$ if and only if $x'z=0$ for all homogeneous elements $z\in A$ of $\mathbb{Z}_{\geq 0}$-degree $>0$. If for this sentence, in addition, the algebra $A$ is generated in $\mathbb{Z}_{\geq 0}$-degree one, then an element $x'$ is an integral in $A$ if and only if $zx'=0$ for all homogeneous elements $z\in A$ of $\mathbb{Z}_{\geq 0}$-degree one if and only if $x'z=0$ or all homogeneous elements $z\in A$ of $\mathbb{Z}_{\geq 0}$-degree one.
    \item\label{item:Aint3}
    For a nonzero element $x'\in A$, there exist homogeneous elements $y,y'\in A$ with respect to the $\mathbb{Z}_{\geq 0}$-grading such that $yx'$ and $x'y'$ are nonzero integrals in $A$.
\end{enumerate}

\end{prop}

\begin{rem}

Recall that a braided $\mathbb{Z}_{\geq 0}$-graded Hopf algebra $A$ in ${}_H^H\EuScript{YD}$ as in the statement of Proposition~\ref{prop:Aint} is connected if and only if $A^0=\mathbf{k}1$.

\end{rem}

\begin{proof}[Proof of Proposition~\ref{prop:Aint}]

Let the notation be as in the statement. Since $A$ is finite dimensional, there exists an integer $m$ such that $A^m\neq 0$ and such that $A^{m'}=0$ for all $m'>m$. By the assumptions on $A$, every element in $A^m$ is an integral in $A$. By the uniqueness of left or right integrals in $A$ up to scalar or the uniqueness of right integrals in $A$ up to scalar (cf.~\cite[Section~2.3 or Proposition~3.2.2]{integrals-braided} or \cite[Lemma~1.12]{andruskiewitsch}), we conclude that $A^m$ is the one-dimensional two-sided ideal of all integrals in $A$. Item~\eqref{item:Aint1} follows from this. Item~\eqref{item:Aint2} follows from Item~\eqref{item:Aint1} and the assumptions on $A$.

Let us prove Item~\eqref{item:Aint3}. Let $x'$ be a nonzero element in $A$. Let $m'$ be the smallest integer such that in the decomposition $\sum_{m''}x^{\prime m''}$ of $x'$ as a sum of homogeneous elements $x^{\prime m''}$ of $\mathbb{Z}_{\geq 0}$-degree $m''$ the element $x^{\prime m'}$ is nonzero. By Item~\eqref{item:Aint1}, it suffices to prove the statement of Item~\eqref{item:Aint3} for $x^{\prime m'}$ instead of $x'$. Hence, we may and will assume right in the beginning that $x'$ is homogeneous of some $\mathbb{Z}_{\geq 0}$-degree $m'$. We now perform an induction on $m-m'$. If $m=m'$, then $x'$ is a nonzero integral in $A$ by Item~\eqref{item:Aint1} and we can set $y=y'=1$. If $m>m'$, then $x'$ is a nonzero element in $A$ which is not a nonzero integral in $A$. By Item~\eqref{item:Aint2}, we can find homogeneous elements $y_1,y_1'\in A$ of $\mathbb{Z}_{\geq 0}$-degree $>0$ such that $y_1x'$ and $x'y_1'$ are both nonzero. By our assumption on $x'$ and our choice of $y_1,y_1'$, we know that $y_1x'$ and $x'y_1'$ are both nonzero homogeneous elements of $\mathbb{Z}_{\geq 0}$-degree $>m'$. By the induction hypothesis, there exists homogeneous elements $y_2,y_2'\in A$ with respect to the $\mathbb{Z}_{\geq 0}$-grading such that $y_2y_1x'$ and $x'y_1'y_2'$ are nonzero integrals in $A$. The elements $y=y_2y_1$ and $y'=y_1'y_2'$ are homogeneous elements of $A$ with respect to the $\mathbb{Z}_{\geq 0}$-grading as required.
\end{proof}

\begin{rem}
\label{rem:bona-fide}

In the rest of this work, we will only be concerned with connected braided $\mathbb{Z}_{\geq 0}$-graded Hopf algebras in ${}_H^H\EuScript{YD}$. As we see from Proposition~\ref{prop:Aint}\eqref{item:Aint1}, for a connected braided $\mathbb{Z}_{\geq 0}$-graded Hopf algebra $A$ in ${}_H^H\EuScript{YD}$, we do not need to make a distinction between left or right integrals in $A$ and integrals in $A$. Therefore, we will not further make the distinction in our statements. From now on, we will state our results only for bona fide integrals.

\end{rem}

\begin{rem}

Let $A$ be a a connected braided $\mathbb{Z}_{\geq 0}$-graded Hopf algebra in ${}_H^H\EuScript{YD}$. In view of Remark~\ref{rem:int-conv} and Proposition~\ref{prop:Aint}\eqref{item:Aint1}, we have the equivalence that
\[
\textit{$A$ is finite dimensional if and only if there exists a nonzero integral in $A$}
\]
which we will use from now on without reference.

\end{rem}

\begin{cor}
\label{cor:ABint}

Let $A$ and $B$ be braided $\mathbb{Z}_{\geq 0}$-graded Hopf algebras in ${}_H^H\EuScript{YD}$. Assume that $A$ or $B$ is connected and that there exists a nondegenerate Hopf duality pairing $\left<-,-\right>$ between $A$ and $B$ which respects the $\mathbb{Z}_{\geq 0}$-grading. Then, both Hopf algebras $A$ and $B$ are connected. Moreover, there exists a nonzero integral in $A$ if and only if there exists a nonzero integral in $B$. Let $x^*$ be a nonzero integral in $A$ and let $x$ be a nonzero integral in $B$. 

\begin{enumerate}
    \item\label{item:pairing}
    We have $\left<x^*,x\right>\neq 0$.
    \item\label{item:pairing2}
    Let $x'\in B$ be a homogeneous element with respect to the $\mathbb{Z}_{\geq 0}$-grading such that $\left<x^*,x\right>=\left<x^*,x'\right>$. Then, it follows that $x=x'$. Let $x^{*\prime}\in A$ be a homogeneous element with respect to the $\mathbb{Z}_{\geq 0}$-grading such that $\left<x^*,x\right>=\left<x^{*\prime},x\right>$. Then, it follows that $x^*=x^{*\prime}$.
\end{enumerate}

\end{cor}

\begin{proof}

Let the notation be as in the statement. By assumption, we have $A^0\cong B^0$, so that $A$ is connected if and only if $B$ is. Because of the presence of a nondegenerate pairing between $A$ and $B$, we know that $A$ is finite dimensional if and only if $B$ is. From this and Proposition~\ref{prop:Aint}\eqref{item:Aint1}, it follows that there exists a nonzero integral in $A$ if and only if there exists a nonzero integral in $B$.

Item~\eqref{item:pairing} follows from Proposition~\ref{prop:Aint}\eqref{item:Aint1} and the assumptions on $\left<-,-\right>$. Item~\eqref{item:pairing2} follows from Item~\eqref{item:pairing} and again from Proposition~\ref{prop:Aint}\eqref{item:Aint1} and the assumptions on $\left<-,-\right>$.
\end{proof}

\begin{prop}
\label{prop:Nicholsint}

Let $V$ be a finite dimensional Yetter-Drinfeld $H$-module. Let $x$ be a nonzero integral in $\EuScript{B}(V)$ and let $x^*$ be a nonzero integral in $\EuScript{B}(V^*)$.

\begin{enumerate}
    \item\label{item:Nicholsint2}
    Every nonzero element $y\in\EuScript{B}(V)$ satisfies $(x^*)\constantoverleftarrow{D_y}\neq 0$ and every nonzero element $y^*\in\EuScript{B}(V^*)$ satisfies $\constantoverrightarrow{D_{y^*}}(x)\neq 0$.
    \item\label{item:Nicholsint3}
    Let $x'\in\EuScript{B}(V)^m$ and $y^*\in\EuScript{B}(V^*)^{m^*}$ for some $m,m^*\in\mathbb{Z}_{\geq 0}$ such that $y^*$ is nonzero and such that $\constantoverrightarrow{D_{y^*}}(x)=\constantoverrightarrow{D_{y^*}}(x')$. Then, we have $x=x'$.
    Let $x^{*\prime}\in\EuScript{B}(V^*)^{m^*}$ and $y\in\EuScript{B}(V)^m$ for some $m,m^*\in\mathbb{Z}_{\geq 0}$ such that $y$ is nonzero and such that $(x^*)\constantoverleftarrow{D_y}=(x^{*\prime})\constantoverleftarrow{D_y}$. Then, we have $x^*=x^{*\prime}$.
\end{enumerate}

\end{prop}

\begin{proof}

Let $x$ and $x^*$ be as in the statement. Let us prove Item~\eqref{item:Nicholsint2}. Let $y$ be a nonzero element in $\EuScript{B}(V)$ and let $y^*$ be a nonzero element in $\EuScript{B}(V^*)$. By Proposition~\ref{prop:Aint}\eqref{item:Aint3}, there exist homogeneous elements $x'\in\EuScript{B}(V)$ and $x^{*\prime}\in\EuScript{B}(V^*)$ with respect to the $\mathbb{Z}_{\geq 0}$-grading such that $yx'$ is a nonzero integral in $\EuScript{B}(V)$ and such that $x^{*\prime}y^*$ is a nonzero integral in $\EuScript{B}(V^*)$. From Corollary~\ref{cor:ABint}\eqref{item:pairing}, it follows that
\begin{align*}
0\neq\,\left<x^*,yx'\right>\,&=\,\big<(x^*)\constantoverleftarrow{D_y},x'\big>\mspace{5.5mu},\\
0\neq\left<x^{*\prime}y^*,x\right>&=\big<x^{*\prime},\constantoverrightarrow{D_{y^*}}(x)\big>\,.
\end{align*}
The result in Item~\eqref{item:Nicholsint2} is clear from these equations.

Let us prove Item~\eqref{item:Nicholsint3}. Let the notation be as in the first sentence of the statement of Item~\eqref{item:Nicholsint3}. By Proposition~\ref{prop:Aint}\eqref{item:Aint1}, \cite[Remark~2.16]{skew} and Item~\eqref{item:Nicholsint2} of this proposition, we note that $x'$ is nonzero and that $m$ is the $\mathbb{Z}_{\geq 0}$-degree of $x$. Again, by Proposition~\ref{prop:Aint}\eqref{item:Aint1}, we see that $x'$ is a nonzero integral in $\EuScript{B}(V)$ and that there exists a nonzero scalar $\lambda$ such that $x'=\lambda x$. It follows that $\constantoverrightarrow{D_{y^*}}(x)=\lambda\constantoverrightarrow{D_{y^*}}(x)$. By Item~\eqref{item:Nicholsint2}, we conclude that $\lambda=1$ and thus $x=x'$. The rest of the statement of Item~\eqref{item:Nicholsint3} can be proved analogously. In the case of a finite dimensional Yetter-Drinfeld $\Gamma$-module whose support consists of involutions, the rest of the statement of Item~\eqref{item:Nicholsint3} follows equally well from the first sentence of the statement of Item~\eqref{item:Nicholsint3} by application of $\bar{\EuScript{S}}$, Proposition~\ref{prop:Aint}\eqref{item:Aint1} and \cite[Proposition~3.7(6), Proposition~4.2, Remark~4.3]{skew}.
\end{proof}

\begin{lem}
\label{lem:char}

Let $A$ be a connected braided $\mathbb{Z}_{\geq 0}$-graded Hopf algebra in ${}_\Gamma^\Gamma\EuScript{YD}$. Let $x$ be a nonzero integral in $A$. Then, the element $x$ is homogeneous of some $\Gamma$-degree which is central in $\Gamma$ and there exists a unique character $\Gamma\to\mathbf{k}^\times,g\mapsto\lambda_g$ such that $gx=\lambda_gx$ for all $g\in\Gamma$. If for this sentence, in addition, $\Gamma$ is generated by involutions, then the character as in the previous sentence satisfies $\lambda_g\in\{-1,1\}$ for all $g\in\Gamma$.

\end{lem}

\begin{proof}

Let the notation be as in the statement. We know from Proposition~\ref{prop:Aint}\eqref{item:Aint1} that $x$ is homogeneous of some $\Gamma$-degree $h$. For every $g\in\Gamma$, the element $gx$ is a nonzero integral in $A$ by Proposition~\ref{prop:Aint}\eqref{item:Aint1}. Hence, for every $g\in\Gamma$, there exists a unique nonzero scalar $\lambda_g$ such that $gx=\lambda_g x$. By comparing $\Gamma$-degrees in the previous equation, we see that $ghg^{-1}=h$ for all $g\in\Gamma$. In other words, the element $h$ is central in $\Gamma$. It is clear that $g\mapsto\lambda_g$ defines a character $\Gamma\to\mathbf{k}^\times$ which is uniquely determined by the property that $gx=\lambda_gx$ for all $g\in\Gamma$. If in addition $\Gamma$ is generated by involutions, we must have $\lambda_g\in\{-1,1\}$ for all $g\in\Gamma$ because $g\mapsto\lambda_g$ is multiplicative and because $\lambda_g^2=1$ for every involution $g\in\Gamma$.
\end{proof}

\begin{defn}

For a vector space, we define an equivalence relation $\sim$ by declaring two vectors $x$ and $x'$ to be equivalent, in formulas $x\sim x'$, if there exists a sign $\epsilon$ such that $x=\epsilon x'$. We use the same symbol $\sim$ for this equivalence relation regardless on which vector space it is defined.

\end{defn}

\begin{cor}
\label{cor:BWintegral}

Let $x$ be a nonzero integral in $\EuScript{B}_W$.

\begin{enumerate}
    \item\label{item:parity}
    The element $x$ is homogeneous of some $\mathbb{Z}_{\geq 0}$-degree $m$ and some $W$-degree $w$. The parity of $\ell(w)$ equals the parity of $m$.
    \item\label{item:sign}
    For all $g\in W$, we have $x\sim gx$. Moreover, we have $x\sim\rho(x)\sim\EuScript{S}(x)\sim\bar{\EuScript{S}}(x)$. 
    \item\label{item:sign_elab}
    In the sense of Item~\eqref{item:sign}, we define signs $\epsilon_\rho$, $\epsilon_{\EuScript{S}}$, $\epsilon_{\bar{\EuScript{S}}}$ such that $x=\epsilon_\rho\rho(x)$, $x=\epsilon_{\EuScript{S}}\EuScript{S}(x)$, $x=\epsilon_{\bar{\EuScript{S}}}\bar{\EuScript{S}}(x)$. Then, we have $\epsilon_\rho=\epsilon_{\bar{\EuScript{S}}}$ and $\epsilon_{\EuScript{S}}=(-1)^m$ where $m$ is defined as in Item~\eqref{item:parity}.
    \item\label{item:sign_elab2}
    With $w$ defined as in Item~\eqref{item:parity}, we further have $wx=(-1)^{\ell(w)}x$.
    \item\label{item:Aint2_elab}
    An element $x'$ is an integral in $\EuScript{B}_W$ if and only if $x_\alpha x'=0$ for all $\alpha\in R^+$ if and only if $x'x_\alpha=0$ for all $\alpha\in R^+$.
    \item\label{item:cor:M}
    Let $x'$ be a nonzero element of $\EuScript{B}_W$. Then, there exist monomials $M,M'\in\EuScript{B}_W$ such that $Mx'$ and $x'M'$ are nonzero integrals in $\EuScript{B}_W$. In particular, there exists a monomial which is a nonzero integral in $\EuScript{B}_W$.
\end{enumerate}

\end{cor}

\begin{proof}

Let the notation be as in the statement. Let us prove Item~\eqref{item:cor:M} first. By Proposition~\ref{prop:Aint}\eqref{item:Aint3}, there exist homogeneous elements $y,y'\in\EuScript{B}_W$ with respect to the $\mathbb{Z}_{\geq 0}$-grading such that $yx'$ and $x'y'$ are nonzero integrals in $\EuScript{B}_W$. Let us assume further that $y$ and $y'$ are chosen as in the proof of Proposition~\ref{prop:Aint}\eqref{item:Aint3}, i.e.\ with respect to the nonzero component of $x'$ of smallest $\mathbb{Z}_{\geq 0}$-degree. We can write $y$ respectively $y'$ as a sum $\sum_i M_i$ respectively $\sum_i M_i'$ of monomials $M_i,M_i'$ all of equal $\mathbb{Z}_{\geq 0}$-degree equal to the $\mathbb{Z}_{\geq 0}$-degree of $y$ and $y'$. Then, we have $\sum_i M_ix'\neq 0$ and $\sum_i x'M_i'\neq 0$ and consequently $M_ix'\neq 0$ and $x'M_j'\neq 0$ for some $i,j$. By Proposition~\ref{prop:Aint}\eqref{item:Aint1},\eqref{item:Aint3} and our choice of $y,y',M_i,M_j'$, we know that $M_ix'$ and $x'M_j'$ are nonzero integrals in $\EuScript{B}_W$. Thus, the monomials $M=M_i$ and $M'=M_j'$ are as required. The particular case follows by application to $x'=1$.

Let us prove Item~\eqref{item:parity}. By Proposition~\ref{prop:Aint}\eqref{item:Aint1} and Lemma~\ref{lem:char}, we can define $m$ and $w$ as in the statement. The rest of the item follows by application of Lemma~\ref{lem:zdeg-length} to $x$.

The first sentence of Item~\eqref{item:sign} follows from Lemma~\ref{lem:char}. The morphisms $\rho$ and $\bar{\EuScript{S}}$ are $\mathbb{Z}_{\geq 0}$-graded involutions by \cite[Proposition~3.7(1),(6)]{skew}. Hence, it follows from Proposition~\ref{prop:Aint}\eqref{item:Aint1} that $x\sim\rho(x)\sim\bar{\EuScript{S}}(x)$. Using these last relations, Proposition~\ref{prop:Aint}\eqref{item:Aint1} or Item~\eqref{item:parity} of this corollary, and the definition of $\bar{\EuScript{S}}$, we finally see that $\EuScript{S}(x)\sim\rho\bar{\EuScript{S}}(x)\sim\rho(x)\sim x$.

Let us prove Item~\eqref{item:sign_elab}. By definition of $\bar{\EuScript{S}}$, it is clear that $\epsilon_{\bar{\EuScript{S}}}=(-1)^m\epsilon_\rho\epsilon_{\EuScript{S}}$ where $m$ is defined as in Item~\eqref{item:parity}. Hence, it suffices to prove that $\epsilon_\rho=\epsilon_{\bar{\EuScript{S}}}$. By \cite[Proposition~3.7(6),~3.10]{skew}, we see that
\[
\left<x,x\right>=\langle\rho(x),\bar{\EuScript{S}}(x)\rangle=\epsilon_\rho\epsilon_{\bar{\EuScript{S}}}\left<x,x\right>\,.
\]
In view of this equation, Corollary~\ref{cor:ABint}\eqref{item:pairing} completes the proof of Item~\eqref{item:sign_elab}.

Let us prove Item~\eqref{item:sign_elab2}. In view of Item~\eqref{item:sign}, we know that $x=\EuScript{S}^2(x)$. And further, with $w$ defined as in Item~\eqref{item:parity}, that $\EuScript{S}^2(x)=(-1)^{\ell(w)}wx$ by Corollary~\ref{cor:nichols-zoeller1}. The result in Item~\eqref{item:sign_elab2} follows from these observations. Finally, Item~\eqref{item:Aint2_elab} is immediate from Proposition~\ref{prop:Aint}\eqref{item:Aint2} because $(x_\alpha)_{\alpha\in R^+}$ is a basis of $V_W$.
\end{proof}

\begin{cor}[Abstract commutativity]
\label{cor:abstr-comm}

Let $w\in W$. Let $y=wx_{w_o}$. Let $x_1,x_2,x_3\in\EuScript{B}_W$ be homogeneous elements with respect to the $\mathbb{Z}_{\geq 0}$-grading such that $x_1yx_2x_3$ is a nonzero integral in $\EuScript{B}_W$ and such that 
\[
(x_1)\constantoverleftarrow{D_\alpha}=(x_2)\constantoverleftarrow{D_\alpha}=(ww_ow^{-1}x_2)\constantoverleftarrow{D_\alpha}=(x_3)\constantoverleftarrow{D_\alpha}=0
\]
for all $\alpha\in T_w$. Then, we have $x_1yx_2x_3=x_1(ww_ow^{-1}x_2)yx_3$.

\end{cor}

\begin{proof}

By assumption, the element $x_1(ww_ow^{-1}x_2)yx_3$ is homogeneous with respect to the $\mathbb{Z}_{\geq 0}$-grading. Further, the element $y$ is nonzero. Therefore, the result follows from Corollary~\ref{cor:prep-abstr-comm} and Proposition~\ref{prop:Nicholsint}\eqref{item:Nicholsint3}.
\end{proof}

\begin{cor}
\label{cor:abstr-comm2}

Let $w$ be an element in the centralizer of $w_o$. Let $y=wx_{w_o}$. Let $x_1,x_2,x_3\in\EuScript{B}_W$ be homogeneous elements with respect to the $\mathbb{Z}_{\geq 0}$-grading such that $x_1yx_2x_3$ is a nonzero integral in $\EuScript{B}_W$ and such that $(x_1)\constantoverleftarrow{D_\alpha}=(x_2)\constantoverleftarrow{D_\alpha}=(x_3)\constantoverleftarrow{D_\alpha}=0$ for all $\alpha\in T_w$. Then, we have $x_1yx_2x_3=x_1(w_ox_2)yx_3$.

\end{cor}

\begin{proof}

By assumption, the element $x_1(w_ox_2)yx_3$ is homogeneous with respect to the $\mathbb{Z}_{\geq 0}$-grading. Further, the element $y$ is nonzero. Therefore, the result follows from Corollary~\ref{cor:prep-abstr-comm2} and Proposition~\ref{prop:Nicholsint}\eqref{item:Nicholsint3}.
\end{proof}

\begin{rem}
\label{rem:suppl_intro}

Whenever we apply Corollary~\ref{cor:ABint}\eqref{item:pairing} to $\EuScript{B}_W$ or one of the results derived with its help, e.g.~Corollary~\ref{cor:abstr-comm},~\ref{cor:abstr-comm2}, we use a consequence of the fact that $\EuScript{B}_W$ admits a \emph{nondegenerate} Hopf duality pairing between $\EuScript{B}_W$ and itself.

\end{rem}

\begin{rem}

Abstract commutativity is actually only used once in this paper in the form of Corollary~\ref{cor:abstr-comm2} to prove Lemma~\ref{lem:invariance-prop}\eqref{item:inv5} but we formulate it here as an independent result as it may be of use elsewhere. In turn, Lemma~\ref{lem:invariance-prop}\eqref{item:inv5} is needed in the proof of Theorem~\ref{thm:inv-hypo} which is a necessary step towards the proof of Theorem~\ref{thm:mainS6}.
    
\end{rem}

\section{Invariance of integrals}
\label{sec:invariance}

Under invariance properties of integrals, we understand formulas which show that a nonzero integral in $\EuScript{B}_W$ is invariant under certain operators which lie in the image of the embeddings of the tensor square into endomorphisms as in Remark~\ref{rem:tensor-square} or are composites thereof. Such properties can be derived manifoldly using the results in Section~\ref{sec:braided-leibniz}. In this section, we present a selection of such.

\begin{lem}
\label{lem:prep-inv}

Let $x\in\EuScript{B}_W$ be an element such that $xx_\alpha=0$ for some $\alpha\in R$. Then, it follows that $(x)\constantoverleftarrow{D_\alpha}x_\alpha=x$.

\end{lem}

\begin{proof}

Let the notation be as in the statement. The braided Leibniz rule implies in particular that $\constantoverleftarrow{D_\alpha}x_\alpha=1-x_\alpha\constantoverleftarrow{D_\alpha}$. The result follows by application of the previous operator to $x$ because of the assumed vanishing.
\end{proof}

\begin{lem}[Invariance properties]
\label{lem:invariance-prop}

Let $x$ be a nonzero integral in $\EuScript{B}_W$.

\begin{enumerate}
    \item\label{item:inv1}
    Let $\alpha\in R$. Then, we have 
    \begin{align*}
    x&=\mathrlap{(x)\constantoverleftarrow{D_\alpha}x_\alpha\,.}\hphantom{(-1)^{\ell(w_o)}y_1y_2(x)\constantoverleftarrow{D_{y_1y_2}}\,. }
    \end{align*}
    \item\label{item:inv2}
    Let $\alpha\in R$. Let $\epsilon$ be the sign such that $s_\alpha x=\epsilon x_\alpha$ (which exists by Corollary~\ref{cor:BWintegral}\eqref{item:sign}). Then, we have 
    \begin{align*}
    x&=\mathrlap{-\epsilon x_\alpha(x)\constantoverleftarrow{D_\alpha}\,.}\hphantom{(-1)^{\ell(w_o)}y_1y_2(x)\constantoverleftarrow{D_{y_1y_2}}\,. }
    \end{align*}
    \item\label{item:inv3}
    Let $w\in W$. Let $y=wx_{w_o}$. Then, we have
    \begin{align*}
    x&=\mathrlap{(x)\constantoverleftarrow{D_y}y\,.}\hphantom{(-1)^{\ell(w_o)}y_1y_2(x)\constantoverleftarrow{D_{y_1y_2}}\,. }
    \end{align*}
    \item\label{item:inv4}
    Let $w$ be an element in the centralizer of $w_o$. Let $y=wx_{w_o}$. Let $\epsilon$ be the sign such that $w_ox=\epsilon x$ (which exists by Corollary~\ref{cor:BWintegral}\eqref{item:sign}). Then, we have
    \begin{align*}
    x&=\mathrlap{(-1)^{\ell(w_o)}\epsilon y(x)\constantoverleftarrow{D_y}\,.}\hphantom{(-1)^{\ell(w_o)}y_1y_2(x)\constantoverleftarrow{D_{y_1y_2}}\,. }
    \end{align*}
    \item\label{item:inv5}
    Let $D$ be a disjoint system of order two. Let $w_1,w_2$ be some ordering of the elements of $D$. Let $y_1=w_1x_{w_o}$ and let $y_2=w_2x_{w_o}$. Then, we have
    \begin{align*}
    x&=(-1)^{\ell(w_o)}y_1y_2(x)\constantoverleftarrow{D_{y_1y_2}}\,,\\
    x&=(x)\constantoverleftarrow{D_{y_1y_2}}y_2y_1\,.
    \end{align*}
\end{enumerate}

\end{lem}

\begin{proof}[Proof of Item~\eqref{item:inv1}]

The first item follows from Lemma~\ref{lem:prep-inv} and Proposition~\ref{prop:Aint}\eqref{item:Aint2}.
\end{proof}

\begin{proof}[Proof of Item~\eqref{item:inv2},\eqref{item:inv3},\eqref{item:inv4}]

Let the notation be as in the statement. By Proposition~\ref{prop:Aint}\eqref{item:Aint1} and \cite[Remark~2.16]{skew}, we know that the terms on the right side of the displayed equations in the three times we want to proof here are all homogeneous with respect to the $\mathbb{Z}_{\geq 0}$-grading. We see that Proposition~\ref{prop:Nicholsint}\eqref{item:Nicholsint3} applies. Thus, it suffices to prove that
\begin{align*}
(x)\constantoverleftarrow{D_\alpha}&=-\epsilon\big(x_\alpha(x)\constantoverleftarrow{D_\alpha}\big)\constantoverleftarrow{D_\alpha}\,,\\
(x)\constantoverleftarrow{D_y}&=\big((x)\constantoverleftarrow{D_y}y\big)\constantoverleftarrow{D_y}\,.\\
(x)\constantoverleftarrow{D_y}&=(-1)^{\ell(w_o)}\epsilon\big(y(x)\constantoverleftarrow{D_y}\big)\constantoverleftarrow{D_y}\,,
\end{align*}
where the symbols appearing in the three equalities are supposed to be defined as in the three corresponding items. The first of these three equalities follows directly from the braided Leibniz rule and \cite[Remark~3.16]{skew}, the second from Fact~\ref{fact:rho_amel}\eqref{item:nil_pair} and Corollary~\ref{cor:tow_inv}, and the third from Fact~\ref{fact:rho_amel}\eqref{item:mult_wo},\eqref{item:nil_pair}, Corollary~\ref{cor:y_van_amel},~\ref{cor:bcommuteswithskew} and \cite[Remark~3.16]{skew}.
\end{proof}

\begin{proof}[Proof of Item~\eqref{item:inv5}]

Let the notation be as in the statement. With the help of Fact~\ref{fact:rho_amel}\eqref{item:mult_wo}, Lemma~\ref{lem:basic-rev}\eqref{item:basic3},\eqref{item:basic4}, Corollary~\ref{cor:ofbskew},~\ref{cor:abstr-comm2}, Item~\eqref{item:inv3},\eqref{item:inv4} and \cite[Remark~3.16]{skew}, we compute that
\begin{align*}
    (-1)^{\ell(w_o)}y_1y_2(x)\constantoverleftarrow{D_{y_1y_2}}&=(-1)^{\ell(w_o)}\epsilon y_1(x)\constantoverleftarrow{D_{y_1y_2}}y_2\\
    &=(-1)^{\ell(w_o)}\epsilon(y_1(x)\constantoverleftarrow{D_{y_1}})\constantoverleftarrow{D_{y_2}}y_2\\
    &=(x)\constantoverleftarrow{D_{y_2}}y_2\\
    &=x\,,\\
    (x)\constantoverleftarrow{D_{y_1y_2}}y_2y_1&=(-1)^{\ell(w_o)}(x)\constantoverleftarrow{D_{y_1y_2}}y_1y_2\\
    &=(x)\constantoverleftarrow{D_{y_1}}y_1\constantoverleftarrow{D_{y_2}}y_2\\
    &=(x)\constantoverleftarrow{D_{y_2}}y_2\\
    &=x\,,
\end{align*}
where $\epsilon$ is the sign such that $w_ox=\epsilon x$ (as in the statement of Item~\eqref{item:inv4} -- which exists by Corollary~\ref{cor:BWintegral}\eqref{item:sign}).
\end{proof}

\begin{cor}
\label{cor:hypo-bracket}

Let $x$ be a nonzero integral in $\EuScript{B}_W$. Let $w$ be an element in the centralizer of $w_o$. Let $y=wx_{w_o}$. Then, we have $\left<(x)\constantoverleftarrow{D_y},(x)\constantoverleftarrow{D_y}\right>\neq 0$.

\end{cor}

\begin{proof}

Let the notation be as in the statement. By Lemma~\ref{lem:invariance-prop}\eqref{item:inv4}, we have $\left<(x)\constantoverleftarrow{D_y},(x)\constantoverleftarrow{D_y}\right>=(-1)^{\ell(w_o)}\epsilon\left<x,x\right>$ where $\epsilon$ is the sign such that $w_ox=\epsilon x$ (which exists by Corollary~\ref{cor:BWintegral}\eqref{item:sign}). The result follows from this and Corollary~\ref{cor:ABint}\eqref{item:pairing}.
\end{proof}

\section{Disjoint systems and integrals}
\label{sec:motiv}

In this section, we explain the relation between complete disjoint systems and integrals in $\EuScript{B}_W$. More specifically, we explain in Lemma~\ref{lem:motiv}:~$\eqref{item:equiv-1}\Rightarrow\eqref{item:equiv-3}$ how the existence of certain integrals in $\EuScript{B}_W$ implies commutativity relations up to scalar multiple.

\begin{lem}
\label{lem:prep-motiv}

Let $D$ be a disjoint system of order two. Let $w_1,w_2$ be some ordering of the elements of $D$. Let $y_1=w_1x_{w_o}$ and let $y_2=w_2x_{w_o}$. If $y_1y_2$ and $y_2y_1$ are linearly dependent, then we necessarily have $y_1y_2=(-1)^{\ell(w_o)}y_2y_1$.

\end{lem}

\begin{proof}

Let the notation be as in the statement. Suppose that $y_1y_2$ and $y_2y_1$ are linearly dependent. By Corollary~\ref{cor:bracket}, we know that both $y_1y_2$ and $y_2y_1$ are nonzero, hence linear dependence implies the existence of a nonzero scalar $\lambda$ such that $y_1y_2=\lambda y_2y_1$. If we apply $\constantoverleftarrow{D_{y_1y_2}}$ to both sides of this equality, we find in view of Corollary~\ref{cor:bracket} that $\lambda=(-1)^{\ell(w_o)}$. This completes the proof.
\end{proof}

\begin{lem}
\label{lem:motiv}

Suppose that $\EuScript{B}_W$ is finite dimensional. Let $D$ be a complete disjoint system of order $r$. Let $w_1,\ldots,w_r$ be some ordering of the elements of $D$. Let $y_i=w_ix_{w_o}$ for all $1\leq i\leq r$. The following items are equivalent.

\begin{enumerate}
    \item\label{item:equiv-1}
    The element $y_{\sigma(1)}\cdots y_{\sigma(r)}$ is a nonzero integral in $\EuScript{B}_W$ for some $\sigma\in\mathbb{S}_r$.
    \item\label{item:equiv-2}
    The element $y_{\sigma(1)}\cdots y_{\sigma(r)}$ is a nonzero integral in $\EuScript{B}_W$ for all $\sigma\in\mathbb{S}_r$.
    \item\label{item:equiv-3}
    We have $y_iy_j=(-1)^{\ell(w_o)}y_jy_i$ for all $1\leq i,j\leq r$.
\end{enumerate}

\end{lem}

\begin{proof}

The Implication~$\eqref{item:equiv-2}\Rightarrow\eqref{item:equiv-1}$ is obvious. The Implication~$\eqref{item:equiv-1}\Rightarrow\eqref{item:equiv-2}$ follows from Corollary~\ref{cor:bracket} and Proposition~\ref{prop:Aint}\eqref{item:Aint1}. 

Let us prove the Implication~$\eqref{item:equiv-2}\Rightarrow\eqref{item:equiv-3}$. The equality claimed in Item~\eqref{item:equiv-3} is obvious for all $1\leq i=j\leq r$. Because of Lemma~\ref{lem:y_van_amel} and Corollary~\ref{cor:y_van_amel} both sides are zero. Let $1\leq i\neq j\leq r$ be arbitrary but fixed. By Proposition~\ref{prop:Aint}\eqref{item:Aint1} and Item~\eqref{item:equiv-2}, there exists a nonzero scalar $\lambda$ such that
$y_iy_jQ=\lambda y_jy_iQ$ where
\[
Q=y_1\cdots\hat{y}_i\cdots\hat{y}_j\cdots y_r\,.
\]
If we now apply $\constantoverleftarrow{D_Q}$ to $y_iy_jQ=\lambda y_jy_iQ$, we find in view of Lemma~\ref{lem:basic-rev}\eqref{item:basic3},\eqref{item:basic4} and Corollary~\ref{cor:bracket} that $y_iy_j=\lambda y_jy_i$. In view of Lemma~\ref{lem:prep-motiv}, we find that $\lambda=(-1)^{\ell(w_o)}$. This proves Item~\eqref{item:equiv-3}.

Let us prove the Implication~$\eqref{item:equiv-3}\Rightarrow\eqref{item:equiv-1}$. Let $x=y_1\cdots y_r$. By Lemma~\ref{lem:y_van_amel}, Corollary~\ref{cor:bracket} and Item~\eqref{item:equiv-3}, the element $x$ is a nonzero monomial which starts with $\gamma$ for all $\gamma\in R^+$. By Lemma~\ref{lem:basic-rev}\eqref{item:basic0} and Corollary~\ref{cor:BWintegral}\eqref{item:Aint2_elab}, it follows that $x$ is a nonzero integral in $\EuScript{B}_W$.
\end{proof}

\begin{conj}
\label{conj:milinski}

Suppose that $\EuScript{B}_W$ is finite dimensional. Let $D$ be a complete disjoint system of order $r$. Let $w_1,\ldots,w_r$ be some ordering of the elements of $D$. Let $y_i=w_ix_{w_o}$ for all $1\leq i\leq r$. Then, the element $y_1\cdots y_r$ is a nonzero integral in $\EuScript{B}_W$.

\end{conj}

\begin{rem}[Motivational remark]
\label{rem:melinski}

In type $\mathsf{A}_1$ and $\mathsf{A}_3$, the assumptions and the conclusion of Conjecture~\ref{conj:milinski} are satisfied for trivial reasons and by \cite[Example~6.4]{milinski-schneider}. This simple observation makes part of our motivation for the notion of disjoint systems, for Conjecture~\ref{conj:milinski} and for the proceeding in this paper in general.

\end{rem}

\section{Coproducts of integrals}
\label{sec:coproduct}

In this section, we explain how finite dimensionality of $\EuScript{B}_W$ in general implies commutativity relations up to multiplication with a nonzero element in $\EuScript{B}_W$ (cf.~Theorem~\ref{thm:concrete}).

\begin{defn}

Let $V$ be a $\mathbb{Z}_{\geq 0}$-graded vector space. Let $Y$ be a $\mathbb{Z}_{\geq 0}$-graded vector subspace of $V$. We call $U$ a $\mathbb{Z}_{\geq 0}$-graded complement of $Y$ in $V$ if $U$ is a $\mathbb{Z}_{\geq 0}$-graded vector subspace of $V$ such that $V\cong Y\oplus U$. It is clear that a $\mathbb{Z}_{\geq 0}$-graded complement of $Y$ in $V$ exists for every $V$ and $Y$ as above.

\end{defn}

\begin{lem}[Key lemma on concrete commutativity]
\label{lem:prep-concrete}

Let $\Theta_1,\Theta_2\subseteq R^+$ be two subsets. Let $y_1$ be a monomial which does only involve $\Theta_1$, and let $y_2$ be a monomial which does only involve $\Theta_2$. Let $y,\bar{y}$ be monomials which do only involve $\Theta_1\cup\Theta_2$ such that $yy_1y_2$ and $\bar{y}y_2y_1$ are linearly independent. Let $b,\bar{b}\in\EuScript{B}_W$ be homogeneous elements with respect to the $\mathbb{Z}_{\geq 0}$-grading and the $W$-grading such that $byy_1y_2$ and $\bar{b}\bar{y}y_2y_1$ are nonzero integrals in $\EuScript{B}_W$. Then, there exist monomials $y',\bar{y}'$ which do only involve $\Theta_1\cup\Theta_2$ and homogeneous elements $b',\bar{b}'\in\EuScript{B}_W$ with respect to the $\mathbb{Z}_{\geq 0}$-grading and the $W$-grading such that
\[
\operatorname{deg}_{\mathbb{Z}_{\geq 0}}y+\operatorname{deg}_{\mathbb{Z}_{\geq 0}}\bar{y}<\operatorname{deg}_{\mathbb{Z}_{\geq 0}}y'+\operatorname{deg}_{\mathbb{Z}_{\geq 0}}\bar{y}'\,,
\]
where $\operatorname{deg}_{\mathbb{Z}_{\geq 0}}$ denotes the $\mathbb{Z}_{\geq 0}$-degree, and such that $b'y'y_1y_2$ and $\bar{b}'\bar{y}'y_2y_1$ are nonzero integrals in $\EuScript{B}_W$.

\end{lem}

\begin{proof}

Let the notation be as in the statement. Without loss of generality, we may assume that the $\mathbb{Z}_{\geq 0}$-degree of $b$ is larger or equal than the $\mathbb{Z}_{\geq 0}$-degree of $\bar{b}$. By assumption and Proposition~\ref{prop:Aint}\eqref{item:Aint1}, there exists a nonzero scalar $\lambda$ such that $byy_1y_2=\lambda\bar{b}\bar{y}y_2y_1$. Let
\begin{align*}
\Delta(b)&=1\otimes b+\sum_i b_{(1)}^i\otimes b_{(2)}^i\,,\\
\Delta(\bar{b})&=1\otimes\bar{b}+\sum_i\bar{b}_{(1)}^i\otimes\bar{b}_{(2)}^i
\end{align*}
be decompositions such that $b_{(2)}^i$ is homogeneous with respect to the $W$-grading and homogeneous of $\mathbb{Z}_{\geq 0}$-degree less than the $\mathbb{Z}_{\geq 0}$-degree of $b$ and such that $\bar{b}_{(2)}^i$ is homogeneous with respect to the $W$-grading and homogeneous of $\mathbb{Z}_{\geq 0}$-degree less than the $\mathbb{Z}_{\geq 0}$-degree of $\bar{b}$. Let $U$ be a $\mathbb{Z}_{\geq 0}$-graded complement of $\mathbf{k}yy_1y_2\oplus\mathbf{k}\bar{y}y_2y_1$ in $\EuScript{B}_W$. Let $\mathbb{P}$ be the natural projection
\[
\mathbb{P}\colon\EuScript{B}_W\cong\mathbf{k}yy_1y_2\oplus\mathbf{k}\bar{y}y_2y_1\oplus U\to\mathbf{k}yy_1y_2\cong\mathbf{k}
\]
onto $\mathbf{k}yy_1y_2\cong\mathbf{k}$. If we apply now the coproduct to both sides of the equation $byy_1y_2=\lambda\bar{b}\bar{y}y_2y_1$, apply further $\mathbb{P}\otimes 1$ to the result and identify it along the isomorphism $\mathbf{k}\otimes\EuScript{B}_W\cong\EuScript{B}_W$, we find that
\[
b+\sum_{i,j}\lambda_{i,j}b_{(2)}^i(yy_1y_2)_j=\lambda\sum_j\bar{\lambda}_j\bar{b}(\bar{y}y_2y_1)_j+\lambda\sum_{i,j}\bar{\lambda}_{i,j}\bar{b}_{(2)}^i(\bar{y}y_2y_1)_j
\]
where the $\lambda_{i,j},\bar{\lambda}_j,\bar{\lambda}_{i,j}$ are scalars which result from the application of $\mathbb{P}$ to the first tensor factors of the decompositions of $\Delta(byy_1y_2)$ and $\Delta(\bar{b}\bar{y}y_2y_1)$ and where the $(yy_1y_2)_j$ and $(\bar{y}y_2y_1)_j$ are sub-monomials of $yy_1y_2$ and $\bar{y}y_2y_1$ which are nonconstant whenever the corresponding scalar is nonzero. If we multiply the previous equation with $yy_1y_2$ from the right, we can find a monomial $y'$ which does only involve $\Theta_1\cup\Theta_2$ such that $\smash{\operatorname{deg}_{\mathbb{Z}_{\geq 0}}y<\operatorname{deg}_{\mathbb{Z}_{\geq 0}}y'}$ and such that $b'y'y_1y_2$ is a nonzero integral in $\EuScript{B}_W$ where $b'$ is either $b_{(2)}^i$, $\smash{\bar{b}}$ or $\smash{\bar{b}_{(2)}^i}$. If we set $\bar{y}'=\bar{y}$ and $\smash{\bar{b}'=\bar{b}}$, we have found $\smash{y',\bar{y}',b',\bar{b}'}$ with all the properties required in the statement. This completes the proof.
\end{proof}

\begin{thm}[Concrete commutativity]
\label{thm:concrete}

Let $\Theta_1,\Theta_2\subseteq R^+$. Let $y_1$ be a monomial which does only involve $\Theta_1$, and let $y_2$ be a monomial which does only involve $\Theta_2$ such that $y_1y_2$ and $y_2y_1$ are both nonzero. If $\EuScript{B}_W$ is finite dimensional, then there exist monomials $y,\bar{y}$ which do only involve $\Theta_1\cup\Theta_2$ such that
\[
\mathbf{k}yy_1y_2=\mathbf{k}\bar{y}y_2y_1\neq 0\,.
\]

\end{thm}

\begin{proof}

Let the notation be as in the statement. Suppose that $\EuScript{B}_W$ is finite dimensional. Suppose that $y_1y_2$ and $y_2y_1$ are linearly independent. Otherwise, we can set $y=\bar{y}=1$ and are done by assumption. Let $y,\bar{y}$ be monomials which do only involve $\Theta_1\cup\Theta_2$ such that $yy_1y_2$ and $\bar{y}y_2y_1$ are linearly independent and such that $\operatorname{deg}_{\mathbb{Z}_{\geq 0}}y+\operatorname{deg}_{\mathbb{Z}_{\geq 0}}\bar{y}$ is maximal. Such a choice of monomials clearly exists by assumption and because the $\mathbb{Z}_{\geq 0}$-degree of any nonzero homogeneous element is bounded above by the $\mathbb{Z}_{\geq 0}$-degree of any nonzero integral in $\EuScript{B}_W$ by Proposition~\ref{prop:Aint}\eqref{item:Aint1}. If we now apply Lemma~\ref{lem:prep-concrete} to this situation, we can find monomials $y',\bar{y}'$ which do only involve $\Theta_1\cup\Theta_2$ such that
\[
\operatorname{deg}_{\mathbb{Z}_{\geq 0}}y+\operatorname{deg}_{\mathbb{Z}_{\geq 0}}\bar{y}<\operatorname{deg}_{\mathbb{Z}_{\geq 0}}y'+\operatorname{deg}_{\mathbb{Z}_{\geq 0}}\bar{y}'\,,
\]
and such that $y'y_1y_2$ and $\bar{y}'y_2y_1$ are linearly dependent and nonzero -- as claimed.
\end{proof}

\begin{rem}

Note that Lemma~\ref{lem:prep-concrete} and Theorem~\ref{thm:concrete} have right analogues which follow by application of $\rho$ and Remark~\ref{rem:rho}.

\end{rem}

\section{Reduction of monomials}
\label{sec:reduction}

Under the title of \textsc{\enquote{Reduction of monomials}}, we discuss in this section the reduction of arbitrary monomials modulo a suitable ideal (or what amounts more or less to the same modulo multiplication with a hypothetical element -- a notion which will be introduced shortly after in Section~\ref{sec:hypo}) to a monomial which does only involve $\Delta$, i.e.\ to a basis element of the standard basis of $\EuScript{N}_W$. For the moment, we achieve this reduction procedure only in type $\mathsf{A}$ using the known relations of $\EuScript{B}_{\mathbb{S}_m}$ (i.e.\ the Fomin-Kirillov relation, the nilpotent relation and the commutation relation as in \cite[Example~4.4]{skew}). The corresponding result in type $\mathsf{A}$ is stated in Corollary~\ref{cor:reduction} based on Lemma~\ref{lem:reduction}.

\begin{defn}

We denote by $\EuScript{J}_W^{\mathcalsmall{l}}$ and $\EuScript{J}_W^{\mathcalsmall{r}}$ the $\mathbb{Z}_{\geq 0}$-graded and $W$-graded left ideal and right ideal in $\EuScript{B}_W$ defined by the equalities
\[
\EuScript{J}_W^{\mathcalsmall{l}}=\textstyle{\sum_{\alpha\in R^+\setminus\Delta}}\EuScript{B}_Wx_\alpha\quad\text{and}\quad\EuScript{J}_W^{\mathcalsmall{r}}=\textstyle{\sum_{\alpha\in R^+\setminus\Delta}}x_\alpha\EuScript{B}_W\,.
\]

\end{defn}

\begin{rem}
\label{rem:ideal}

Note that we have the trivial relations $\EuScript{J}_W^{\mathcalsmall{l}}=\rho(\EuScript{J}_W^{\mathcalsmall{r}})$ and $\EuScript{J}_W^{\mathcalsmall{r}}=\rho(\EuScript{J}_W^{\mathcalsmall{l}})$ by definition of $\rho$.

\end{rem}

\begin{lem}
\label{lem:intersection}

We have $\EuScript{N}_W\cap\EuScript{J}_W^{\mathcalsmall{l}}=\EuScript{N}_W\cap\EuScript{J}_W^{\mathcalsmall{r}}=0$.

\end{lem}

\begin{proof}

The first claimed equality follows by application of $\rho$ and Remark~\ref{rem:ideal} and Fact~\ref{fact:rho_amel}\eqref{item:rho_invert}. We prove the second claimed equality now. To this end, it suffices to prove that $x_w\notin\EuScript{J}_W^{\mathcalsmall{r}}$ for all $w\in W$ because the ideal $\EuScript{J}_W^{\mathcalsmall{r}}$ is $W$-graded by definition. Suppose for a contradiction that 
\[
x_w=\textstyle{\sum_{\alpha\in R^+\setminus\Delta}}x_\alpha M_\alpha
\]
for some $w\in W$ and some monomials $M_\alpha$ of $\mathbb{Z}_{\geq 0}$-degree $\ell(w)-1$. If we apply $\constantoverleftarrow{D_{w^{-1}}}$ to the previous displayed equation, we find in view of Fact~\ref{fact:rho_amel}\eqref{item:nil_pair} and Lemma~\ref{lem:basic-rev}\eqref{item:basic3},\eqref{item:basic4} that 
\[
1=\textstyle{\sum_{\alpha\in R^+\setminus\Delta}}x_\alpha(M_\alpha)\constantoverleftarrow{D_{w^{-1}}}\,.
\]
But the right side of the last equation must be zero for $\mathbb{Z}_{\geq 0}$-degree reasons as stipulated in \cite[Remark~2.16]{skew} which is clearly a contradiction.
\end{proof}

\subsection*{From now on and for the rest of this section, we assume that $R$ is of type $\mathsf{A}$}

\begin{lem}[Key lemma on reduction of monomials]
\label{lem:reduction}

Let $\lambda\in\mathbf{k}$ and $\alpha_1,\ldots,\alpha_m\in R^+$ be such that not all $\alpha_1,\ldots,\alpha_m$ are simple roots and such that the monomial $M=\lambda x_{\alpha_1}\ldots x_{\alpha_m}$ satisfies $M+\EuScript{J}_W^{\mathcalsmall{r}}\neq 0$. Let $1\leq j<m$ be the maximal index such that $\alpha_1,\ldots,\alpha_j\in\Delta$ and such that $\alpha_{j+1}\notin\Delta$ (which exists by the assumptions in the previous sentence). Then, there exist $\lambda_i\in\mathbf{k}$ and $\alpha_1^i,\ldots,\alpha_{j+1}^i\in R^+$ such that
\begin{gather*}
M+\EuScript{J}_W^{\mathcalsmall{r}}=\textstyle{\sum_i}\lambda_ix_{\alpha_1^i}\cdots x_{\alpha_{j+1}^i}x_{\alpha_{j+2}}\cdots x_{\alpha_m}+\EuScript{J}_W^{\mathcalsmall{r}}\,,\\
\alpha_1^i,\ldots,\alpha_j^i\in\Delta\text{ and }\operatorname{ht}(\alpha_{j+1}^i)<\operatorname{ht}(\alpha_{j+1})\,.
\end{gather*}

\end{lem}

\begin{proof}

Let the notation be as in the statement. We prove the lemma by induction on $j$. Assume first that $j=1$. It is clear that $(\alpha_1,\alpha_2)>0$ since otherwise $M+\EuScript{J}_W^{\mathcalsmall{r}}=0$ by \cite[Example~4.4]{skew}. It follows that $\gamma=\alpha_2-\alpha_1$ is a positive root. If we set $\alpha=\alpha_1$, we find that $M+\EuScript{J}_W^{\mathcalsmall{r}}=x_\gamma x_\alpha x_{\alpha_3}\cdots x_{\alpha_m}+\EuScript{J}_W^{\mathcalsmall{r}}\neq 0$ by \cite[Example~4.4]{skew}. Consequently, the root $\gamma$ is simple, and we are done because $\alpha$ is simple by assumption.

We prove the induction step. Suppose that $j>1$ and that the assertion is true for all monomials and all integers $<j$. Let $w=s_{\alpha_1}\cdots s_{\alpha_j}$. We distinguish three cases now. In the first case, we assume that there exists $\beta\in\Delta$ such that $w(\beta)<0$ and such that $(\alpha_{j+1},\beta)>0$. Suppose that a $\beta$ as in the previous sentence is given. Then, we know that $\gamma=\alpha_{j+1}-\beta$ is a positive root. By repeated application of the induction hypothesis and by \cite[Example~4.4]{skew}, we compute that
\begin{align*}
M+\EuScript{J}_W^{\mathcalsmall{r}}&={}\mathrlap{x_{ws_\beta}x_\gamma x_\beta x_{\alpha_{j+2}}\cdots x_{\alpha_m}}\hphantom{\lambda x_{ws_{\beta}s_\gamma}x_\beta x_{\alpha_{j+2}}\cdots x_{\alpha_m}}-{}\mathrlap{x_{ws_{\beta}}x_{\beta+\gamma}x_\gamma x_{\alpha_{j+2}}\cdots x_{\alpha_m}}\hphantom{\mu x_{ws_{\beta}s_{\alpha_{j+1}}}x_\gamma x_{\alpha_{j+2}}\cdots x_{\alpha_m}}+\EuScript{J}_W^{\mathcalsmall{r}}\\
&=\lambda x_{ws_{\beta}s_\gamma}x_\beta x_{\alpha_{j+2}}\cdots x_{\alpha_m}+\mu x_{ws_{\beta}s_{\alpha_{j+1}}}x_\gamma x_{\alpha_{j+2}}\cdots x_{\alpha_m}+\EuScript{J}_W^{\mathcalsmall{r}}
\end{align*}
for some $\lambda,\mu\in\mathbf{k}$ (which might possibly be zero). This completes the proof of the first case because the height of $\gamma$ is strictly less than the height of $\alpha_{j+1}$.

We consider the second case. In the second case, we assume that there exists $\beta\in\Delta$ such that $w(\beta)<0$ and such that $(\alpha_{j+1},\beta)=0$. Suppose that a $\beta$ as in the previous sentence is given. By repeated application of the induction hypothesis and by \cite[Example~4.4]{skew}, we compute that
\[
M+\EuScript{J}_W^{\mathcalsmall{r}}=x_{ws_\beta}x_{\alpha_{j+1}}x_\beta x_{\alpha_{j+2}}\cdots x_{\alpha_{m}}+\EuScript{J}_W^{\mathcalsmall{r}}=\lambda x_{ws_\beta s_{\alpha_{j+1}}}x_\beta x_{\alpha_{j+2}}\cdots x_{\alpha_{m}}+\EuScript{J}_W^{\mathcalsmall{r}}
\]
for some $\lambda\in\mathbf{k}$ (which is unique and necessarily nonzero). This clearly completes the proof of the second case.

We consider the third case. In the third case, we assume that for all $\beta\in\Delta$ such that $w(\beta)<0$ we necessarily have $(\alpha_{j+1},\beta)<0$. Let us fix some arbitrary $\gamma\in\Delta$ such that $w(\gamma)<0$. Let $\alpha=\alpha_{j+1}$ for short. By repeated application of the induction hypothesis and by \cite[Example~4.4]{skew}, we compute that
\begin{align*}
M+\EuScript{J}_W^{\mathcalsmall{r}}&={}\mathrlap{x_{ws_\gamma}x_{\alpha}x_{\alpha+\gamma}x_{\alpha_{j+2}}\cdots x_{\alpha_{m}}}\hphantom{\lambda x_{ws_\gamma s_\alpha}x_{\alpha+\gamma}x_{\alpha_{j+2}}\cdots x_{\alpha_{m}}}+{}\mathrlap{x_{ws_\gamma}x_{\alpha+\gamma}x_{\gamma}x_{\alpha_{j+2}}\cdots x_{\alpha_{m}}}\hphantom{\mu x_{ws_\gamma s_{\alpha+\gamma}}x_{\gamma}x_{\alpha_{j+2}}\cdots x_{\alpha_{m}}}+\EuScript{J}_W^{\mathcalsmall{r}}\\
&=\lambda x_{ws_\gamma s_\alpha}x_{\alpha+\gamma}x_{\alpha_{j+2}}\cdots x_{\alpha_{m}}+\mu x_{ws_\gamma s_{\alpha+\gamma}}x_{\gamma}x_{\alpha_{j+2}}\cdots x_{\alpha_{m}}+\EuScript{J}_W^{\mathcalsmall{r}}
\end{align*}
for some $\lambda,\mu\in\mathbf{k}$ (which might possibly be zero). To complete the proof of the third case, we only have to analyse further the first summand in the previous equation. Let $\beta\in\Delta$ be the unique simple root such that $(\alpha,\beta)>0$ and $(\gamma,\beta)<0$. It follows that $\alpha'=\alpha-\beta$ is a positive root. (The definition of the roots $\alpha,\alpha',\beta,\gamma$ is illustrated in Figure~\ref{fig:dynkin}.) With these definitions we compute that $ws_\gamma s_\alpha(\beta)=-w(\alpha')<0$ where the first equality follows because $\alpha'$ is orthogonal to $\gamma$ and where the second inequality follows because all simple roots $\beta'$ in the support of $\alpha$ (and hence of $\alpha'$) satisfy $(\alpha,\beta')\geq 0$ and consequently $w(\beta')>0$ by the assumption in the third case under consideration. Using this insight, the orthogonality of $\beta$ and $\alpha+\gamma$, repeated application of the induction hypothesis and \cite[Example~4.4]{skew}, we compute that
\begin{align*}
x_{ws_\gamma s_\alpha}x_{\alpha+\gamma}x_{\alpha_{j+2}}\cdots x_{\alpha_{m}}+\EuScript{J}_W^{\mathcalsmall{r}}&={}\mathrlap{x_{ws_\gamma s_\alpha s_\beta}x_{\alpha+\gamma}x_\beta x_{\alpha_{j+2}}\cdots x_{\alpha_{m}}}\hphantom{\lambda x_{ws_\gamma s_\alpha s_\beta s_{\alpha+\gamma}}x_\beta x_{\alpha_{j+2}}\cdots x_{\alpha_{m}}}+\EuScript{J}_W^{\mathcalsmall{r}}\\
&=\lambda x_{ws_\gamma s_\alpha s_\beta s_{\alpha+\gamma}}x_\beta x_{\alpha_{j+2}}\cdots x_{\alpha_{m}}+\EuScript{J}_W^{\mathcalsmall{r}}
\end{align*}
for some $\lambda\in\mathbf{k}$ (which might possibly be zero). This completes the proof of the third case, of the induction step and hence of the lemma.
\end{proof}

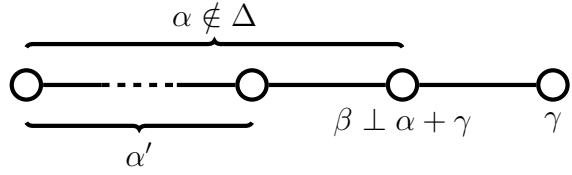
\begin{figure}
\centering
\begin{tikzpicture}[ultra thick]

\def\a{1}
\tikzset{dynkin/.style={circle,draw,minimum size=2mm}}
\path
(0,0)      node[dynkin] (N1) {} 
++(0:\a)   coordinate (A) ++ (0:\a) coordinate (B)
++(0:\a)   node[dynkin] (N2) {} 
++(0:2*\a) node[dynkin] (N3) {} +(-90:.5) node{$\beta\perp\alpha+\gamma$}
++(0:2*\a) node[dynkin] (N4) {} +(-90:.5) node{$\gamma$};

\draw[dashed] (A)--(B);
\draw (N1)--(A) (B)--(N2)--(N3)--(N4);
\draw[decorate,decoration={brace,raise=5mm},ultra thick]
(N2.center)--(N1.center) node[midway,below=6mm]{$\alpha'$};
\draw[decorate,decoration={brace,raise=5mm},ultra thick]
(N1.center)--(N3.center) node[midway,above=6mm]{$\alpha\notin\Delta$};

\end{tikzpicture}
\caption{Illustration of the situation in the third case of the induction step in the proof of Lemma~\ref{lem:reduction}. We want to thank Black Mild for providing the code for this figure in an answer to our question on \TeX~StackExchange \cite{tex}.}
\label{fig:dynkin}
\end{figure}

\begin{cor}
\label{cor:reduction0}

Let $\lambda\in\mathbf{k}$ and $\alpha_1,\ldots,\alpha_m\in R^+$ be such that the monomial 
\[
M=\lambda x_{\alpha_1}\ldots x_{\alpha_m}
\]
satisfies $M+\EuScript{J}_W^{\mathcalsmall{r}}\neq 0$. Let $0\leq j\leq m$ be an arbitrary index. Then, there exists a unique $\mu\in\mathbf{k}$ (which is necessarily nonzero) such that
\[
M+\EuScript{J}_W^{\mathcalsmall{r}}=\mu x_{s_{\alpha_1}\cdots s_{\alpha_j}}x_{\alpha_{j+1}}\cdots x_{\alpha_m}+\EuScript{J}_W^{\mathcalsmall{r}}\,.
\]

\end{cor}

\begin{proof}

Let the notation be as in the statement. By repeated application of Lemma~\ref{lem:reduction}, there exist $\lambda_i\in\mathbf{k}$ and $\alpha_1^i,\ldots,\alpha_j^i\in\Delta$ such that
\[
M+\EuScript{J}_W^{\mathcalsmall{r}}=\textstyle{\sum_i}\lambda_ix_{\alpha_1^i}\cdots x_{\alpha_j^i}x_{\alpha_{j+1}}\cdots x_{\alpha_m}+\EuScript{J}_W^{\mathcalsmall{r}}\,.
\]
Because $\EuScript{J}_W^{\mathcalsmall{r}}$ is $W$-graded by definition, and by possibly discarding summands in the above sum, we may assume that each summand in the above sum has the same $W$-degree equal to the $W$-degree of $M$, and that $\smash{x_{s_{\alpha_1}\cdots s_{\alpha_j}}=x_{\alpha_1^i}\cdots x_{\alpha_j^i}}$ for all $i$. If we set $\mu=\sum_i\lambda_i$, we have found the desired expression for $M+\EuScript{J}_W^{\mathcalsmall{r}}$ with a scalar $\mu$ which is unique (and necessarily nonzero) because $M+\EuScript{J}_W^{\mathcalsmall{r}}$ is nonzero by assumption.
\end{proof}

\begin{rem}
\label{rem:analogues}

Note that Lemma~\ref{lem:reduction} and Corollary~\ref{cor:reduction0} have obvious left analogues (i.e.\ analogues for the left ideal $\EuScript{J}_W^{\mathcalsmall{l}}$ instead of the right ideal $\EuScript{J}_W^{\mathcalsmall{r}}$) which follow by application of $\rho$ and Remark~\ref{rem:ideal}.

\end{rem}

\begin{cor}[Reduction of monomials]
\label{cor:reduction}

Let $M\in\EuScript{B}_W$ be a monomial of $W$-degree $w$. Then, there exist unique $\mu,\lambda\in\mathbf{k}$ such that
\[
M+\EuScript{J}_W^{\mathcalsmall{l}}=\mu x_w+\EuScript{J}_W^{\mathcalsmall{l}}\quad\text{and}\quad
M+\EuScript{J}_W^{\mathcalsmall{r}}=\lambda x_w+\EuScript{J}_W^{\mathcalsmall{r}}\,.
\]

\end{cor}

\begin{proof}

We prove the second equality. The first follows analogously in view of Remark~\ref{rem:analogues}. The uniqueness of $\lambda$ is clear because $x_w+\EuScript{J}_W^{\mathcalsmall{r}}$ is nonzero by Lemma~\ref{lem:intersection}. The existence of $\lambda$ follows directly from Corollary~\ref{cor:reduction0} applied to $j=m$ in case $M+\EuScript{J}_W^{\mathcalsmall{r}}\neq 0$, and is trivial in case $M+\EuScript{J}_W^{\mathcalsmall{r}}=0$ (we simply set $\lambda=0$ and this is the only $\lambda$ which fits).
\end{proof}

\begin{cor}[Isomorphism theorems]
\label{cor:iso}

We have natural isomorphisms of $\mathbb{Z}_{\geq 0}$-graded and $W$-graded vector spaces
\[
\hphantom{M+\EuScript{J}_W^{\mathcalsmall{l}}=\lambda x_w+\EuScript{J}_W^{\mathcalsmall{l}}}\mathllap{\EuScript{B}_W/\EuScript{J}_W^{\mathcalsmall{l}}\cong\EuScript{N}_W\quad}\text{and}\quad\mathrlap{\EuScript{B}_W/\EuScript{J}_W^{\mathcalsmall{r}}\cong\EuScript{N}_W\,.}\hphantom{M+\EuScript{J}_W^{\mathcalsmall{r}}=\lambda x_w+\EuScript{J}_W^{\mathcalsmall{r}}\,.}
\]

\end{cor}

\begin{proof}

This corollary is immediate from Lemma~\ref{lem:intersection} and Corollary~\ref{cor:reduction}.
\end{proof}

\begin{cor}
\label{cor:directsum}

We have two split short exact sequences of $\mathbb{Z}_{\geq 0}$-graded and $W$-graded vector spaces
\[
\hphantom{M+\EuScript{J}_W^{\mathcalsmall{l}}=\lambda x_w+\EuScript{J}_W^{\mathcalsmall{l}}}\mathllap{\EuScript{J}_W^{\mathcalsmall{l}}\to\EuScript{B}_W\to\EuScript{N}_W\quad}\text{and}\quad\mathrlap{\EuScript{J}_W^{\mathcalsmall{r}}\to\EuScript{B}_W\to\EuScript{N}_W\,.}\hphantom{M+\EuScript{J}_W^{\mathcalsmall{r}}=\lambda x_w+\EuScript{J}_W^{\mathcalsmall{r}}\,,}
\]
and consequently two direct sum decompositions in the same sense as above:
\[
\hphantom{M+\EuScript{J}_W^{\mathcalsmall{l}}=\lambda x_w+\EuScript{J}_W^{\mathcalsmall{l}}}\mathllap{\EuScript{B}_W\cong\EuScript{N}_W\oplus\EuScript{J}_W^{\mathcalsmall{l}}\quad}\text{and}\quad\mathrlap{\EuScript{B}_W\cong\EuScript{N}_W\oplus\EuScript{J}_W^{\mathcalsmall{r}}}\hphantom{M+\EuScript{J}_W^{\mathcalsmall{r}}=\lambda x_w+\EuScript{J}_W^{\mathcalsmall{r}}\,}
\]
    
\end{cor}

\begin{proof}
This is immediate form Corollary~\ref{cor:iso}. Alternatively, one can see at least the directness of the sums in the statement directly from Lemma~\ref{lem:intersection}  
\end{proof}

\begin{rem}
\label{rem:generalize-iso}

We expect the results in this section -- and the general results in the sections to follow which are only restricted to type $\mathsf{A}$ in as much as they depend on reduction of monomials in this type -- not only to hold in type $\mathsf{A}$ but for arbitrary $W$. The main obstacle in proving a general analogue of reduction of monomials is that (some of) the relations of $\EuScript{B}_W$ are not explicitly known unless $R$ is simply laced. However, in \cite[Definition~2.1(iv)]{kirillov-maeno-noncommutative}, there are $4$-term relations for non simply laced types proposed which are left to be verified in $\EuScript{B}_W$ and which might be useful in realizing generalized isomorphism theorems in the sense of Corollary~\ref{cor:iso}.

\end{rem}

\begin{rem}
\label{rem:generalize-nilcoxeter}

Corollary~\ref{cor:iso} is also important as a guideline for how to possibly define a nilCoxeter analogue for Nichols algebras over arbitrary groups, not necessarily Coxeter groups. 
    
\end{rem}

\section{Hypothetical elements}
\label{sec:hypo}

In this section, we use the results in Section~\ref{sec:reduction}, to show in Theorem~\ref{thm:main-hypo} how nonzero left or right hypothetical elements (cf.~Definition~\ref{def:hypo}) in type $\mathsf{A}$ can be lifted to nonzero integrals in $\EuScript{B}_W$ under the assumption that $\EuScript{B}_W$ is finite dimensional. As a consequence of this, we can derive Theorem~\ref{thm:subalgebra} and Theorem~\ref{thm:hypo-antipode} where the latter treats the invariance of hypothetical elements under the antipode. 

\begin{defn}
\label{def:hypo}

Let $P\in\EuScript{B}_W$ be such that $(P)\constantoverleftarrow{D_\beta}=0$ for all $\beta\in\Delta$.

\begin{itemize}
    \item 
    We say that $P$ is a left hypothetical element if $x_\alpha P=0$ for all $\alpha\in R^+\setminus\Delta$. We say that $P$ is a nonzero left hypothetical element if $P$ is nonzero and if $P$ is a left hypothetical element.
    \item
    We say that $P$ is a right hypothetical element if $Px_\alpha=0$ for all $\alpha\in R^+\setminus\Delta$. We say that $P$ is a nonzero right hypothetical element if $P$ is nonzero and if $P$ is a right hypothetical element.
    \item
    We say that $P$ is a hypothetical element if $P$ is a left and right hypothetical element. We say that $P$ is a nonzero hypothetical element if $P$ is nonzero and if $P$ is a hypothetical element.
\end{itemize}

\end{defn}

\begin{rem}
\label{rem:existence}

If $\EuScript{B}_W$ is finite dimensional, it is clear that there exist monomials which do only involve $R^+\setminus\Delta$ and which are nonzero hypothetical elements. Indeed, similar as in the proof of Corollary~\ref{cor:BWintegral}\eqref{item:cor:M}, one considers nonzero monomials in the subalgebra of $\EuScript{B}_W$ generated by $x_\alpha$ where $\alpha\in R^+\setminus\Delta$ of maximal $\mathbb{Z}_{\geq 0}$-degree (cf.~Lemma~\ref{lem:basic-rev}\eqref{item:basic3}).

\end{rem}

\begin{lem}
\label{lem:left-to-right-hypo}

An element $P\in\EuScript{B}_W$ is a left hypothetical element if and only if $\rho(P)$ is a right hypothetical element. An element $P\in\EuScript{B}_W$ is a right hypothetical element if and only if $\rho(P)$ is a left hypothetical element.

\end{lem}

\begin{proof}

This result follows immediately from Corollary~\ref{cor:vanish_rho}, \cite[Proposition~3.7(6)]{skew} and the definition of $\rho$.
\end{proof}

\begin{cor}
\label{cor:rhohypo}

Let $P$ be a nonzero hypothetical element. Then, $\rho(P)$ and $w_oP$ are in each case again nonzero hypothetical elements. 

\end{cor}

\begin{proof}

The assertion for $\rho(P)$ follows directly from Lemma~\ref{lem:left-to-right-hypo}. The assertion for $w_oP$ is obvious since $w_oSw_o=S$ by \cite[Section~5.6, Exercise~2]{humphreys-coxeter}.
    
\end{proof}

\begin{lem}
\label{lem:lift-with-xwo}

Let $P\in\EuScript{B}_W$ be a nonzero element such that $(P)\constantoverleftarrow{D_\beta}=0$ for all $\beta\in\Delta$. Then, we have $x_{w_o}P\neq 0$ and $Px_{w_o}\neq 0$.

\end{lem}

\begin{proof}

Let the notation be as in the statement. In order to proof the lemma, it suffices to proof that $(x_{w_o}P)\constantoverleftarrow{D_{w_o}}\neq 0$ and $(Px_{w_o})\constantoverleftarrow{D_{w_o}}\neq 0$. But, by Fact~\ref{fact:rho_amel}\eqref{item:nil_pair}, Lemma~\ref{lem:basic-rev}\eqref{item:basic4} and Corollary~\ref{cor:bcommuteswithskew}, those expressions evaluate as $w_oP$ and $P$, respectively, which are both nonzero by assumption.
\end{proof}

\begin{cor}
\label{cor:lift-with-xwo}

Let $P\in\EuScript{B}_W$ be a nonzero element which does only involve $R^+\setminus\Delta$. Then, we have $x_{w_o}P\neq 0$ and $Px_{w_o}\neq 0$.
    
\end{cor}

\begin{proof}

This follows immediately from Lemma~\ref{lem:basic-rev}\eqref{item:basic3} and Lemma~\ref{lem:lift-with-xwo}.
\end{proof}

\subsection*{From now on and for the rest of this section, we assume that $R$ is of type $\mathsf{A}$}

\begin{thm}[Main theorem on hypothetical elements~(I)]
\label{thm:main-hypo}

\leavevmode
\begin{enumerate}
    \item\label{item:hypo-l}
    Let $P\in\EuScript{B}_W$ be a nonzero left hypothetical element. If $\EuScript{B}_W$ is finite dimensional, then $x_{w_o}P$ is a nonzero integral in $\EuScript{B}_W$
    \item\label{item:hypo-r}
    Let $P\in\EuScript{B}_W$ be a nonzero right hypothetical element. If $\EuScript{B}_W$ is finite dimensional, then $Px_{w_o}$ is a nonzero integral in $\EuScript{B}_W$
\end{enumerate}

\end{thm}

\begin{proof}

Item~\eqref{item:hypo-l} follows from Item~\eqref{item:hypo-r}, Fact~\ref{fact:rho_amel}\eqref{item:rho_invert}, Lemma~\ref{lem:left-to-right-hypo} and \cite[Proposition~3.7(6)]{skew}. We prove Item~\eqref{item:hypo-r} now. Let $P\in\EuScript{B}_W$ be a nonzero right hypothetical element, and suppose that $\EuScript{B}_W$ is finite dimensional. Let $M$ be a monomial chosen as in the proof of Proposition~\ref{prop:Aint}\eqref{item:Aint3} and Corollary~\ref{cor:BWintegral}\eqref{item:cor:M} with respect to the nonzero component of $P$ of smallest $\mathbb{Z}_{\geq 0}$-degree such that $PM$ is a nonzero integral in $\EuScript{B}_W$. Let $m$ be the $\mathbb{Z}_{\geq 0}$-degree of $M$, and let $w$ be the $W$-degree of $M$. Then, we necessarily have $\ell(w_o)\leq m$ by Lemma~\ref{lem:lift-with-xwo}, and further $PM=\lambda Px_w$ for some unique $\lambda\in\mathbf{k}$ (which is necessarily nonzero) by Corollary~\ref{cor:reduction}. It follows that $w=w_o$ and hence that $Px_{w_o}$ is a nonzero integral in $\EuScript{B}_W$ -- as desired.
\end{proof}

\begin{thm}[Main theorem on hypothetical elements~(II)]
\label{thm:hypo-II}

Suppose that $\EuScript{B}_W$ is finite dimensional.

\begin{enumerate}
    \item\label{item:II-one}
    Every left or right hypothetical element is a hypothetical element. Following and consistent with the convention in Remark~\ref{rem:bona-fide}, we speak from now on, except in the proof of this theorem, only about bona fide hypothetical elements.
    \item\label{item:II-two}
    Every hypothetical element $P$ satisfies $P=(x)\constantoverleftarrow{D_{w_o}}$ for some uniquely determined integral $x$, namely $x=Px_{w_o}$, in $\EuScript{B}_W$. In particular, the vector space of all hypothetical elements is one-dimensional.
    \item
    We have $\left<P^*,P\right>\neq 0$ for any nonzero hypothetical elements $P$ and $P^*$.
    \item\label{item:II-three}
    Every hypothetical element $P$ is homogeneous of $W$-degree $w_o$.
\end{enumerate}

\end{thm}

\begin{proof}

We start to prove the first sentence of Item~\eqref{item:II-two} for right hypothetical elements. The rest of the content of Item~\eqref{item:II-two} will follow from Remark~\ref{rem:existence} and Proposition~\ref{prop:Aint}\eqref{item:Aint1} once Item~\eqref{item:II-one} is established. Note first that the uniqueness of $x$, once its existence is established, is clear from Proposition~\ref{prop:Nicholsint}\eqref{item:Nicholsint2}. Note further that we may assume that $P\neq 0$ since otherwise we have to set $x=0$ and the claimed assertion is obvious. In this sense, let $P$ be a nonzero right hypothetical element. We know from Theorem~\ref{thm:main-hypo}\eqref{item:hypo-r} that $x=Px_{w_o}$ is a nonzero integral in $\EuScript{B}_W$ which satisfies $P=(x)\constantoverleftarrow{D_{w_o}}$ by Fact~\ref{fact:rho_amel}\eqref{item:nil_pair} and Lemma~\ref{lem:basic-rev}\eqref{item:basic4}. In order to prove Item~\eqref{item:II-one}, we only have to note, in view of Lemma~\ref{lem:left-to-right-hypo} and \cite[Proposition~3.7(6)]{skew}, that for the $P$ under consideration, we have $P\sim\rho(P)$. But this follows from Fact~\ref{fact:rho_amel}\eqref{item:mult_wo}, Proposition~\ref{prop:rhoD}, Corollary~\ref{cor:BWintegral}\eqref{item:sign}, the portion of Item~\eqref{item:II-two} we already proved and \cite[Remark~3.16, Proposition~6.5]{skew}. To understand the last item, we just have to note that for two nonzero hypothetical elements $P$ and $P^*$ the bracket $\left<P^*,P\right>$ is a nonzero multiple of $\left<(x)\constantoverleftarrow{D_{w_o}},(x)\constantoverleftarrow{D_{w_o}}\right>$ for some nonzero integral $x$ in $\EuScript{B}_W$ -- by Proposition~\ref{prop:Aint}\eqref{item:Aint1} and Item~\eqref{item:II-two} -- which is indeed nonzero by Corollary~\ref{cor:hypo-bracket}. Item~\eqref{item:II-three} follows directly from Item~\eqref{item:II-two}, Lemma~\ref{lem:char} and Remark~\ref{rem:center}.
\end{proof}

\begin{defn}
\label{def:subalgebra}

Let $\EuScript{B}_W'$ be the subalgebra of $\EuScript{B}_W$ generated by $x_\alpha$ where $\alpha\in R^+\setminus\Delta$. The algebra $\EuScript{B}_W'$ is clearly $\mathbb{Z}_{\geq 0}$-graded because its generators are homogeneous of $\mathbb{Z}_{\geq 0}$-degree one. If $\EuScript{B}_W$ is finite dimensional, then so is $\EuScript{B}_W'$, and one can look at the nonzero component of $\EuScript{B}_W'$ of largest $\mathbb{Z}_{\geq 0}$-degree, which we denote, in this situation, by $\EuScript{B}_W^{\prime\,\mathrm{top}}$ from now on.

\end{defn}

\begin{rem}
\label{rem:action}

We make the important remark that $\EuScript{B}_W'$ carries a left action of $\EuScript{B}_W$ defined by $y\mapsto\constantoverrightarrow{D_y}$. Indeed, this follows directly by definition of the coproduct which was explained in detail in \cite[Section~5]{skew}.

\end{rem}

\begin{thm}
\label{thm:subalgebra}

Suppose that $\EuScript{B}_W$ is finite dimensional.

\begin{enumerate}
    \item 
    Every one-dimensional left or right ideal in $\EuScript{B}_W'$ equals $\EuScript{B}_W^{\prime\,\mathrm{top}}$. In particular, we know that $\EuScript{B}_W^{\prime\,\mathrm{top}}$ is a one-dimensional $\mathbb{Z}_{\geq 0}$-graded two-sided ideal in $\EuScript{B}_W'$.
    \item
    The vector space of all hypothetical elements is given by $\EuScript{B}_W^{\prime\,\mathrm{top}}$ and all elements of this space are monomials which do only involve $R^+\setminus\Delta$.
\end{enumerate}

\end{thm}

\begin{proof}

It is clear that $\EuScript{B}_W^{\prime\,\mathrm{top}}$ consists of hypothetical elements, and further, that any one-dimensional left or right ideal is generated by a nonzero hypothetical element, cf.~Lemma \ref{lem:basic-rev}\eqref{item:basic3}. The result follows from this and Theorem~\ref{thm:hypo-II}.
\end{proof}

\begin{thm}
\label{thm:rightvanish}

Suppose that $\EuScript{B}_W$ is finite dimensional. Let $P$ be a hypothetical element. Then, $\constantoverrightarrow{D_\beta}(P)=0$ for all $\beta\in\Delta$.
    
\end{thm}

\begin{proof}
Suppose for a contradiction that $\constantoverrightarrow{D_\beta}(P)\neq 0$ for some $\beta\in\Delta$. Let $\alpha\in R^+\setminus\Delta$. By the braided Leibniz rule and the fact that $P$ is a hypothetical element, we find that 
\[
0=\constantoverrightarrow{D_\beta}(Px_\alpha)=\constantoverrightarrow{D_\beta}(P)x_\alpha+P\constantoverrightarrow{D_{w_o(\beta)}}(x_\alpha)=\constantoverrightarrow{D_\beta}(P)x_\alpha
\]
where $w_o$ is the $W$-degree of $P$ after Theorem~\ref{thm:hypo-II}\eqref{item:II-three}. Taking into account Remark~\ref{rem:action}, it follows that $\constantoverrightarrow{D_\beta}(P)\EuScript{B}_W'=\mathbf{k}\constantoverrightarrow{D_\beta}(P)$ is a one-dimensional right ideal in $\EuScript{B}_W'$ which is clearly a contradiction to Theorem~\ref{thm:subalgebra} as $P$ and $\constantoverrightarrow{D_\beta}(P)$ are linearly independent. 
\end{proof}

\begin{thm}[Invariance of hypothetical elements under the antipode]
\label{thm:hypo-antipode}

Suppose that $\EuScript{B}_W$ is finite dimensional. Let $P$ be a nonzero hypothetical element. Then, $P\sim\EuScript{S}(P)\sim\bar{\EuScript{S}}(P)$. 

\end{thm}

\begin{proof}

By Corollary~\ref{cor:rhohypo}, it suffices in view of Theorem~\ref{thm:subalgebra} to prove that $\EuScript{S}(P)$ is a nonzero hypothetical element. Everything else follows from this. It is obvious that $\EuScript{S}(P)$ satisfies the left decent condition since $\bar{\EuScript{S}}(P)x_{\alpha}=\bar{\EuScript{S}}(P(w_ox_{\alpha}))=0$ where $w_o$ is the $W$-degree of $P$ after Theorem~\ref{thm:hypo-II}\eqref{item:II-three} and where $\alpha\in R^+\setminus\Delta$. Since every left hypothetical is a hypothetical element per se, it suffices to prove the vanishing $({\EuScript{S}}(P))\constantoverleftarrow{D_\beta}=-s_{\beta}\EuScript{S}\constantoverrightarrow{D_\beta}(P)=0$ for all $\beta\in\Delta$ where the last equality follows from Proposition~\ref{prop:rhoD}. But this vanishing is immediate from Theorem~\ref{thm:rightvanish}. 
\end{proof}

\section{Product decomposition after Liu et al.}

If we compare our results of the two previous sections with the results of \cite{liusubalgebras}, specifically \cite[Theorem~3.14, Theorem~4.1]{liusubalgebras}, we see that they are in some sense an exemplification of theirs, however an exemplification which is due to its explicitness useful and needed for our purpose. In a similar vein and inspired by \cite{liusubalgebras}, we want to derive in this section one product decomposition, namely Theorem~\ref{thm:productdecomposition}, see also Corollary~\ref{cor:tensorproduct}, in the sense of \cite[Theorem~4.1]{liusubalgebras} of $\mathcal{B}_W$ if $R$ is of type $\mathsf{A}$ which is of particular importance -- and further derive consequences.

\subsection*{From now on and for the rest of this section, we assume that $R$ is of type $\mathsf{A}$}

\begin{thm}[{\cite[Theorem~4.1]{liusubalgebras}}]
\label{thm:productdecomposition}

Let $\mathcal{B}_W'^+$ be the positive part of the subalgebra $\mathcal{B}_W'$ defined in Definition~\ref{def:subalgebra}, i.e. the direct sum of all graded components of $\mathbb{Z}_{\geq 0}$-degree greater or equal than one. Then, we have $\EuScript{J}_W^{\mathcalsmall{l}}=\EuScript{N}_W\EuScript{B}_W'^+=\EuScript{B}_W\EuScript{B}_W'^+$ and $\EuScript{J}_W^{\mathcalsmall{r}}=\EuScript{B}_W'^+\EuScript{N}_W=\EuScript{B}_W'^+\EuScript{B}_W$, and consequently $\EuScript{B}_W=\EuScript{N}_W\EuScript{B}_W'=\EuScript{B}_W'\EuScript{N}_W$.

\end{thm}

\begin{proof}
It is clear that it suffices to prove the assertions for $\EuScript{J}_W^{\mathcalsmall{r}}$ as the left handed version follows by application of $\rho$. To this end, we plug in the direct sum decomposition from Corollary~\ref{cor:directsum} repeatedly into the definition of $\EuScript{J}_W^{\mathcalsmall{r}}$ and find
\[
\EuScript{J}_W^{\mathcalsmall{r}}=\sum_{\alpha\in R^+\setminus\Delta}x_\alpha\EuScript{N}_W+\sum_{\alpha\in R^+\setminus\Delta}x_\alpha\EuScript{J}_W^{\mathcalsmall{r}}=\sum_{\alpha\in R^+\setminus\Delta}x_\alpha\EuScript{N}_W+\sum_{\alpha_1,\alpha_2\in R^+\setminus\Delta}x_{\alpha_1}x_{\alpha_2}\EuScript{N}_W+\cdots
\]
which means precisely that $\EuScript{J}_W^{\mathcalsmall{r}}=\EuScript{B}_W'^+\EuScript{N}_W$. Since $\EuScript{J}_W^{\mathcalsmall{r}}$ is a right ideal by definition, the other claimed equality for $\EuScript{J}_W^{\mathcalsmall{r}}$ follows. Finally, the product decomposition of $\EuScript{B}_W$ is clear from this and Corollary~\ref{cor:directsum}. 
\end{proof}

\begin{cor}[{\cite[Theorem~4.1]{liusubalgebras}}]
\label{cor:tensorproduct}

The product decomposition of $\EuScript{B}_W$ in the previous theorem is even a tensor product decomposition, i.e.
\[
\EuScript{B}_W=\EuScript{N}_W\otimes\EuScript{B}_W'=\EuScript{B}_W'\otimes\EuScript{N}_W\,.
\] 

\end{cor}

\begin{proof}

To prove the claimed tensor product decompositions of $\EuScript{B}_W$, it clearly suffices to prove that
\[
\EuScript{B}_W=\bigoplus_{w\in W}\EuScript{B}_W'\otimes\mathbf{k}x_w\,.
\]
To this end, let $\sum_{w\in W}a_w\otimes x_w=0$, where $a_w\in\EuScript{B}_W'$. Suppose for a contradiction that there exists $w\in W$ with $a_w\neq 0$. Then, choose and fix an element $v\in W$ with maximal length such that $a_v\neq 0$. If we apply $\constantoverleftarrow{D_{v^{-1}}}$ to the expression $\sum_{w\in W}a_w\otimes x_w=0$, we find, in view of Lemma~\ref{lem:basic-rev}\eqref{item:basic3},\eqref{item:basic4}, that $\sum_{w\in W}a_w\otimes(x_w)\constantoverleftarrow{D_{v^{-1}}}=0$. For degree reasons and the choice of $v$ and by Fact~\ref{fact:rho_amel}\eqref{item:nil_pair}, all terms in the previous sum vanish except the one for $w=v$. Hence, we find that $a_v=0$ -- a contradiction.    
\end{proof}

\begin{cor}

Let $w\in W$. Let $y=wx_{w_o}$ and let $x$ be an arbitrary element in $\EuScript{B}_W$. Then, the element $(x)\constantoverleftarrow{D_y}$ does only involve $R^+\setminus T_w$.
    
\end{cor}

\begin{proof}
Let the notation be as in the statement. By Theorem~\ref{thm:productdecomposition} applied to $\EuScript{B}_W=w\EuScript{B}_W$, we can write $x=\sum_{u\in W}a_u(wx_u)$ where $a_u\in w\EuScript{B}_W'$, i.e.\ $a_u$ does only involve $R^+\setminus T_w$ but is not necessarily a monomial.  If we apply $\constantoverleftarrow{D_y}$ to this expression of $x$, we find in view of Lemma~\ref{lem:basic-rev}\eqref{item:basic4} that $(x)\constantoverleftarrow{D_y}=a_{w_o}$ which does only involve roots in $R^+\setminus T_w$. This completes the proof.    
\end{proof}

\begin{cor}
Let $D$ be a complete disjoint system. Then, we have
\[
\bigcap_{w\in D}w\EuScript{B}_W'=\mathbf{k}1\,.
\]
\end{cor}

\begin{proof}
Let $x$ be an element in the intersection mentioned in the statement. We have to show that $x$ is constant. Since $x\in w\EuScript{B}_W'$, we know that $(x)\constantoverleftarrow{D_\alpha}=0$ for all $\alpha\in T_w$. Since $T=\bigcup_{w\in D}wSw^{-1}$ by assumption, we infer that even $(x)\constantoverleftarrow{D_\alpha}=0$ for all $\alpha\in R^+$. By the nondegeneracy of the Hopf duality pairing (cf.~\cite[Criterion~3.6]{bazlov1}), we know that $x$ is constant.
\end{proof}

\begin{cor}
\label{cor:nondegenerate_sub}

The Hopf duality pairing restricted to $\EuScript{B}_W'\otimes\EuScript{B}_W'$ is nondegenerate.
    
\end{cor}

\begin{proof}
Let $a\in\EuScript{B}_W'$. Suppose that $\left<a,b\right>=0$ for all $b\in\EuScript{B}_W'$. We have to show that $a=0$. To this end, it suffices in view of Theorem~\ref{thm:productdecomposition} to prove that $\left<a,x_ub\right>=0$ for all $u\in W$ and all $b\in\EuScript{B}_W'$. But this is obvious by assumption because the bracket under consideration equals $\left<(a)\constantoverleftarrow{D_u},b\right>$ and $(a)\constantoverleftarrow{D_u}=0$ whenever $u\neq 1$.  
\end{proof}

\begin{cor}
\label{cor:nondegenerate_cor}

Let $a\in\EuScript{B}_W'$ be such that $\constantoverrightarrow{D_\alpha}(a)=0$ for all $\alpha\in R^+\setminus\Delta$ or that $(a)\constantoverleftarrow{D_\alpha}=0$ for all $\alpha\in R^+\setminus\Delta$. Then, $a$ is constant.
    
\end{cor}

\begin{proof}

Let the notation be as in the statement and suppose that $a$ is nonconstant. Then, we can find a nonzero graded piece $a^m\in\EuScript{B}_W'$ of $a$ of positive $\mathbb{Z}_{\geq 0}$-degree $m$ such that $\constantoverrightarrow{D_\alpha}(a^m)=0$ for all $\alpha\in R^+\setminus\Delta$ or that $(a^m)\constantoverleftarrow{D_\alpha}=0$ for all $\alpha\in R^+\setminus\Delta$ -- which in each case entails $\left<a^m,b\right>=0$ for all $b\in\EuScript{B}_W'$. By Corollary~\ref{cor:nondegenerate_sub}, it follows that $a^m=0$ -- a contradiction.
\end{proof}

\begin{cor}
\label{cor:innilcoxeter}

Let $a\in\EuScript{B}_W$ be such that $(a)\constantoverleftarrow{D_\alpha}=0$ for all $\alpha\in R^+\setminus\Delta$. Then, $a\in\EuScript{N}_W$.

\end{cor}

\begin{proof}
Indeed, let us write $a$ as $a=\sum_{w\in W}x_wa_w$ where $a_w\in\EuScript{B}_W'$ in view of Theorem~\ref{thm:productdecomposition}. Without loss of generality, we may assume that $a_w$ is homogeneous with respect to the $\mathbb{Z}_{\geq 0}$-grading for all $w\in W$. If for all $w\in W$ the element $a_w$ is constant, then we are done. Suppose for a contradiction that there exists $w$ such that $a_w$ is nonconstant. Then, choose and fix an element $v\in W$ with maximal (and thus positive) $\mathbb{Z}_{\geq 0}$-degree such that $a_v$ is nonconstant. Since $a_v\in\EuScript{B}_W'$ is in particular nonzero, we can find a dual $a_v^*\in\EuScript{B}_W'$ which is homogeneous of positive $\mathbb{Z}_{\geq 0}$-degree such that $\left<a_v,a_v^*\right>\neq 0$ by Corollary~\ref{cor:nondegenerate_sub}. If we apply $\constantoverleftarrow{D_{a_v*}}$ to $a$, we find in view of the vanishing assumption on $a$ and Lemma~\ref{lem:basic-rev}\eqref{item:basic3},\eqref{item:basic4}, that $\sum_wx_w\left<a_w,a_v^*\right>=0$ where the sum ranges over all $w\in W$ such that the $\mathbb{Z}_{\geq 0}$-degree of $a_w$ equals the $\mathbb{Z}_{\geq 0}$-degree of $a_v^*$. Since the standard basis of the nilCoxeter algebra is linear independent, it follows in particular that $\left<a_v,a_v^*\right>=0$ which is a contradiction.
\end{proof}

\begin{cor}

Let $a$ be an arbitrary element in $\EuScript{B}_W$. Suppose that $(a)\constantoverleftarrow{D_\beta}=0$ for all $\beta\in\Delta$. Then, the element $a$ does only involve $R^+\setminus\Delta$.    
\end{cor}

\begin{proof}

Indeed, let us write $a$ as $\sum_{w\in W}a_wx_w$ where $a_w\in\EuScript{B}_W'$ in view of Theorem~\ref{thm:productdecomposition}. If $a=0$ or $a_w=0$ for all $w\in W\setminus\{1\}$, we are done. So, let us assume for a contradiction that there exists $w$ of positive length such that $a_w\neq 0$. Then, choose and fix an element $v\in W$ with maximal (and thus positive) length such that $a_v\neq 0$. If we apply $\constantoverleftarrow{D_{v^{-1}}}$ to $a$, we find, in view of Lemma~\ref{lem:basic-rev}\eqref{item:basic3},\eqref{item:basic4}, that $\sum_{w\in W}a_w(x_w)\constantoverleftarrow{D_{v^{-1}}}=0$. For degree reasons and the choice of $v$ and by Fact~\ref{fact:rho_amel}\eqref{item:nil_pair}, all terms in the previous sum vanish except the one for $w=v$ which evaluates as $a_v$. Hence, we find that $a_v=0$ -- a contradiction.   
\end{proof}

\section{Invariance of hypothetical elements}
\label{sec:invhypo}

We understand by invariance of hypothetical elements the analogue properties as in Section~\ref{sec:invariance} but with a nonzero integral in $\EuScript{B}_W$ replaced by a nonzero hypothetical element, i.e.\ roughly speaking by an integral in $\EuScript{B}_W'$ although the later is not necessarily a braided Hopf algebra. We derive one such invariance property of hypothetical elements in type $\mathsf{A}$, namely Theorem~\ref{thm:inv-hypo}, which is key and which is the analogue of Lemma~\ref{lem:invariance-prop}\eqref{item:inv3} in type $\mathsf{A}$ for $\EuScript{B}_{\mathbb{S}_m}$ replaced by $\EuScript{B}_{\mathbb{S}_m}'$.

\subsection*{From now on and for the rest of this section, we assume that $R$ is of type $\mathsf{A}$}

\begin{thm}
\label{thm:inv-hypo}

Suppose that $\EuScript{B}_W$ is finite dimensional. Let $P$ be a nonzero hypothetical element in $\EuScript{B}_W'$.  Let $D$ be a normalized disjoint system. Let $w\in D\setminus\{1\}$. Let $y=wx_{w_o}$. Then, $(P)\constantoverleftarrow{D_y}y$ is again a nonzero hypothetical element. 

\end{thm}

\begin{proof}

Let the notation be as in the statement. Let us first prove the claim that $(P)\constantoverleftarrow{D_y}y$ is nonzero. Indeed, let $x$ be a nonzero integral in $\EuScript{B}_W$ such that $P=(x)\constantoverleftarrow{D_{w_o}}$ -- which exists by Theorem~\ref{thm:hypo-II}\eqref{item:II-two}. By Lemma~\ref{lem:invariance-prop}\eqref{item:inv5}, we know further that $x=(x)\constantoverleftarrow{D_{x_{w_o}y}}yx_{w_o}=(P)\constantoverleftarrow{D_y}yx_{w_o}$ is nonzero -- and consequently that the portion $(P)\constantoverleftarrow{D_y}y$ of this expression is also nonzero. This proves the first assertion.  

By Theorem~\ref{thm:hypo-antipode}, we know that $(P)\constantoverleftarrow{D_y}y$ is a nonzero hypothetical element if and only if $\bar{\EuScript{S}}((P)\constantoverleftarrow{D_y}y)$ is a nonzero hypothetical element. Therefore, for $\mathbb{Z}_{\geq 0}$-degree reasons and the previous proved claim, it suffices to prove that 
\[
\bar{\EuScript{S}}((P)\constantoverleftarrow{D_y}y)=\bar{\EuScript{S}}((P)\constantoverleftarrow{D_y})y=\constantoverrightarrow{D_y}(\bar{\EuScript{S}}(P))y\sim\constantoverrightarrow{D_y}(P)y
\]
lies in $\EuScript{B}_W'$, where we used in the previous reformulation \cite[Proposition~3.7(4), Proposition~4.2, Proposition~6.5]{skew}, Theorem~\ref{thm:hypo-II}\eqref{item:II-three} and Theorem~\ref{thm:hypo-antipode}. But this is indeed obvious by Remark~\ref{rem:action}.
\end{proof}

\section{Conclusion}

In this section, we prove the main theorem of this investigation, namely Theorem~\ref{thm:mainS6}, basically based on all the results derived so far.

\subsection*{From now on and for the rest of this section, we assume that $R$ is of type $\mathsf{A}$}

\begin{lem}[Key lemma for the conclusion]
\label{lem:keyconclusion}

Let $\alpha\in R^+\setminus\Delta$. Let $a\in\EuScript{B}_W$ be such that $(a)\constantoverleftarrow{D_\alpha}=0$. By Theorem~\ref{thm:productdecomposition}, we can write $a=\sum_{w\in W}x_wa_w$ where $a_w\in\EuScript{B}_W'$. Then, we have $(a_w)\constantoverleftarrow{D_\alpha}=0$ for all $w\in W$. 
    
\end{lem}

\begin{proof}

Let the notation be as in the statement. Without loss of generality, we can assume that $a$ and $a_w$ are homogeneous with respect to the $\mathbb{Z}_{\geq 0}$-grading for all $w\in W$. Suppose for a contradiction that there exists $w\in W$ such that $(a_w)\constantoverleftarrow{D_\alpha}\neq 0$. Then, choose and fix an element $v\in W$ such that $a_v$ is of maximal $\mathbb{Z}_{\geq 0}$-degree such that $(a_v)\constantoverleftarrow{D_\alpha}\neq 0$. By Corollary~\ref{cor:nondegenerate_sub}, we can find a dual $a_v^*\in\EuScript{B}_W'$ which starts with $\alpha$ and which is homogeneous with respect to the $\mathbb{Z}_{\geq 0}$-degree such that $\left<a_v.a_v^*\right>\neq 0$. If we apply $\constantoverleftarrow{D_{a_v^*}}$ to $a$, we find by the vanishing assumption on $a$ and by Lemma~\ref{lem:basic-rev}\eqref{item:basic3},\eqref{item:basic4} that $\sum_w x_w\left<a_w,a_v^*\right>=0$ where the sum ranges over all $w\in W$ such that the $\mathbb{Z}_{\geq 0}$-degree of $a_w$ equals the $\mathbb{Z}_{\geq 0}$-degree of $a_v^*$. Since the standard basis of the nilCoxeter algebra is linear independent, it follows in particular that $\left<a_v,a_v^*\right>=0$ which is a contradiction. 
\end{proof}

\subsection*{From now on and for the rest of this section, we assume that $R$ is of type $\mathsf{A}_5$}

\begin{thm}
\label{thm:mainS6}

The algebra $\EuScript{B}_{\mathbb{S}_6}$ is infinite dimensional.
    
\end{thm}

\begin{proof}

Assume for a contradiction that $\EuScript{B}_W$ is finite dimensional. Let $D$ be a normalized complete disjoint system in $W$, which exists by Example~\ref{ex:S6}. Let $w_0,w_1,w_2$ be some ordering of the elements of $D$ such that $w_0=1$. Let $y_0=w_0x_{w_o}$, $y_1=w_1x_{w_o}$, $y_2=w_2x_{w_o}$. Let $P$ be a nonzero hypothetical element in $\EuScript{B}_W'$, which exists by our assumption about the finite dimensionality of $\EuScript{B}_W$ and Theorem~\ref{thm:hypo-II}\eqref{item:II-two}.  

Let $a=(P)\constantoverleftarrow{D_{y_1y_2}}$. By application of invariance of hypothetical elements twice, i.e.\ Theorem~\ref{thm:inv-hypo}, and by Corollary~\ref{cor:ofbskew}
\[
P=(P)\constantoverleftarrow{D_{y_1}}y_1\constantoverleftarrow{D_{y_2}}y_2=(-1)^{\ell(w_o)}ay_1y_2\,,
\]
in particular, we know that $ay_1y_2$ is a nonzero hypothetical element.
Let $x=x_{w_o}P$. Then, we know by Theorem~\ref{thm:main-hypo}\eqref{item:hypo-l} that $x$ is a nonzero integral in $\EuScript{B}_W$. With this definition, we evaluate $y_0ay_1$ as 
\[
(x)\constantoverleftarrow{D_{y_1y_2}}y_1=(-1)^{\ell(w_o)}(x)\constantoverleftarrow{D_{y_1}}y_1\constantoverleftarrow{D_{y_2}}=(-1)^{\ell(w_o)}(x)\constantoverleftarrow{D_{y_2}}
\]
by Lemma~\ref{lem:basic-rev}\eqref{item:basic3},\eqref{item:basic4}, by Corollary~\ref{cor:ofbskew} and by Lemma~\ref{lem:invariance-prop}\eqref{item:inv3} on the one hand; and as $y_0a_1y_1$ where $a_1\in\EuScript{B}_W'$ appears in the product decomposition $a=\sum_{w\in W}x_wa_w$ with $a_w\in\EuScript{B}_W'$ according to Theorem~\ref{thm:productdecomposition} on the other hand. In total, we thus have $y_0a_1y_1=(-1)^{\ell(w_o)}(x)\constantoverleftarrow{D_{y_2}}$.
If we know apply $\constantoverleftarrow{D_{y_0}}$ to the last equality and divide by the irrelevant sign, we find in view of Corollary~\ref{cor:bcommuteswithskew} that $w_o(a_1)y_1=(x)\constantoverleftarrow{D_{y_2y_0}}$.
Multiplication of the last equality form the right with $y_2$ yields 
\[
w_o(a_1)y_1y_2=(x)\constantoverleftarrow{D_{y_2y_0}}y_2=(-1)^{\ell(w_o)}(x)\constantoverleftarrow{D_{y_2}}y_2\constantoverleftarrow{D_{y_0}}=(-1)^{\ell(w_o)}(x)\constantoverleftarrow{D_{y_0}}
\]
by Corollary~\ref{cor:ofbskew} and Lemma~\ref{lem:invariance-prop}\eqref{item:inv3}. In particular, we see from this and Theorem~\ref{thm:hypo-II}\eqref{item:II-two} that $w_o(a_1)y_1y_2$ and thus, by Lemma~\ref{cor:rhohypo}, $a_1y_1y_2$ are nonzero hypothetical elements. 
We observe that $(a_1)\constantoverleftarrow{D_\alpha}=0$ for all $\alpha\in T_{w_2}$ since $(a)\constantoverleftarrow{D_\alpha}=0$ for all $\alpha\in T_{w_2}$ as a conclusion of Lemma~\ref{lem:keyconclusion}.
This means that we can apply $\constantoverleftarrow{D_{y_2}}$ to the nonzero hypothetical element $a_1y_1y_2$ which becomes a nonzero quantity by Theorem~\ref{thm:inv-hypo} and which evaluates as $a_1y_1$ in view of the previous observation and Lemma~\ref{lem:basic-rev}\eqref{item:basic2-3}\eqref{item:basic3},\eqref{item:basic4}. 
By definition of $a_1y_1$ as derivative $\constantoverleftarrow{D_{y_2}}$ of something, it is clear that $(a_1y_1)\constantoverleftarrow{D_\alpha}=0$ for all $\alpha\in T_{w_2}$.
On the other hand, by construction, $a_1y_1$ does only involve $R^+\setminus\Delta=T_{w_1}\cup T_{w_2}$, and it follows from Lemma~\ref{lem:basic-rev}\eqref{item:basic3} that $(a_1y_1)\constantoverleftarrow{D_\beta}=0$ for all $\beta\in\Delta$.
In total, we then have $(a_1y_1)\constantoverleftarrow{D_{\alpha}}=0$ for all $\alpha\in T_{w_0}\cup T_{w_2}$. By Corollary~\ref{cor:innilcoxeter}, it finally follows that $a_1y_1$ is a nonzero element in $w_1\EuScript{N}_W$. For $\mathbb{Z}_{\geq 0}$-degree reasons, this means that $a_1y_1=\lambda y_1$ for some nonzero scalar $\lambda$. It follows that $y_1y_2$ is a nonzero hypothetical element of $W$-degree $1$ which is a contradiction to Theorem~\ref{thm:hypo-II}\eqref{item:II-three}. This completes the proof.
\end{proof}

\section{Consequences}
\label{sec:consequences}

Here, we draw some immediate consequences from Theorem~\ref{thm:mainS6}.

\begin{thm}
\label{thm:main}

The algebras $\EuScript{B}_{S_m}$ are infinite dimensional for all $m\geq 6$.
    
\end{thm}

\begin{proof}

For $m\geq 6$, the algebra $\EuScript{B}_{\mathbb{S}_6}$ is a subalgebra of $\EuScript{B}_{\mathbb{S}_m}$. Hence, since $\EuScript{B}_{\mathbb{S}_6}$ is infinite dimensional by Theorem~\ref{thm:mainS6}, so is $\EuScript{B}_{\mathbb{S}_m}$.
\end{proof}

\begin{cor}
\label{cor:main}

The algebras $\EuScript{E}_m$ are infinite dimensional for all $m\geq 6$. 
    
\end{cor}

\begin{proof}
The algebra $\EuScript{B}_{\mathbb{S}_m}$ is a quotient of $\EuScript{E}_{m}$ for all $m$. Since $\EuScript{B}_{\mathbb{S}_m}$ is infinite dimensional for all $m\geq 6$ by Theorem~\ref{thm:main}, so is $\EuScript{E}_m$ for all $m\geq 6$.  
\end{proof}

\begin{thm}
\label{thm:milinski-conj}

Conjecture~\ref{conj:milinski} is satisfied in type $\mathsf{A}$.
    
\end{thm}

\begin{proof}

By Theorem~\ref{thm:main} and Remark~\ref{rem:coxeter}, the assumptions of Conjecture~\ref{conj:milinski} are only possibly satisfied if $R$ is of type $\mathsf{A}_1$ or $\mathsf{A}_3$. In these two cases, they are indeed satisfied together with the conclusion of Conjecture~\ref{conj:milinski} for trivial reasons (cf.~Remark~\ref{rem:melinski} and \cite[Example~6.4]{milinski-schneider}).
\end{proof}

\bibliographystyle{aomplain}
\bibliography{lib}

\end{document}